\documentclass[12pt]{amsart}
\usepackage{amsthm}
\newtheorem{theorem}{Theorem}[section]
\newtheorem{lemma}[theorem]{Lemma}

\theoremstyle{definition}
\newtheorem{definition}[theorem]{Definition}

\theoremstyle{remark}
\newtheorem{remark}[theorem]{Remark}

\counterwithin*{section}{part}

\usepackage{amssymb}
\usepackage{empheq}
\usepackage{stmaryrd}
\usepackage{enumerate}
\usepackage[shortlabels]{enumitem}
\usepackage{calc}
\usepackage{url}
\usepackage{comment}

\newcommand{\norm}[1]{\left\lVert#1\right\rVert}
\usepackage{mathtools}
\DeclarePairedDelimiter{\ceil}{\lceil}{\rceil}

\def\EE{\mathbb{E}}
\def\PP{\mathbb{P}}
\def\NN{\mathbb{N}}

\def\RR{\mathbb{R}}

\newcommand{\dcr}[3]{D_{{#1},{#2}}({#3})}

\newcommand{\AS}{\quad\PP\text{-a.s.}}

\usepackage[left=2.0cm,
                right=2.0cm,
                top=2.5cm,
                bottom=3.5cm,
                headheight=12pt,
                a4paper]{geometry}

\newcommand{\flucinf}[2]{J_{{#1}}\left({#2}\right)}
\newcommand{\fluc}[3]{J_{{#1}, {#2}}\left({#3}\right)}

\newcommand{\ql}[2]{\left\langle\overrightarrow{#1},\overrightarrow{#2}\right\rangle}

\begin{document}

\title[Generalized fluctuation bounds for stochastic algorithms]{Generalized fluctuation bounds for stochastic algorithms in the presence of compactness}

\author[M. Neri, N. Pischke, T. Powell]{Morenikeji Neri${}^{\MakeLowercase a}$, Nicholas Pischke${}^{\MakeLowercase b}$, Thomas Powell${}^{\MakeLowercase b}$}
\date{\today}
\maketitle
\vspace*{-5mm}
\begin{center}
{\scriptsize 
${}^a$ Department of Mathematics, Technische Universit\"at Darmstadt,\\
Schlossgartenstra\ss{}e 7, 64289 Darmstadt, Germany,\\ 
${}^b$ Department of Computer Science, University of Bath,\\
Claverton Down, Bath, BA2 7AY, United Kingdom,\\
E-mails: neri@mathematik.tu-darmstadt.de, nnp39@bath.ac.uk, trjp20@bath.ac.uk}
\end{center}

\maketitle 
\begin{abstract}
We provide a convergence result for sequences of random variables taking values in a metric space that satisfy a stochastic quasi-Fej\'er monotonicity condition, in the context of a (local) compactness assumption. Our result is quantitative in that we derive an explicit and effective construction which, in terms of only a few moduli representing quantitative witnesses to key properties of the sequence of random variables and the underlying metric space involved, provides a metastable rate of pointwise convergence, a type of generalized fluctuation bound. That quantitative result in particular relies on the development of a finitary theory of martingales, culminating in a fully finitary Robbins-Siegmund theorem. We outline how this result particularises to the circumstances of the seminal work of Combettes and Pesquet on stochastic quasi-Fej\'er monotone sequences in separable Hilbert spaces, and we provide an initial application by illustrating how these results can be used to provide a metastable rate of pointwise convergence for a stochastic Krasnoselskii-Mann scheme solving a stochastic common fixed point problem for nonexpansive maps over proper Hadamard spaces. This work is set in the context of recent applications of the logic-based methodology of proof mining to probability theory, and represents its most sophisticated case study to date.
\end{abstract}

\section{Introduction}

\subsection{Background and motivation}

Fej\'er monotonicity is one of the most fundamental concepts in the modern study of numerical algorithms. At its core, it captures the idea that a method $(x_n)$ in some metric space $(X,d)$ decreases towards a target set $F\subseteq X$, or more formally:
\[
\forall n\in\NN\ \forall z\in F\left(d(x_{n+1},z)\leq d(x_n,z)\right).
\] 
Extensions of this basic property have proven to be powerful organisational principles in the convergence analysis of algorithms, in which $F$ generally represents some set of solutions to a problem under consideration. This is primarily because Fej\'er monotonicity occupies a unique position where (1) large classes of complex algorithms conform to one of its many variants and (2) it allows for a streamlined convergence analysis of associated methods, where additional properties which guarantee convergence can be articulated in a similarly abstract and uniform way. 

Historically, the name Fej\'er monotonicity for the above concept was coined by Motzkin and Schoenberg \cite{MotzkinSchoenberg1954}, after a use made of a similar concept by Fej\'er in \cite{Fejer1922} as well as some subsequent work. While these works were based on Euclidean spaces, extensions to Hilbert spaces along with the development of a large amount of the related theory is due to Eremin \cite{Eremin1968a,Eremin1968b} and Ermol'ev \cite{Ermolev1969,Ermolev1971,ErmolevTuniev1973}, together with other exponents of the ``Russian school'' of convex programming. Much of the modern appreciation of Fej\'er monotonicity as a unifying tool can be attributed to the work of Bauschke and Combettes, especially the surveys and expositions of Combettes \cite{Combettes2001,Combettes2009} and the central monograph \cite{BauschkeCombettes2017}. Today, the zoo of variants of Fej\'er monotonicity is vast, incorporating, for example, varying metrics \cite{CombettesVu2013}, adaptations to different distance functions like Bregman distances \cite{BauschkeBorweinBorwein2003} or very general classes of distances on metric spaces \cite{Pischke2025} (see also \cite{NeriPischkePowell2025,PischkePowell2026}), or extensions and variations for different classes of spaces like Hadamard spaces \cite{BacakSearstonSims2012} or general metric spaces \cite{KohlenbachLeusteanNicolae2018,KohlenbachLopezAcedoNicolae2019}. The most recent work in this direction includes so-called Fej\'er$^*$ monotonicity \cite{ArakcheevBauschke2025b,BehlingBelloCruzIusemLiuSantos2024,BehlingBelloCruzIusemAlvesRibeiroSantos2024} as well as the extensive generalization of Fej\'er monotonicity introduced by Kohlenbach and Pinto \cite{KohlenbachPinto2023}. Several of the aforementioned works also consider quantitative aspects in one way or the other.

In the case of stochastic algorithms, Fej\'er monotonicity naturally lifts to a corresponding supermartingale property, which when extended with error terms becomes satisfied by a vast array of stochastic methods. So-called stochastic quasi-Fej\'er monotonicity already originated in the work of Ermol'ev \cite{Ermolev1969,Ermolev1971,ErmolevTuniev1973}, and has more recently been isolated as an abstract concept in the more general setting of Hilbert spaces in the work of Combettes and Pesquet \cite{CombettesPesquet2015} (see also \cite{CombettesPesquet2019} for a quantitative perspective and \cite{Pischke2026} for extensions of the main abstract result from \cite{CombettesPesquet2015} to Hadamard spaces), where it is subsequently used to establish convergence of a range of optimization algorithms. In the generalised setting of \cite{CombettesPesquet2015}, a stochastic algorithm $(x_n)$ adapted to some filtration $(\mathcal{F}_n)$ over a probability space $(\Omega,\mathsf{F},\PP)$ and taking values in some separable Hilbert space $(X,\norm{\cdot})$ is stochastically quasi-Fej\'er monotone w.r.t.\ a set $F\subseteq X$ if
\[
\EE[\phi(\norm{x_{n+1}-z})\mid \mathcal{F}_n]+\theta_n(z)\leq (1+\chi_n(z))\phi(\norm{x_n-z})+\eta_n(z) \;\; \PP\text{-a.s.}
\label{stoch:Fej}\tag{$\ast$}
\]
for all $z\in F$ and $n\in\NN$, where $\phi:[0,\infty)\to [0,\infty)$ is strictly increasing and coercive (i.e.\ $\lim_{t\to+\infty}\phi(t)=+\infty$), and $(\theta_n(z))$ as well as $(\chi_n(z)),(\eta_n(z))$ are sequences of random variables adapted to $(\mathcal{F}_n)$, with the latter two assumed to be summable almost surely.

In the special case where the solution to the problem is unique or at least suitably regular, the convergence analysis of (stochastically) quasi-Fej\'er monotone sequences simplifies considerably, and it is often possible to prove strong convergence under minimal additional assumptions. Moreover, in such a case one can often provide explicit rates of convergence towards this solution (provided one has a sufficient quantitative rendering of the uniqueness or regularity property). Such reasoning was exploited recently in work of the authors \cite{NeriPischkePowell2025,PischkePowell2026} to provide rates of convergence for stochastically quasi-Fej\'er monotone sequences in an abstract setting, in particular working over metric spaces and general classes of distance, building on preceding work for deterministic sequences in \cite{KohlenbachLopezAcedoNicolae2019}. Similar results, albeit restricted to Hilbert spaces and specific forms of problems and types of regularity, can also be found in \cite{CombettesPesquet2019}.

However, where solutions are not necessarily unique (or at least regular), establishing strong convergence is more subtle, and generally relies on a (local) compactness assumption. Crucially, and already for deterministic algorithms, effective rates of convergence are then generally no longer possible.\footnote{This can in fact be made formally precise using tools from mathematical logic and computability theory, using so-called Specker sequences \cite{Specker1949} (see in particular also \cite{Neumann2015}). We refer to \cite{KohlenbachLeusteanNicolae2018,KohlenbachLopezAcedoNicolae2019} for further discussions.} The resulting problem of establishing suitable quantitative results for deterministic Fej\'er monotone sequences in the presence of compactness in such general cases was first considered by Kohlenbach, Leu\c{s}tean and Nicolae \cite{KohlenbachLeusteanNicolae2018}. Roughly speaking, there it is shown that if $(x_n)$ is a quasi-Fej\'er monotone sequence in some compact space $X$ with respect to some $F$ that is closed in a certain explicit sense, and $(x_n)$ satisfies an approximation property with respect to $F$, then $(x_n)$ strongly converges to a point in $F$, and moreover, given suitable quantitative versions of these assumptions, an explicit and highly uniform bound for $\exists N$ in the following \emph{metastable} formulation of the Cauchy property can be given:
\[
\forall k\in\NN\ \forall g:\NN\to\NN\ \exists N\in\NN\  \forall i,j\in [N;N+g(N)]\left(d(x_i,x_j)\leq \frac{1}{k+1}\right),
\]
where we write $[n;m]:=[n,m]\cap\NN$. Metastable convergence (and, more generally, metastable versions of infinitary concepts) first explicitly occurred in the work of Kohlenbach (see in particular \cite{Kohlenbach2005,KohlenbachLambov2004}) and has subsequently been rediscovered and popularised by Tao (see \cite{Tao2008a} and also \cite{Tao2008b}). It is now frequently utilised across various parts of analysis and optimization, as it is typically possible to find explicit rates of metastability (i.e.\ bounds on $\exists N$ above, in terms of $k$ and $g$) even in situations where computable rates of convergence or the like do not exist, and thus metastability provides a much needed substitute for quantitative information on the convergence behaviour in these cases. Informally speaking, rates of metastability represent an alternative quantitative characterisation of convergence in terms of the generalised \emph{fluctuation} behaviour of the sequence $(x_n)$, providing for arbitrary functions $g:\NN\to\NN$ a bound on the existence of an interval $[N;N+g(N)]$ in which no $1/(k+1)$-fluctuations can occur.\footnote{For more details on how rates of metastability can be viewed as generalised fluctuation bounds see \cite{KohlenbachSafarik2014}, and for stochastic processes the more recent \cite{NeriPischkePowell2025b}.}

Compared to the deterministic setting, quantitative results for stochastic quasi-Fej\'er monotonicity for general non-unique (or non-regular) solution sets were never considered, and indeed the underlying techniques from mathematical logic that facilitated the construction of rates of metastability in \cite{KohlenbachLeusteanNicolae2018} (cf.\ Section \ref{sec:introLogic}) were, with a few isolated exceptions, never deployed in a probabilistic setting until a few years ago. Here one is confronted with a series of significant obstacles, including the increased logical complexity of almost-sure statements, and the arsenal of highly technical machinery required for the convergence analysis of stochastic quasi-Fej\'er monotone sequences, chief among them the Robbins-Siegmund supermartingale convergence theorem \cite{RobbinsSiegmund1971}. Deploying some recent results which bring these logical techniques to bear on probability for the first time, we tackle these obstacles and provide a first quantitative study of stochastic algorithms solving general problems in the presence of compactness.

\subsection{Overview of the main technical results}

Our central object of study is a new and highly general finitary variant of stochastic quasi-Fej\'er monotonicity for algorithms taking values in a general metric space $(X,d)$. Following \cite{KohlenbachLeusteanNicolae2018}, we consider solution sets that can be decomposed in an abstract way as $F=\bigcap_{k\in\NN} AF_k$ via a sequence of measurable approximation sets $AF_0\supseteq AF_1\supseteq\ldots$, giving us the flexibility to uniformly capture various problem formulations and associated notions of approximate solutions from the literature on stochastic optimization.

A sequence $(x_n)$ of random variables taking values in $X$ that is adapted to some filtration $(\mathcal{F}_n)$ as above is then called \emph{uniformly stochastically quasi-Fej\'er monotone} w.r.t.\ $F$, $(AF_k)$ and $(\mathcal{F}_n)$ if\footnote{\label{foot:diff1}This notion omits the negative correction term $(\theta_n)$ from \eqref{stoch:Fej}, which is generally superfluous for a quasi-Fej\'er type argument and often used simply for establishing an associated approximation property, which we deal with separately in this paper.}
\begin{gather*}
\forall \lambda>0\ \forall r,n\in\mathbb{N}\ \exists k\in\mathbb{N}\ \forall z\in AF_k\ \forall m\leq n\\
\left(\PP\left(\EE[G(d(x_{m+1},z))\mid \mathcal{F}_m]> (1+\chi_m)G(d(x_{m},z))+\eta_m+\frac{1}{r+1}\right)<\lambda\right),
\end{gather*}
where $G:[0,\infty)\to [0,\infty)$ is a strictly increasing, coercive and continuous function with $G(0)=0$,\footnote{\label{foot:diff2}Compared to the $\phi$ used in \eqref{stoch:Fej}, which is only continuous almost everywhere, we here assume that the function $G$ is continuous everywhere and also satisfies $G(0)=0$, as this later allows for an easier handling of approximation errors which is crucial for our arguments. However, as most functions employed in the context of Fej\'er monotonicity satisfy these extended assumptions, this is a very mild restriction.} and $(\eta_n),(\chi_n)$ are adapted to $(\mathcal{F}_n)$ and almost surely summable.\footnote{\label{foot:diff3}The move to an approximate solution $z\in AF_k$ in the above notion generally renders the dependence of $(\eta_n),(\chi_n)$ on $z$ occurring in \eqref{stoch:Fej} senseless, which is why our notion considers the errors to be uniformly chosen. Practically, this not a great sacrifice as many, if not most, methods can be treated in the context of \eqref{stoch:Fej} already without such a dependence on $z$.} While discussed later in more detail, a rather direct argument in fact shows that this uniform notion is equivalent to a usual version of stochastic quasi-Fej\'er monotonicity over (locally) compact metric spaces (cf.\ the proof of Theorem \ref{thm:final}). In particular, over finite-dimensional Hilbert spaces (and under suitable boundedness assumptions on the solution set $F$ and its approximations $AF_k$), this property is implied by the previously considered notion \eqref{stoch:Fej} (for error terms uniform in $z$).

The main result of the paper is an associated convergence theory for this general notion of stochastic quasi-Fej\'er monotonicity in a metric setting (cf.\ Section \ref{sec:main}), resulting in a novel strong convergence result in which almost sure convergence of $(x_n)$ towards some $F$-valued random variable $x$ is established under a (local) compactness assumption on the space, a closure assumption of $F$, and a rather weak stochastic approximation property on the sequence $(x_n)$ formulated using $(AF_k)$ (cf.\ Theorems \ref{thm:pre_final} and \ref{thm:final}). These theorems, while novel in their own right, are however merely qualitative renderings of a full quantitative convergence theorem (cf.\ Theorem \ref{thm:mainQuant}) which provides explicit generalized fluctuation information on the sequence $(x_n)$ in the form of a bound $\Delta(\lambda,k,g)$ satisfying 
\[
\PP\left(\forall n\leq \Delta(\lambda,k,g)\ \exists i,j\in [n;n+g(n)]\left(d(x_i,x_j)>\frac{1}{k+1}\right)\right)<\lambda,
\]
for all $\lambda>0$, $k\in\mathbb{N}$ and $g:\mathbb{N}\to\mathbb{N}$, a so-called metastable rate of pointwise convergence (also known as a $\lambda$-uniform bound for the $\varepsilon$-metastable pointwise convergence in \cite{AvigadGerhardyTowsner2010}, where the notion was first isolated in a logical context). The bound $\Delta$ can be explicitly constructed from some simple data representing quantitative formulations of the main assumptions, making it highly uniform and in particular independent of the distribution of the underlying probability space. 

We envisage our quantitative theorem being used in one of two ways: In straightforward scenarios the rather intricate general construction of $\Delta$ will simplify considerably, and then, potentially with the help of ad-hoc optimizations, could be used to provide low-complexity certificates on the fluctuation behaviour of the algorithm in question. Alternatively, where the exact definition of $\Delta$ is of less interest, its uniformity properties can still be invoked to characterise the data on which such generalized fluctuation bounds depend. A concrete example capturing all of this is provided in Section \ref{sec:common}, where we consider a stochastic common fixed-point problem over a separable Hadamard spaces, that is a complete separable metric spaces of nonpositive curvature in the sense of Alexandrov, making use of our general metric setting. In that context, our main results are then used to prove the convergence of a stochastic Krasnoselskii-Mann scheme (cf.\ Theorem \ref{thm:application}). We envisage that many more sophisticated examples of this kind are possible, and further applications will be provided in future work.

Establishing our main quantitative result requires us to reformulate, in a fully finitary manner, a great deal of machinery that traditionally underlies the convergence of stochastic algorithms (such as Doob's martingale convergence theorem \cite{Doob1953} and its well-known extension to almost-supermartingales by Robbins and Siegmund \cite{RobbinsSiegmund1971}). In that way, the entire first half of the paper (cf.\ Sections \ref{sec:martingale} and Section \ref{sec:rs}) is dedicated to the development of a fully finitary supermartingale convergence theory, resulting in a quantitative supermartingale convergence theorem (cf.\ Theorem \ref{thm:fin:learnable:MCT}) which is then extended to a quantitative Robbins-Siegmund theorem (cf.\ Theorem \ref{thm:finRS}). For that, we build on initial work in this direction by the first and third author \cite{NeriPowell2025a,NeriPowell2026} which provides basic fluctuation information for (almost) supermartingales, but here we require a much more careful quantitative analysis, where our definition of uniform stochastic quasi-Fej\'er monotonicity requires us to provide fluctuation information for a weaker kind of process satisfying only an approximate finitary variant of the martingale property, resolving both (in)equalities and probabilities via approximate errors. Along the way, this involves developing finitary analogues of the Doob decomposition theorem (cf.\ Lemma \ref{lem:decomp}), the optional stopping theorem (cf.\ Lemma \ref{lem:quant:optional}), and crossing inequalities (cf.\ Lemma \ref{lem:fin:dcrs:ineq}), along with a subtle quantitative treatment of integrability via so-called moduli of absolute continuity (as introduced in \cite{PischkePowell2024}). We consider the development of this machinery to be a contribution in its own right, independently of our main result on Fej\'er monotonicity, representing a quantitative toolkit that can be used as a base for future work in quantitative stochastic optimization.

\subsection{Relationship to logic and proof mining}\label{sec:introLogic}

This paper belongs to the logic-based proof mining programme, a subfield of mathematical logic which aims at repurposing tools and techniques from the foundations of mathematics to provide additional computational information for ineffectively proven theorems in mathematics. This programme has found wide success in particular in the context of (nonlinear) analysis and optimization, and we refer to the seminal monograph \cite{Kohlenbach2008} as well as the surveys \cite{Kohlenbach2019,KohlenbachOliva2003} for further background on both the theory and practice of this field.

Indeed, this paper forms a synthesis of several different strands of research in proof mining, in particular its most recent and advanced developments, and arguably represents one of the most logically intricate case studies to come out of the proof mining tradition, in particular illustrated by the complexity of the general construction of the stochastic metastable bound $\Delta$. Concretely, the present paper builds on a series of papers giving quantitative meaning to Fej\'er monotone sequences in the nonstochastic setting (see \cite{KohlenbachLeusteanNicolae2018} and \cite{KohlenbachPinto2023,Pischke2025}, among others), similarly originating from this logical point of view, and in particular lifts these to stochastic sequences by further developing several recent works on quantitative aspects of martingales \cite{NeriPowell2025a,NeriPowell2026}, which in turn build on older results in ergodic theory \cite{AvigadGerhardyTowsner2010}, all likewise applications of proof mining. As such, the present results complement and extend previous work of the authors on the quantitative convergence behavior of stochastic quasi-Fej\'er monotone sequences under regularity assumptions assumptions \cite{NeriPischkePowell2025,PischkePowell2026}, similarly developed using methods from proof mining, and like the other probabilistic applications mentioned above makes heavy use of logical insights codified by the so-called logical metatheorems for probability theory \cite{NeriOlivaPischke2026,NeriPischke2024}, the logical foundations for applications of proof mining to that area. In this way, the paper represents a tour through many of the key recent developments in the proof mining programme, all of which are brought together to solve the problem at hand: Making quantitative sense of stochastic quasi-Fej\'er monotonicity in the presence of compactness.

All these logical aspects notwithstanding, everything that follows is formulated without explicit reference to logic (aside from a few brief remarks highlighting some of the logical aspects) and can be comprehended by a reader with no knowledge of logic or proof theory, a typical feature of applications of proof mining.

\subsection{Preliminaries and notation}

Throughout this paper, we fix an ambient probability space $(\Omega,\mathsf{F},\PP)$. If not specified otherwise, all probabilistic notions such as measurability, random variables, almost sureness, etc., are understood relative to that space. All properties as well as (in)equalities between random variables are understood to hold only almost surely, if not stated otherwise. We denote the indicator function of a measurable set $A\in\mathsf{F}$ by $\mathbf{1}_A$. We will later be concerned with variables taking values in metric spaces $(X,d)$, in which case we understand the notion of measurability or random variable to be defined using the Borel $\sigma$-algebra $\mathcal{B}(X)$ of that space.

Further, we fix a filtration $\mathcal{F}:=(\mathcal{F}_n)$ on that space. Given that filtration, we write $\ell_+(\mathcal{F})$ for the set of sequences of non-negative random variables $(\xi_n)$ that are adapted to the filtration, i.e.\ where $\xi_n$ is $\mathcal{F}_n$-measurable for all $n\in\mathbb{N}$. Further, we write $\ell^1_+(\mathcal{F})$ for the set of all $(\xi_n)\in \ell_+(\mathcal{F})$ such that $\sum_{n=0}^\infty\xi_n<+\infty$. We also call a sequence of random variables $(x_n)$ a stochastic process, and we say that this process is adapted to $\mathcal{F}$ if $x_n$ is $\mathcal{F}_n$-measurable for all $n\in\mathbb{N}$.

To specify complex events, we follows the slightly nonstandard approach of using quantifiers instead of unions and intersections. Concretely, let $\varphi(n)\in\mathsf{F}$ be a measurable property for any $n\in\mathbb{N}$. Then we write
\[
\exists n\ \varphi(n):=\bigcup_{n\in\mathbb{N}}\varphi(n)\text{ and }\forall n\ \varphi(n):=\bigcap_{n\in\mathbb{N}}\varphi(n).
\]
Naturally, this extends to properties with an arbitrary finite amount of variables and any type of quantifier prefix. Crucially, we can use the continuity of the probability measure to derive the following type of ``prenexation'' results for the above notation, which formulate standard properties about monotone sequences of events using quantifier notation:

\begin{lemma}[Lemma 3.1 in \cite{NeriPowell2025a}]\label{lem:prenex}
Let $\varphi(n)\in\mathsf{F}$ be a measurable property for all $n\in\mathbb{N}$ and $p\in [0,1]$. If $\varphi$ is anti-monotone, i.e.\ $\varphi(n+1)\subseteq\varphi(n)$ for all $n\in\mathbb{N}$, then
\begin{enumerate}
\item $\PP(\forall n\, \varphi(n))\geq p$ if, and only if, $\forall n(\PP(\varphi(n))\geq p)$,
\item $\PP(\forall n\, \varphi(n))\leq p$ if, and only if, $\forall \lambda >0\exists n(\PP(\varphi(n))<p+\lambda)$.
\end{enumerate}
If $\varphi$ is monotone, i.e.\ $\varphi(n)\subseteq\varphi(n+1)$ for all $n\in\mathbb{N}$, then
\begin{enumerate}
\item $\PP(\exists n\, \varphi(n))\leq p$ if, and only if, $\forall n(\PP(\varphi(n))\leq p)$,
\item $\PP(\exists n\, \varphi(n))\geq p$ if, and only if, $\forall \lambda >0\exists n(\PP(\varphi(n))>p-\lambda)$.
\end{enumerate}
\end{lemma}

Lastly, to talk about intervals of natural numbers, we use the notation $[n;m]:=[n,m]\cap\mathbb{N}$ throughout.

\section{Finitary martingales and quantitative martingale convergence}
\label{sec:martingale}

In order to give a quantitative treatment to stochastic quasi-Fej\'er monotone sequences, we first need to develop a quantitative theory of martingales. This task will occupy us for this as well as the following section and culminates in an quantitative variant of the Robbins-Siegmund theorem on almost-supermartingales \cite{RobbinsSiegmund1971}, which will be crucial for the main results on stochastic algorithms that follow. As such, our whole approach relies on, and extends, existing work on quantitative aspects of martingale convergence, as developed recently by the first and last author \cite{NeriPowell2025a,NeriPowell2026}, in two crucial ways:

First, supermartingales are replaced throughout with the weaker, finitary notion of being a $\lambda$-$\varepsilon$-$N$-supermartingale, a quantitative variant of the usual notion of a supermartingale which
\begin{enumerate}
\item concerns only finite fragments of the process with length $N\in\mathbb{N}$,
\item relaxes the descent condition by including an approximate error $\varepsilon>0$,
\item weakens the almost-sure to being false with low probability $\lambda>0$.
\end{enumerate} 
Switching to this variant requires a careful adaptation of several standard results on martingales to the finitary setting, along with a corresponding strengthening of the recent quantitative martingale convergence theorem \cite[Theorem 7.4]{NeriPowell2025a}, which did not feature such considerations as there the supermartingale property was not resolved.\footnote{In logical terms, it treated the associated theorems on supermartingales in their rule form.}

Second, we use this new quantitative martingale convergence theorem to provide an alternative quantitative variant of the Robbins-Siegmund theorem compared to that recently given in \cite{NeriPowell2026} by not only strengthening the result to finitary almost supermartingales, similar to above, but with convergence now formulated in terms of \emph{fluctuations}, a crucial adjustment for the delicate calculations around Fej\'er monotone sequences that follow.

The main notion that we henceforth work with is the following finitary weakening of a supermartingale:

\begin{definition}[Finitary supermartingale]
Let $\lambda,\varepsilon>0$ and let $(X_n)$ be a sequence of random variables adapted to $(\mathcal{F}_n)$ with $\EE[\vert X_n\vert]<+\infty$ for all $n\in\mathbb{N}$. The process $(X_n)$ is called a $\lambda$-$\varepsilon$-supermartingale if 
\[
\PP\left(\EE[X_{n+1}\mid \mathcal{F}_n]> X_n+\varepsilon\right)<\lambda
\]
for all $n\in\mathbb{N}$. It is called a $\lambda$-$\varepsilon$-$N$-supermartingale if this property is only required to hold for all $n<N$.
\end{definition}

Analogously, we call $(X_n)$ a $\lambda$-$\varepsilon$(-$N$)-submartingale if the displayed property is replaced by the requirement that
\[
\PP\left(\EE[X_{n+1}\mid \mathcal{F}_n]< X_n-\varepsilon\right)<\lambda
\]
for any $n\in\mathbb{N}$ (any $n<N$). Clearly, $(X_n)$ is a $\lambda$-$\varepsilon$(-$N$)-supermartingale if, and only if, $(-X_n)$ is a $\lambda$-$\varepsilon$(-$N$)-submartingale.

\begin{remark}[For logicians]
The above property of a $\lambda$-$\varepsilon$-$N$-supermartingale is a full finitization of the usual supermartingale property, following the works on the logical foundations of proof mining and probability \cite{NeriOlivaPischke2026,NeriPischke2024}, an aspect that crucially underlies the conception of this paper. Concretely, these logical results guarantee that there are suitable quantitative variants based on finitary supermartingales for large classes of theorems phrased for ordinary supermartingales. The following quantitative results on supermartingales all arose through such a logical analysis of their respective infinitary counterparts.
\end{remark}

The main result of the present section is a quantitative variant of the martingale convergence theorem phrased for $\lambda$-$\varepsilon$-$N$-supermartingales. For that, as mentioned above, we need to extend a range of standard results for ordinary supermartingales (an account of which can be found in any standard textbook in advanced probability, see e.g.\ \cite{Klenke2020}) to this more general setting of finitary supermartingales. 

The first such result is a quantitative version of Doob's decomposition theorem, which allows one to decompose a stochastic process $(X_n)$ into a sum $X_n=Y_n+Z_n$ such that $(Y_n)$ is a martingale and $(Z_n)$ a predictable process (i.e.\ $Z_n$ is $\mathcal{F}_{n-1}$ measurable for all $n\geq 1$), which is moreover nonincreasing whenever $(X_n)$ is a supermartingale. 

\begin{lemma}\label{lem:decomp}
Let $(X_n)$ be adapted to $(\mathcal{F}_n)$ and be such that $\EE[\vert X_n\vert]<+\infty$ for all $n\in\mathbb{N}$. Then $X_n=Y_n+Z_n$ for any $n\in\mathbb{N}$ where $(Y_n)$ is a martingale w.r.t.\ $(\mathcal{F}_n)$ and $(Z_n)$ is predictable w.r.t.\ $(\mathcal{F}_n)$. Furthermore, if $(X_n)$ is a $\lambda$-$\varepsilon$-$N$-supermartingale, then $(Z_n)$ is $\lambda$-$\varepsilon$-$N$-nonincreasing, i.e.
\[
\PP\left(Z_{n+1}> Z_n+\varepsilon\right)<\lambda
\]
for all $n<N$.
\end{lemma}
\begin{proof}
Proceed as usual in the proof of Doob's decomposition theorem and define $Z_0:=0$ as well as $Z_n:=\sum_{k=1}^n(\EE[X_k\mid \mathcal{F}_{k-1}]-X_{k-1})$ for $n\geq 1$. Then each $Z_n$ is $\mathcal{F}_{n-1}$-measurable and $(Z_n)$ is hence predictable. We further define $Y_n:=X_n-Z_n$ and it follows as usual that $\EE[\vert Y_n\vert]<+\infty$ and that $(Y_n)$ is a martingale w.r.t.\ $(\mathcal{F}_n)$. If $(X_n)$ is now a $\lambda$-$\varepsilon$-$N$-supermartingale, it easily follows that 
\[
\PP\left(Z_{n+1}>Z_n+\varepsilon\right)\leq \PP\left( \EE[X_{n+1}\mid \mathcal{F}_n]>X_n+\varepsilon\right)<\lambda
\]
for any $n<N$, i.e.\ that $(Z_n)$ is $\lambda$-$\varepsilon$-$N$-nonincreasing.
\end{proof}

The next result we consider is Doob's optional stopping theorem, that is that for bounded stopping times $\rho\leq \tau$, we get $\EE[X_\tau\mid\mathcal{F}_\rho]\leq X_\rho$ for a supermartingle $(X_n)$ and the stopped filtration defined via
\[
\mathcal{F}_\rho:=\{A\in\mathcal{F}\mid A\cap \{\rho\leq t\}\in \mathcal{F}_t\text{ for all }t\in\mathbb{N}\}.
\]
Before we can give a quantitative version of this result, we need to develop a few further (quantitative) properties of random variables and martingales.

The first such result is a quantitative version of the fact that the expected value of an ordinary supermartingale decreases. This property transfers to an approximate version for finitary supermartingales. However, to give such a quantitative result, we have to quantitatively resolve the property that each respective random variable is integrable in a more comprehensive way than simply providing an upper bound on the mean.

The enriched data we rely on here is a so-called modulus of absolute continuity (cf.\ \cite[Definition 3.2]{PischkePowell2024}):

\begin{definition}
\label{def:modulus:cont}
Let $X$ be an integrable random variable. We call a function $\mu:(0,\infty)\to (0,\infty)$ such that
\[
\forall \varepsilon>0\ \forall A\in \mathcal{F}\left( \PP(A)<\mu(\varepsilon)\to \EE[\vert X\vert 1_A]\leq\varepsilon\right)
\]
a modulus of absolute continuity for $X$.
\end{definition}

By e.g.\ Lemma 13.1 in \cite{Williams1991}, such a modulus always exists for integrable $X$. However, the quantitative information it provides complements that of a simple upper bound on the mean, with both together providing a pair of data fully characterising the integrability of a random variable quantitatively. Later, in Section \ref{sec:common}, we discuss concrete instantiations for such moduli in the context of applications to stochastic approximation methods.

\begin{remark}[For logicians]\label{rem:modAbsContDer} 
The above notion of a modulus of absolute continuity is motivated by the proof-theoretic methodology of proof mining. In particular, the underlying logical results (see again \cite{NeriPischke2024} and in particular also \cite{NeriOlivaPischke2026}) guarantee that such moduli can be extracted from large classes of proofs of the integrability of a random variable. 
\end{remark}

Using this modulus, we can then immediately obtain the following quantitative descent property for finitary martingales.

\begin{lemma}\label{lem:quant:desc}
Let $\mu_n$ be a modulus of absolute continuity for $X_n$. Further, let $(X_n)$ be a $\mu^M_N(\varepsilon/(4N))$-$\varepsilon/(2N)$-$N$-supermartingale, where $\mu^M_N(\varepsilon):=\min\{\mu_n(\varepsilon)\mid n\leq N\}$. Then 
\[
\EE[X_m]\leq\EE[X_n]+\varepsilon
\]
for all $n\leq m\leq N$.
\end{lemma}
\begin{proof}
As $(X_n)$ is a $\mu^M_N(\varepsilon/(4N))$-$\varepsilon/(2N)$-$N$-supermartingale, we get
\[
\PP\left( \EE[X_{n+1}\mid \mathcal{F}_n]>X_n+\varepsilon/(2N)\right)<\mu^M_N(\varepsilon/(4N))\leq \min\left\{\mu_{n+1}(\varepsilon/(4N)),\mu_n(\varepsilon/(4N))\right\}
\]
for any $n<N$. Denote the set inside the probability above by $A_n$. Then we have
\begin{align*}
\EE[X_{n+1}]&=\EE[X_n]+\EE[(X_{n+1}-X_n)\mathbf{1}_{A_n}]+ \EE[(X_{n+1}-X_n)\mathbf{1}_{A_n^c}]\\
&\leq \EE[X_n]+\EE[|X_{n+1}|\mathbf{1}_{A_n}]+\EE[|X_n|\mathbf{1}_{A_n}]+\EE[(\EE[X_{n+1}\mid\mathcal{F}_n]-X_n)\mathbf{1}_{A_n^c}]\\
&\leq\EE[X_n]+\varepsilon/(4N)+\varepsilon/(4N)+\varepsilon/(2N)=\EE[X_n]+\varepsilon/N
\end{align*}
where the second line follows since $A_n^c\in\mathcal{F}_n$, and the third using $\EE[X_{n+1}\mid\mathcal{F}_n]-X_n\leq \varepsilon/(2N)$ holds on $A_n^c$ and that $\PP(A_n)<\min\left\{\mu_{n+1}(\varepsilon/(4N)),\mu_n(\varepsilon/(4N))\right\}$. This in particular yields
\[
\EE[X_m]\leq \EE[X_n]+(m-n)\varepsilon/N\leq \EE[X_n]+\varepsilon
\]
for all $n\leq m\leq N$.
\end{proof}

Before moving on, we also require a quantitative version of the property that $\EE[X\mid \mathsf{F}_0]\leq 0$ if, and only if, $\EE[X\mathbf{1}_A]\leq 0$ for all $A\in \mathsf{F}_0$, whenever $X$ is an integrable random variable and $\mathsf{F}_0$ is a given sub-$\sigma$-algebra of $\mathsf{F}$.

\begin{lemma}\label{lem:quant:condChara}
Let $X$ be an integrable random variable and $\mathsf{F}_0$ be a given sub-$\sigma$-algebra of $\mathsf{F}$.
\begin{enumerate}
\item For any $\varepsilon,\lambda>0$, if $\EE[X\mathbf{1}_A]<\varepsilon\lambda$ for all $A\in\mathsf{F}_0$, then $\PP(\EE[X\mid\mathsf{F}_0]>\varepsilon)<\lambda$.
\item Let $\mu$ be a modulus of absolute continuity for $X$. For any $\varepsilon>0$, if $\PP(\EE[X\mid\mathsf{F}_0]>\varepsilon/2)<\mu(\varepsilon/2)$, then $\EE[X\mathbf{1}_A]<\varepsilon$ for all $A\in\mathsf{F}_0$.
\end{enumerate}
\end{lemma}
\begin{proof}
\begin{enumerate}
\item Suppose for a contradiction that $\PP(\EE[X\mid\mathsf{F}_0]>\varepsilon)\geq \lambda$, and write $A_0$ for the set inside the probability. Note that $A_0\in\mathsf{F}_0$. But then it holds that
\[
\varepsilon\lambda>\EE[X\mathbf{1}_{A_0}]=\EE[\EE[X\mid\mathsf{F}_0]\mathbf{1}_{A_0}]\geq\varepsilon\PP(A_0)\geq\varepsilon\lambda,
\]
which is a contradiction.
\item Assume $\PP(\EE[X\mid\mathsf{F}_0]>\varepsilon/2)<\mu(\varepsilon/2)$, and write $A_0$ for the set insides the probability. As before, we have $A_0\in\mathsf{F}_0$, and so given an arbitrary $A\in\mathsf{F}_0$, we have $A\cap A_0,A\cap A^c_0\in\mathsf{F}_0$. This yields
\[
\EE[X\mathbf{1}_A]=\EE[X\mathbf{1}_{A\cap A_0}]+\EE[X\mathbf{1}_{A\cap A_0^c}]\leq \EE[|X|\mathbf{1}_{A\cap A_0}]+\EE[\EE[X\mid\mathsf{F}_0]\mathbf{1}_{A\cap A_0^c}]\leq\varepsilon,
\]
where, in the first inequality, we have used that $A\cap A_0^c\in\mathsf{F}_0$ and, in the second inequality, that $\PP(A\cap A_0)<\mu(\varepsilon/2)$.\qedhere
\end{enumerate}
\end{proof}

We now turn to stopping times. The first result we provide there is a quantitative version of the result that a boundedly stopped process is integrable. For simplicity, we restrict ourselves to nonnegative processes.

\begin{lemma}\label{lem:quant:stoppedProc}
Let $\tau$ be a bounded stopping time with $\tau\leq N$. Further, let $(X_n)$ be a nonnegative stochastic process, and let $\mu_n$ be a modulus of absolute continuity for each $X_n$. Then $\mu^M_N(\varepsilon/N):=\min\{\mu_{n}(\varepsilon/N)\mid n\leq N\}$ is a modulus of absolute continuity for $X_\tau$.
\end{lemma} 
\begin{proof}
Note that $X_{\tau}\leq\max_{n\leq N}X_n$ almost surely, so that
\[
\EE[X_\tau\mathbf{1}_A]\leq \EE\left[\max_{n\leq N}(X_n\mathbf{1}_A)\right]\leq \sum_{n\leq N}\EE[X_n\mathbf{1}_A]\leq N(\varepsilon/N)=\varepsilon
\]
for $A$ such that $\PP(A)<\mu^M_N(\varepsilon/N)\leq \mu_{n}(\varepsilon/N)$ for all $n\leq N$.
\end{proof}

We are now in the position to provide a fully quantitative variant of Doob's optional stopping theorem, phrased for finitary martingales using moduli of absolute continuity. Again, we restrict ourselves to nonnegative processes.

\begin{lemma}\label{lem:quant:optional}
Let $\rho,\tau$ be bounded stopping times w.r.t.\ $(\mathcal{F}_n)$ with $\rho\leq\tau\leq N$. Further, let $(X_n)$ be a nonnegative stochastic process, and let $\mu_n$ be a modulus of absolute continuity for each $X_n$. If $(X_n)$ is a $\mu^M_N(\varepsilon\lambda/(16N))$-$(\varepsilon\lambda/(8N))$-$N$-supermartingale w.r.t.\ $(\mathcal{F}_n)$, then
\[
\PP\left( \EE[X_\tau\mid\mathcal{F}_\rho]> X_\rho+\varepsilon\right)<\lambda.
\]
Here, as before: $\mu^M_N(\varepsilon):=\min\{\mu_{n}(\varepsilon)\mid n\leq N\}$.
\end{lemma}
\begin{proof}
Let $A_0$ denote the set of all $\omega$ such that $\rho(\omega)\leq\tau(\omega)\leq N$ as well as
\[
X_{n+1}(\omega)\leq X_n(\omega)+(\varepsilon\lambda/(8N))
\]
for all $n<N$. Note that $\PP(A_0^c)< \mu^M_N(\varepsilon\lambda/(16N))$ by assumption on $(X_n)$. Then, for any $\omega\in A_0$, we have
\[
X_{\tau(\omega)}(\omega)=X_{\rho(\omega)}(\omega)+\sum_{n=0}^{N-1}(X_{n+1}(\omega)-X_n(\omega))\mathbf{1}_{\rho\leq n<\tau}(\omega).
\]
Now, for $A\in\mathcal{F}_\rho$, we have by the above that
\[
\EE[X_\tau\mathbf{1}_{A\cap A_0}]\leq \EE[X_\rho\mathbf{1}_{A\cap A_0}]+\sum_{n=0}^{N-1}\EE[(X_{n+1}-X_n)\mathbf{1}_{\rho\leq n<\tau}\mathbf{1}_{A\cap A_0}].
\]
Note that $\{\rho\leq n<\tau\}\cap A\in \mathcal{F}_n$ as $A\in\mathcal{F}_\rho$. Therefore, since $\mu^M_N(\varepsilon/2)$ is a modulus of absolute continuity for $X_{n+1}-X_n$, and since
\[
\PP(\EE[X_{n+1}-X_n\mid\mathcal{F}_n]>\varepsilon\lambda/8N)<\mu^M_N(\varepsilon\lambda/(16N)),
\]
part (2) of Lemma \ref{lem:quant:condChara} yields that $\EE[(X_{n+1}-X_n)\mathbf{1}_{\rho\leq n<\tau}\mathbf{1}_{A}]<\varepsilon\lambda/(4N)$ for all $n<N$. In particular, we have
\begin{align*}
\EE[(X_{n+1}-X_n)\mathbf{1}_{\rho\leq n<\tau}\mathbf{1}_{A\cap A_0}]&=\EE[(X_{n+1}-X_n)\mathbf{1}_{\rho\leq n<\tau}\mathbf{1}_{A}]-\EE[(X_{n+1}-X_n)\mathbf{1}_{\rho\leq n<\tau}\mathbf{1}_{A\cap A_0^c}]\\
&\leq \EE[(X_{n+1}-X_n)\mathbf{1}_{\rho\leq n<\tau}\mathbf{1}_{A}]+\EE[X_n\mathbf{1}_{\rho\leq n<\tau}\mathbf{1}_{A\cap A_0^c}]<\varepsilon\lambda/(2N).
\end{align*}
We have thus shown that $\EE[(X_\tau-X_\rho)\mathbf{1}_{A\cap A_0}]\leq \varepsilon\lambda/2$. By Lemma \ref{lem:quant:stoppedProc}, $\mu^M_N(\varepsilon/N)$ is a modulus of absolute continuity for $X_\tau$. Thereby, we have
\begin{align*}
\EE[(X_\tau-X_\rho)\mathbf{1}_A]&=\EE[(X_\tau-X_\rho)\mathbf{1}_{A\cap A_0}]+\EE[(X_\tau-X_\rho)\mathbf{1}_{A\cap A_0^c}]\leq \varepsilon\lambda/2+\varepsilon\lambda/16<\varepsilon\lambda
\end{align*}
as $\PP(A_0^c)<\mu^M_N(\varepsilon\lambda/(16N))$. Part (1) of Lemma \ref{lem:quant:condChara} now yields that $\PP(\EE[X_\tau\mid\mathcal{F}_\rho]> X_\rho+\varepsilon)<\lambda$ as claimed.
\end{proof}

\begin{remark}\label{rem:optStopStrong}
If $(X_n)$ is already a finitary martingale in the stronger sense that the associated probability is not resolved, then the above Lemma \ref{lem:quant:optional} can be simplified in a way that the resulting quantitative information no longer depends on associated moduli of absolute continuity. Concretely, suppose that $(X_n)$ satisfies $\EE[X_{n+1}\mid\mathcal{F}_n]\leq X_n+\varepsilon/N$
for all $n<N$, that is $(X_n)$ is a $\lambda$-$(\varepsilon/N)$-$N$-supermartingale w.r.t.\ $(\mathcal{F}_n)$ for any $\lambda>0$. Then we can show $\EE[X_\tau\mid\mathcal{F}_\rho]\leq X_\rho+\varepsilon$. To see this, write $X_n=Y_n+Z_n$ with $(Y_n)$ and $(Z_n)$ as in Doob's decomposition theorem. Then, as $(X_n)$ is a $\lambda$-$(\varepsilon/N)$-$N$-supermartingale for any $\lambda$, we get that $(Z_n)$ is $\lambda$-$(\varepsilon/N)$-$N$-nonincreasing for any $\lambda>0$, by Lemma \ref{lem:decomp}. By Doob's optional stopping theorem for martingales, we get $\EE[Y_\tau\mid\mathcal{F}_\rho]=Y_\rho$ as $(Y_n)$ is a martingale. Thus, we get $\EE[X_\tau\mid\mathcal{F}_\rho]=Y_\rho+\EE[Z_\tau\mid\mathcal{F}_\rho]$. As $Z_{n+1}\leq Z_n+\varepsilon/N$ for any $n<N$, we have $Z_{\rho}+\varepsilon\geq Z_{\tau}$ and so in particular $\EE[Z_\rho\mid\mathcal{F}_\rho]+\varepsilon\geq\EE[Z_\tau\mid\mathcal{F}_\rho]$. As $Z_\rho=\EE[Z_\rho\mid\mathcal{F}_\rho]$, one derives $\EE[X_\tau\mid\mathcal{F}_\rho]\leq X_\rho+\varepsilon$ which was the claim. 

Note that this type of strengthening, under stronger approximate martingale properties, can also be observed for all preceding results, and for various of the following results. However, contrary to the preceding results, the optional stopping theorem is the first result where this special case presented above is not a corollary of our approach, which is why we explicitly remark it here.
\end{remark}

For two processes $C:=(C_n)$ and $X:=(X_n)$, we define the discrete stochastic integral $C\bullet X$ as usual via
\begin{equation*}
(C\bullet X)_n := \sum_{i=1}^n C_i(X_i - X_{i-1}).
\end{equation*}
The following is a variant for finitary supermartingales of the fundamental result about the discrete stochastic integral (c.f.\ \cite[Theorem 9.39]{Klenke2020}) that if $C$ is a nonnegative and bounded predictable process and $X$ is a supermartingale, then $C\bullet X$ is also a supermartingale. 

\begin{lemma}\label{lem:quant:transform}
Let $C:= (C_n)$ be a bounded, nonnegative and predictable process and let $K>0$ be such that $C_n\leq K$ for any $n\in\mathbb{N}$. Further, let $X:=(X_n)$ be a $\lambda$-($\varepsilon/K$)-$N$-supermartingale. Then $C\bullet X$ is a $\lambda$-$\varepsilon$-$N$-supermartingale.
\end{lemma}
\begin{proof}
Write $Y_n:=(C\bullet X)_n$ and let $\omega$ be such that $C_n(\omega)\leq K$ as well as $\EE[X_{n+1}\mid\mathcal{F}_n](\omega)\leq X_n(\omega)+\varepsilon/K$. Then we have
\[
\EE[Y_{n+1}-Y_n\mid\mathcal{F}_n](\omega)=C_{n+1}(\omega)(\EE[X_{n+1}\mid \mathcal{F}_n](\omega)-X_n(\omega))\leq C_{n+1}(\omega) \varepsilon/K\leq \varepsilon,
\]
for almost all such $\omega$, using that $\EE[Y_{n+1}-Y_n\mid\mathcal{F}_n]=C_{n+1}(\EE[X_{n+1}\mid \mathcal{F}_n]-X_n)$. Then we have
\[
\PP(\EE[Y_{n+1}-Y_n\mid\mathcal{F}_n]>\varepsilon)\leq \PP(C_{n+1}>K)+\PP(\EE[X_{n+1}\mid \mathcal{F}_n]-X_n>\varepsilon/K)<\lambda
\]
which completes the proof.
\end{proof}

In another result, we have that we can convert moduli of absolute continuity from elements of a nonnegative stochastic process to a discrete stochastic integral.

\begin{lemma}\label{lem:quant:transformMod}
Let $C:= (C_n)$ be a bounded, nonnegative and predictable process and let $K>0$ be such that $C_n\leq K$ for any $n\in\mathbb{N}$. Further, let $X:=(X_n)$ be a nonnegative stochastic process with moduli of absolute continuity $(\mu_n)$. Then $\mu^M_N(\varepsilon/(KN)):=\min\{\mu_{n}(\varepsilon/(KN))\mid n\leq N\}$ is a modulus of absolute continuity for $(C\bullet X)_n$ for all $n\leq N$.
\end{lemma}
\begin{proof}
Given $n\leq N$, we have
\[
\EE[(C\bullet X)_n\mathbf{1}_A]=\sum_{i=1}^n\EE[C_i(X_i-X_{i-1})\mathbf{1}_A]\leq K\sum_{i=1}^n\EE[X_i\mathbf{1}_A]\leq K\sum_{i=1}^n\varepsilon/KN \leq \varepsilon
\]
whenever $A$ is such that $\PP(A)<\mu^M_N(\varepsilon/(KN)):=\min\{\mu_{n}(\varepsilon/(KN))\mid n\leq N\}$.
\end{proof}

The final result we need before stating our quantitative martingale convergence theorem is a downcrossing inequality for finitary supermartingales. Crossing inequalities, typically involving \emph{upcrossings}, are the standard way of establishing martingale convergence. For our quantitative version, downcrossings offer a slightly more direct route in the case of supermartingales, as previously illustrated in the case of almost-supermartingales in \cite[Remark 7.11]{NeriPowell2025a}, which is in turn a variant of standard downcrossing inequalities as presented in e.g.\ \cite{Doob1961}. We first recall the notion of a downcrossing:

\begin{definition}
Given $N\in\NN$, $\alpha<\beta$ and a sequence of random variables $(X_n)$, we denote by $\dcr{N}{[\alpha,\beta]}{X_n}$ the number of times the process $(X_n)$ downcrosses the interval $[\alpha,\beta]$ before time $N$. More precisely, $\dcr{N}{[\alpha,\beta]}{X_n}(\omega)=k$ for the maximal $k\in\NN$ such that there exists a sequence $i_1<j_1\leq i_2<j_2\leq \ldots\leq i_k<j_k<N$ such that $\beta\leq X_{i_l}(\omega)$ and $X_{j_l}(\omega)\leq \alpha$ for all $l=1,\ldots,k$.
\end{definition}

Our quantitative downcrossing inequality, phrased for finitary supermartingales, then takes the following form:

\begin{lemma}\label{lem:fin:dcrs:ineq}
Let $\varepsilon >0$, $N \in \NN$ and $\alpha < \beta$ be given with $0<\delta\leq \beta-\alpha$. Let $(X_n)$ be a nonnegative stochastic process with moduli of absolute continuity $(\mu_n)$. If $(X_n)$ is a $\mu^M_N(\varepsilon\delta/(4N^2))$-$(\varepsilon\delta/(2N))$-$N$-supermartingale, then
\[
\EE[\dcr{N}{[\alpha,\beta]}{X_n}]\leq\frac{1}{\beta-\alpha}(\EE[X_0-X_N]+\EE[(X_N-\beta)^+])+\varepsilon.
\]
Here, as before: $\mu^M_N(\varepsilon):=\min\{\mu_n(\varepsilon)\mid n\leq N\}$. In case that $\EE[X_0]\leq K$, the inequality simplifies to
\[
\EE[\dcr{N}{[\alpha,\beta]}{X_n}]\leq\frac{1}{\beta-\alpha}(K+\EE[(X_N-\beta)^+])+\varepsilon.
\]
\end{lemma}
\begin{proof}
Given $\alpha<\beta$, define the $\{0,1\}$-valued predictable process $C:=(C_n)$ by $C_0:=1$ iff $X_0\ge\beta$ and
\[
C_{n+1}:=1\text{ iff }\begin{cases}C_n=1\text{ and }X_n> \alpha,\text{ or}\\C_n=0\text{ and }X_n\ge\beta.\end{cases}
\]
Let $C':=(C'_n)$ be the dual predictable process defined by $C_n':= 1-C_n$. Following the usual estimates of used to establish the downcrossing inequality, we obtain that
\[
(\beta-\alpha)\dcr{N}{[\alpha,\beta]}{X_n}\leq -\sum_{i=1}^{N}C_i(X_i-X_{i-1}) + (X_N-\beta)^{+}\tag{$\dagger$}
\]
holds pointwise everywhere. Since $C'$ is a nonnegative predictable process, bounded by $1$ and $(X_n)$ is a $\mu^M_N(\varepsilon\delta/(4N^2))$-$(\varepsilon\delta/(2N))$-$N$-supermartingale, we must have that $C'\bullet X$ is a $\mu^M_N(\varepsilon\delta/(4N^2))$-$(\varepsilon\delta/(2N))$-$N$-supermartingale by Lemma \ref{lem:quant:transform}, where further $(C'\bullet X)_0=0$. Lemma \ref{lem:quant:transformMod} yields that $\mu^M_N(\varepsilon/N)$ is a modulus of absolute continuity for $(C'\bullet X)_n$ for all $n\leq N$. Therefore, Lemma \ref{lem:quant:desc} yields
\[
\EE\left[\sum_{i=1}^{N} C_i'(X_i-X_{i-1})\right] = \EE[(C'\bullet X)_N]\le \EE[(C'\bullet X)_0] +\varepsilon\delta\leq\varepsilon(\beta-\alpha).
\]
This in turn yields
\begin{align*}
\EE\left[-\sum_{i=1}^{N} C_i(X_i-X_{i-1})\right]-\varepsilon(\beta-\alpha)&\le \EE\left[-\sum_{i=1}^{N} C_i(X_i-X_{i-1})\right]+\EE\left[-\sum_{i=1}^{N} C_i'(X_i-X_{i-1})\right]\\
&= \EE\left[\sum_{i=1}^{N} (C_i+C'_i)(X_{i-1}-X_i)\right]=\EE\left[X_0-X_N\right]
\end{align*}
and so it follows from $(\dagger)$ that
\[
(\beta-\alpha)\EE[\dcr{N}{[\alpha,\beta]}{X_n}]\leq \EE[X_0-X_N]+\EE[(X_N-\beta)^+] +\varepsilon(\beta-\alpha)
\]
and the result follows immediately.
\end{proof}

We now state and prove the main result of the section: a quantitative version of Doob's $L_1$-martingale convergence theorem for the special case of nonnegative supermartingales. While it is an immediate consequence of the martingale convergence theorem that any nonnegative supermartingale converges almost surely to some $X$ with $\EE[|X|]<+\infty$, producing a computable rate of convergence is generally not possible (as discussed in the introduction). As such, and following \cite{NeriPowell2025a}, we instead provide a rate of metastability for the corresponding Cauchy property, in the simplified form of a learnable rate of uniform convergence (cf.\ \cite[Theorem 7.4]{NeriPowell2025a}). The main difference here is that our quantitative result is now given in terms of finitary supermartingales. 

\begin{theorem}\label{thm:fin:learnable:MCT}
For any $\varepsilon,\lambda \in (0,1]$ and $K>0$, define
\[
N_{K}(\lambda,\varepsilon):=\left\lceil \frac{512(K+1)^2}{\lambda^2\varepsilon^2}\right\rceil.
\]
Let $(X_n)$ be a nonnegative stochastic process with moduli of absolute continuity $(\mu_n)$. Then, if $(X_n)$ is a $\mu^M_N(\varepsilon^2/(32N^2_{K}(\lambda,\varepsilon)))$-$(\varepsilon^2/(16N_{K}(\lambda,\varepsilon)))$-$N_{K}(\lambda,\varepsilon)$-supermartingale with $\EE[X_0]<K$, there exists an $i\leq N_{K}(\lambda,\varepsilon)$ such that
\[
\PP\left(\exists k,l\in [a_i;b_i] \left(|X_k-X_l|\geq\varepsilon\right)\right)<\lambda,
\]
uniformly over any $a_0<b_0\leq a_1<b_1\leq \ldots < N_{K}(\lambda,\varepsilon)$. Here, as before: $\mu^M_N(\varepsilon):=\min\{\mu_n(\varepsilon)\mid n\leq N\}$.
\end{theorem}
\begin{proof}
Fix $\varepsilon,\lambda$ and $K$ and let $(X_n)$ be a nonnegative $\mu^M_N(\varepsilon^2/(32N^2_{K}(\lambda,\varepsilon)))$-$(\varepsilon^2/(16N_{K}(\lambda,\varepsilon)))$-$N_{K}(\lambda,\varepsilon)$-supermartingale with $\EE[X_0]<K$. First observe that, by (the proof of) Lemma \ref{lem:quant:desc}, we have $\EE[X_n] \le \EE[X_0] +  \varepsilon/8<  K + 1$ for each $n \le N_{K}(\lambda,\varepsilon)$, as $\varepsilon\leq 1$. Now define the events
\[
A_i:=\{\omega\mid \exists k,l\in [a_i;b_i](|X_k(\omega)-X_l(\omega)|\geq \varepsilon)\} \ \ \ \mbox{and} \ \ \ B_i:=\{\omega\mid X_{a_i}(\omega)\leq 2(K+1)/\lambda\}
\]
for $i=0,\ldots,N_{K}(\lambda,\varepsilon)$. Divide the interval $[0,2(K+1)/\lambda]$ into $p:=\lceil 8(K+1)/(\lambda\varepsilon)\rceil$ many equal subintervals which we label $[\alpha_j,\beta_j]$ for $j=1,\ldots,p$. Add another subinterval of the same width on either side of $[0,2(K+1)/\lambda]$, which we label $[\alpha_0,\beta_0]$ and $[\alpha_{p+1},\beta_{p+1}]$ respectively. We note that each subinterval has width $\leq\varepsilon/4$.

Now, if $\omega\in A_i\cap B_i$ there exist $k(\omega),l(\omega)\in [a_i;b_i]$ with $|X_{k(\omega)}(\omega)-X_{l(\omega)}(\omega)|\geq \varepsilon$. Then by the triangle inequality, either $|X_{a_i}(\omega)-X_{k(\omega)}(\omega)|\geq \varepsilon/2$ or $|X_{a_i}(\omega)-X_{l(\omega)}(\omega)|\geq \varepsilon/2$, and since also $X_{a_i}(\omega)\leq 2(K+1)/\lambda$, this means that one of the intervals $[\alpha_j,\beta_j]$ for $j=0,\ldots,p+1$ is upcrossed or downcrossed somewhere in $[a_i;b_i]$. Therefore, defining the event
\[
C_{i,j}:=\mbox{``$(X_n)$ crosses $[\alpha_j,\beta_j]$ somewhere in $[a_i;b_i]$''},
\]
we have shown that $A_i\cap B_i\subseteq\bigcup_{j=0}^{p+1} C_{i,j}$ and this yields
\[
\sum_{i=0}^{N_{K}(\lambda,\varepsilon)} \PP(A_i\cap B_i)\leq \sum_{i=0}^{N_{K}(\lambda,\varepsilon)}\sum_{j=0}^{p+1}\PP(C_{i,j}).
\]
For any $j\in \{0,\dots,p+1\}$, we further have
\[
\sum_{i=0}^{N_{K}(\lambda,\varepsilon)}\PP(C_{i,j})= \sum_{i=0}^{N_{K}(\lambda,\varepsilon)}\EE\left[\mathbf{1}_{C_{i,j}}\right]=\EE\left[\sum_{i=0}^{N_{K}(\lambda,\varepsilon)} \mathbf{1}_{C_{i,j}}\right]\leq 2\EE[\dcr{N_{K}(\lambda,\varepsilon)}{[\alpha_j,\beta_j]}{X_n}]+1,
\]
where for the final inequality we note that the quantity $\sum_{i=0}^{N_{K}(\lambda,\varepsilon)} \mathbf{1}_{C_{i,j}}$ is bounded by the total number of times that $(X_n)$ crosses $[\alpha_j,\beta_j]$ within its initial segment of length $N_K(\lambda,\varepsilon)$, and since between every two downcrossings there exists exactly one upcrossing, the total number of crossings (being the sum of the number of downcrossings and upcrossings) is bounded by $2\dcr{N_{K}(\lambda,\varepsilon)}{[\alpha_j,\beta_j]}{X_n}+1$. Applying Lemma \ref{lem:fin:dcrs:ineq}, we see that
\[
\EE[\dcr{N_{K}(\lambda,\varepsilon)}{[\alpha_j,\beta_j]}{X_n}]\leq \frac{1}{\beta_j-\alpha_j}(K+\EE[(X_{N_{K}(\lambda,\varepsilon)}-\beta_j)^+])+1 
\]
since $(X_n)$ is an $\mu^M_N(\varepsilon^2/(32N^2_{K}(\lambda,\varepsilon)))$-$(\varepsilon^2/(16N_{K}(\lambda,\varepsilon)))$-$N_{K}(\lambda,\varepsilon)$-supermartingale by assumption and we can show that $\beta_j-\alpha_j\geq \varepsilon /8$ for any $j$. We have
\[
\frac{1}{\beta_j-\alpha_j}(K+\EE[(X_{N_{K}(\lambda,\varepsilon)}-\beta_j)^+])+1  \leq \frac{8(K+\EE[X_{N_{K}(\lambda,\varepsilon)}])}{\varepsilon}+1\leq \frac{8(2K+1)+1}{\varepsilon}
\]
using that $\beta_j\geq 0$ for all $j$ (and $\varepsilon\leq 1$). Picking a $j_0\in \{0,\dots,p+1\}$ where $\PP(C_{i,j_0})$ is maximal among all $\PP(C_{i,j})$, we now get
\[
\sum_{i=0}^{N_{K}(\lambda,\varepsilon)}\PP(A_i\cap B_i)\leq (p+2)\sum_{i=0}^{N_{K}(\lambda,\varepsilon)}\PP(C_{i,j_0})
\]
for this $j_0$ and so, combined with the above and noting $\varepsilon, \lambda \in (0,1]$, we get
\begin{align*}
\sum_{i=0}^{N_{K}(\lambda,\varepsilon)}\PP(A_i\cap B_i)&\leq (p+2)(2\EE[\dcr{N_{K}(\lambda,\varepsilon)}{[\alpha_{j_0},\beta_{j_0}]}{X_n}]+1)\\
&< \left(\frac{8(K+1)}{\lambda\varepsilon}+3\right)\left(\frac{16(2K+1)+2}{\varepsilon}+1\right)\leq \frac{256(K+1)^2}{\lambda\varepsilon^2}.
\end{align*}
Therefore there is some $i\in\{0,\ldots,N_{K}(\lambda,\varepsilon)\}$ such that
\[
\PP(A_i\cap B_i)< \frac{1}{N_{K}(\lambda,\varepsilon)+1} \frac{256(K+1)^2}{\lambda\varepsilon^2}\leq \frac{\lambda^2\varepsilon^2}{512(K+1)^2}\frac{256(K+1)^2}{\lambda\varepsilon^2}=\frac{\lambda}{2}.
\]
But then, using Boole's inequality, it follows that $\PP(A_i)+\PP(B_i)-1\leq \PP(A_i\cap B_i)<\frac{\lambda}{2}$ and by Markov's inequality and the fact that $a_i< N_{K}(\lambda,\varepsilon)$, we have
\[
\PP(B_i)=1-\PP(B_i^c)=1-\PP\left(X_{a_i}>\frac{2(K+1)}{\lambda}\right)\geq 1-\frac{\EE[X_{a_i}]}{2(K+1)/\lambda}>1-\frac{\lambda}{2}
\]
so that $\PP(A_i)<\lambda$. 
\end{proof}

\section{A quantitative Robbins-Siegmund theorem}\label{sec:rs}

We now develop our theory of finitary martingales further to give a quantitative result on almost-supermartingale convergence, more specifically a quantitative version of the classic Robbins-Siegmund theorem \cite{RobbinsSiegmund1971}, which forms a central part of our main analysis just as the (qualitative) Robbins-Siegmund theorem forms a crucial part of the study of stochastic quasi-Fej\'er monotone sequences (see in particular \cite{CombettesPesquet2015}). Here we are able to provide rates of metastability in the simplified form of a bound on the number of fluctuations, which in turn helps simplify the complexity of our main result considerably. While we broadly follow the approach of the recent quantitative analysis of the Robbins-Siegmund theorem in \cite{NeriPowell2026}, the main fluctuation bound presented in Theorem \ref{thm:finRS} crucially dispenses of the so-called uniformly learnable rates derived therein and replaces them by fluctuation bounds. In that way, the bounds presented in  Theorem \ref{thm:finRS} should actually be viewed as an extension to $L_1$-bounded $\lambda$-$\varepsilon$-$N$-almost-supermartingales of the much older work of Chashka \cite{Chashka1994} and Kachurovskii \cite{Kachurovskii1996} on fluctuations in martingales.

To begin, for the benefit of the reader, we recall the notion of fluctuations here:

\begin{definition}
Given $\varepsilon>0$, $N\in\NN$ and a sequence of random variables $(X_n)$, we denote by $J_{N,\varepsilon}(X_n)$ the number of $\varepsilon$-fluctuations experienced by $(X_n)$ before time $N$. More precisely, $J_{N,\varepsilon}(X_n)(\omega)=k$ for the maximal $k\in\NN$ such that there exists a sequence $i_1<j_1\leq i_2<j_2\leq \ldots\leq i_k<j_k<N$ such that $|X_{i_l}(\omega)-X_{i_j}(\omega)|\geq \varepsilon$ for all $l=1,\ldots,k$.
\end{definition}

We first collect the necessary minor results we need for our main task. The following lemma outlines some standard inequalities relating fluctuations and various (arithmetical) operations on stochastic processes:

\begin{lemma}\label{lem:fluc:ineq}
Let $(X_n)$, $(Y_n)$ be sequences of random variables. Then for any $a,b,c>0$, any $\varepsilon>0$ and any $n\in\NN$:
\begin{enumerate}[(i)]
\item $\PP\left(\max_{i\leq n}|X_i+Y_i|\geq a\right)\leq \PP\left(\max_{i\leq n}|X_i|\geq a/2\right)+\PP\left(\max_{i\leq n}|Y_i|\geq a/2\right)$,
\item $\PP\left(\max_{i\leq n}|X_iY_i|\geq a\right)\leq \PP\left(\max_{i\leq n}|X_i|\geq \sqrt{a}\right)+\PP\left(\max_{i\leq n}|Y_i|\geq \sqrt{a}\right)$,
\item $\PP\left(J_{n,\varepsilon}(X_i+Y_i)\geq a\right)\leq \PP\left(J_{n,\varepsilon/2}(X_i)\geq a/2\right)+\PP\left(J_{n,\varepsilon/2}(Y_i)\geq a/2\right)$,
\item $\PP\left(J_{n,\varepsilon}(X_iY_i)\geq a\right)\leq \PP\left(J_{n,\varepsilon/c}(X_i)\geq a/2\right)+\PP\left(J_{n,\varepsilon/b}(Y_i)\geq a/2\right)+\PP\left(\max_{i\leq n}|X_i|\geq b\right)+\PP\left(\max_{i\leq n}|Y_i|\geq c\right)$.
\end{enumerate}
\end{lemma}

\begin{proof}
All bounds are rather immediate. The most involved is (iv), for which we provide a proof as an example: Write $x_i:=X_i(\omega)$ and $y_i:=Y_i(\omega)$ and assume that $\max_{i\leq n}|x_i|<b$ and $\max_{i\leq n}|y_i|<c$. If we have a fluctuation $|x_iy_i-x_jy_j|\geq \varepsilon$, then $|y_i||x_i-x_j|+|x_j||y_i-y_j|\geq \varepsilon$ and thus either $|x_i-x_j|\geq \varepsilon/|y_i|>\varepsilon/c$ or $|y_i-y_j|\geq \varepsilon/|x_j|>\varepsilon/b$. In other words, any single $\varepsilon$-fluctuation of the sequence $(x_ky_k)$ corresponds to an $\varepsilon/c$ fluctuation of $(x_k)$ or an $\varepsilon/b$-fluctuation of $(y_k)$, and thus $J_{n,\varepsilon}(x_ky_k)\leq J_{n,\varepsilon/c}(x_k)+J_{n,\varepsilon/b}(y_k)$. We have therefore shown that
\[
\left\{J_{n,\varepsilon}(X_iY_i)\geq a\cap \sup_{i\leq n}|X_i|<b\cap \sup_{i\leq n}|Y_i|<c\right\}\subseteq \left\{J_{n,\varepsilon/c}(X_i)\geq a/2\cup J_{n,\varepsilon/b}(Y_i)\geq a/2\right\}
\]
and the inequality follows.
\end{proof}

Beyond fluctuations, we will also require some further quantitative results on finitary martingales in the following, starting with a quantitative version of the standard fact that stopped supermartingales are also supermartingales (cf.\ \cite[Theorem 10.15]{Klenke2020}).

\begin{lemma}\label{lem:stoppingtime}
Let $\tau$ be a stopping time w.r.t.\ $(\mathcal{F}_n)$. If $(X_n)$ is a $\lambda$-$\varepsilon$-$N$-supermartingale w.r.t.\ $(\mathcal{F}_n)$, so is the stopped process $(X_{n\wedge\tau})$, where $n\wedge\tau:=\min\{n,\tau\}$.
\end{lemma}

\begin{proof}
Fixing $n<N$, by basic properties of conditional expectations and stopping times, we have
\[
\EE[X_{(n+1)\wedge \tau} \mid \mathcal{F}_n]=\sum_{i=0}^n X_i\mathbf{1}_{\{\tau=i\}}+\mathbf{1}_{\{\tau\geq n+1\}}\EE\left[X_{n+1} \mid \mathcal{F}_n\right],
\]
on some set $\Omega'$ of measure one. Now, let $\omega\in \Omega'$ be such that $\EE[X_{n+1}\mid\mathcal{F}_n](\omega)\leq X_n(\omega)+\varepsilon$. Then the above yields
\[
\EE[X_{(n+1)\wedge \tau} \mid \mathcal{F}_n](\omega)\leq \sum_{i=0}^{n-1} X_i(\omega)\mathbf{1}_{\{\tau=i\}}(\omega)+\mathbf{1}_{\{\tau\geq n\}}(\omega)X_n(\omega) + \mathbf{1}_{\{\tau\geq n+1\}}(\omega)\varepsilon\leq X_{n\wedge \tau}(\omega)+\varepsilon.
\]
In particular, we have shown that 
\[
\PP(\EE[X_{(n+1)\wedge \tau} \mid \mathcal{F}_n]> X_{n\wedge \tau}+\varepsilon)\leq \PP(\EE[X_{n+1}\mid\mathcal{F}_n]> X_n+\varepsilon)
\]
which yields the claim.
\end{proof}

The next results presents our finitary variant of a classic inequality of Ville \cite{Ville1939} (see also \cite{Metivier1982}).

\begin{lemma}\label{lem:fin:Ville}
Let $(X_n)$ be a nonnegative stochastic process with moduli of absolute continuity $(\mu_n)$. If $(X_n)$ is a $\mu^M_N(\varepsilon/(8N^2))$-$\varepsilon/(2N)$-$N$-supermartingale, then for any $\alpha>0$ and any $k\leq N$, we have
\[
\PP\left( \exists t\in [0;k]\left( X_t\geq\alpha\right)\right)\leq \frac{\EE[X_0]}{\alpha}+\frac{\varepsilon}{\alpha}.
\]
Here, as before: $\mu^M_N(\varepsilon):=\min\{\mu_n(\varepsilon)\mid n\leq N\}$.
\end{lemma}
\begin{proof}
We define the stopping time $\tau=\inf\{t\leq N\mid X_t\geq\alpha\}$. By Lemma \ref{lem:quant:stoppedProc}, we have that $\mu^M_N(\varepsilon/N)$ is a modulus of absolute continuity for $X_\tau$. In particular, as $X_{k\wedge \tau}\leq X_k+X_\tau$, we therefore have that $\mu^M_N(\varepsilon/(2N))$ is a modulus of absolute continuity for $X_{k\wedge \tau}$. By Lemma \ref{lem:stoppingtime}, $X_{k\wedge\tau}$ is a $\mu^M_N(\varepsilon/(8N^2))$-$\varepsilon/(2N)$-$N$-supermartingale. Lemma \ref{lem:quant:desc} thereby yields that $\EE[X_{k\wedge\tau}]\leq \EE[X_0]+\varepsilon$ as $k\wedge\tau\leq k$. Therefore, since if $\tau(\omega)\leq k$, then $X_{k\wedge\tau(\omega)}(\omega)=X_{\tau(\omega)}(\omega)\geq\alpha$, it follows that
\[
\PP\left( \exists t\in [0;k]\left( X_t\geq\alpha\right)\right)=\PP(\tau\leq k)\leq \PP(X_{k\wedge\tau} \geq\alpha)\leq\frac{\EE[X_{k\wedge\tau}]}{\alpha}\leq\frac{\EE[X_0]}{\alpha}+\frac{\varepsilon}{\alpha},
\]
where we used Markov's inequality for $X_{k\wedge\tau}$.
\end{proof}

The first main result of this section is to provide a so-called modulus of finite fluctuations (following the terminology of \cite[Definition 4.5]{NeriPowell2025a}) for nonnegative supermartingales, now adapted to the finitary setting. It can be shown (cf.\ Section 4.2 of \cite{NeriPowell2025a}) that if $(X_n)$ converges, then for any $\varepsilon>0$ we have
\[
\lim_{k\to \infty}\PP\left(\lim_{N\to \infty}J_{N,\varepsilon}(X_n)\geq k\right)=0.
\]
A modulus of finite fluctuations in the ordinary sense is defined to be (essentially) a rate of convergence for the above property parametrised by $\varepsilon>0$, namely a function $\phi:(0,\infty)^2\to \RR$ satisfying
\[
\PP\left(J_{N,\varepsilon}(X_n)\geq \phi(\lambda,\varepsilon)\right)<\lambda
\]
for all $\lambda,\varepsilon>0$ and $N\in\NN$. Such convergence rates provide interesting quantitative information on the oscillatory characteristics of processes, and have been widely studied in the context of both probability theory and ergodic theory (see in particular \cite{Kachurovskii1996}). Indeed, explicit moduli of finite fluctuations for $L_p$-bounded martingales are already provided in \cite{Chashka1994} and later in \cite{JonesKaufmanRosenblattWierdl1998}. In the following result we generalise these to the case of finitary supermartingales, combining elements of the proof of \cite[Theorem 1]{Chashka1994} with our quantitative result on finitary supermartingale convergence (Theorem \ref{thm:fin:learnable:MCT}). We note that an alternative approach here would have been to instead attempt to adapt the entire set of proofs in \cite{Chashka1994} or \cite{JonesKaufmanRosenblattWierdl1998} to manufacture finitary versions, though it is not obvious how these can be extended to supermartingales rather than just martingales, and in any case we conjecture that the improvement in the approximation bound, if any, would be minor.

\begin{theorem}\label{thm:martingalefluc}
Fix $\lambda,\varepsilon>0$ and $N\in\mathbb{N}$. Let $(X_n)$ be a nonnegative stochastic process with moduli of absolute continuity $(\mu_n)$. Let $(X_n)$ be a $\mu^M_N(\varepsilon_0\lambda_0/(16N))$-$(\varepsilon_0\lambda_0/(8N))$-$N$-supermartingale w.r.t.\ $(\mathcal{F}_n)$, where $\lambda_0:=\min\{\mu^M_e(\varepsilon^2/(128e^2))\mid e\leq N\}$ and $\varepsilon_0:=\varepsilon^2/(64N)$. Let $K>0$ be such that $\EE[X_0]<K$. Then, for any $\lambda>0$, we have
\[
\PP\left(\fluc{N}{\varepsilon}{X_n}\geq \Bigl\lceil \frac{2048(K+1)^2}{\lambda^2\varepsilon^2}\Bigr\rceil \right)<\lambda.
\]
Here, as before, we use the notation $\mu^M_N(\varepsilon):=\min\{\mu_{n}(\varepsilon)\mid n\leq N\}$. 
\end{theorem}
\begin{proof}
We define the stopping times $(\tau_j)$ as follows: We set $\tau_0:=0$ and define 
\[
\tau_j:=\begin{cases}\inf\{i\in [\tau_{j-1};N]\mid \vert X_{\tau_{j-1}}-X_i\vert\ge \varepsilon/2\} &\text{if existent},\\N&\text{otherwise}.\end{cases}
\]
Now, setting $e:= \lceil 2048(K+1)^2/(\lambda^2\varepsilon^2)\rceil$, and using that fact that if, for any point $m\in\NN$, $(X_n)$ experiences an $\varepsilon$-fluctuation after time $m$, then there is some point $l\geq m$ such that $|X_m-X_l|\geq \varepsilon/2$, we see that
\begin{equation*}
\PP(\fluc{N}{\varepsilon}{X_n}\geq e) \le \PP\left(\bigcap _{i \le e}\left(|X_{\tau_i}-X_{\tau_{i+1}}|\ge \varepsilon/2\right)\right).
\end{equation*}
Since $\tau_j \le N$ for all $j \in \NN$, by Lemma \ref{lem:quant:optional}, we have that $(X_{\tau_n})$ is a nonnegative $\lambda_0$-$\varepsilon_0$-$N$-supermartingale with respect to the filtration $(\mathcal{F}_{\tau_n})$ with $\EE[X_{\tau_0}]= \EE[X_0]< K$. We now apply Theorem \ref{thm:fin:learnable:MCT}, noting that we can assume $e\leq N$ (otherwise the theorem is trivially true) and therefore that $(X_{\tau_n})$ being a $\lambda_0$-$\varepsilon_0$-$N$-supermartingale implies that is it also an $\mu^M_e(\varepsilon^2/(128e^2))$-$(\varepsilon^2/(64e))$-$e$-supermartingale. Taking $a_i := i$ and $b_i:= i+1$, there thereby exists an $i_0 \le e$ satisfying 
\begin{equation*}
\PP(|X_{\tau_{i_0}}-X_{\tau_{{i_0}+1}}|\ge \varepsilon/2) = \PP(\exists k,l \in [{i_0};{i_0}+1]\, |X_{\tau_k}-X_{\tau_l}|\ge \varepsilon/2)< \lambda.
\end{equation*}
The result follows by noting that then
\begin{equation*}
\PP(\fluc{N}{\varepsilon}{X_n}\geq e) \le \PP\left(\bigcap_{i \le e} \left(|X_{\tau_i}-X_{\tau_{i+1}}|\ge \varepsilon/2\right)\right) \le \PP(|X_{\tau_{i_0}}-X_{\tau_{{i_0}+1}}|\ge \varepsilon/2)<\lambda.\qedhere
\end{equation*}
\end{proof}

\begin{remark}\label{rem:martFlucStrong}
Akin to Remark \ref{rem:optStopStrong}, if $(X_n)$ is already a finitary martingale in the stronger sense that the associated probability is not resolved, then the above Theorem \ref{thm:martingalefluc} can be simplified in a way that the resulting quantitative information no longer depends on associated moduli of absolute continuity. The proof then in particular relies on the more uniform variant of the quantitative optional stopping theorem established in that remark.
\end{remark}

A random variable $X$ is called almost surely finite if $X<+\infty$ $\PP$-a.s., and we resolve that property quantitatively through what we refer to as a modulus of almost sure finiteness (cf.\ \cite[Definition 4.7]{Neri2025}), that is a function $f:(0,\infty)\to\mathbb{N}$ such that $\PP\left(X\geq f(\lambda)\right)<\lambda$ for all $\lambda>0$, which is essentially just a way of reformulating a tail bound for $X$.

Our last preliminary quantitative result before we move to the Robbins-Siegmund theorem is the following, converting a modulus of almost sure finiteness for $\sup_{n\in\NN}X_n<+\infty$ into explicit information on the fluctuation behavior:

\begin{lemma}\label{lem:monotonerates}
If $(X_n)$ is a monotone, nonnegative sequence of random variables such that $\sup_{n\in\NN}X_n<+\infty$ with a modulus of almost sure finiteness $f$, then we have
\[
\PP\left(J_{n,\varepsilon}(X_i)\geq \frac{2f(\lambda)}{\varepsilon}\right)<\lambda 
\]
for all $n\in\mathbb{N}$ and $\varepsilon,\lambda>0$.
\end{lemma}

\begin{proof}
Fix $n\in\mathbb{N}$ as well as $\varepsilon,\lambda>0$. First we have for any $a>0$ that
\[
\PP\left(J_{n,\varepsilon}(X_i)\geq a\right)< \PP\left(J_{n,\varepsilon}(X_i)\geq a\cap \max_{i\leq n} X_i<f\left(\lambda\right)\right)+\lambda
\]
Now, let $J_{n,\varepsilon}(X_i(\omega))\geq a$ and $\max_{i\leq n} X_i(\omega)<f\left(\lambda\right)$ for $a:=2f(\lambda)/\varepsilon>\lceil f(\lambda)/\varepsilon\rceil$. Since the sequence $(X_i(\omega))_{i\leq n}$ is monotone and lies in $[0,f(\lambda)]$, there can be at most $\lceil f(\lambda)/\varepsilon\rceil$-many $\varepsilon$-fluctuations and so we have $\PP\left(J_{n,\varepsilon}(X_i)\geq a\cap \max_{i\leq n} X_i<f\left(\lambda\right)\right)=0$ which yields the result.
\end{proof}

We now come to the main result of the section, and the culmination of our development of finitary martingale theory. The well-known Robbins-Siegmund theorem \cite{RobbinsSiegmund1971} is an extension of Doob's convergence theorem to certain almost supermartingales, stating that whenever $(X_n)$ is a nonnegative process adapted to $(\mathcal{F}_n)$ such that
\[
\EE[X_{n+1}\mid\mathcal{F}_n]\leq (1+\chi_n)X_n+\eta_n-\zeta_n 
\]
for parameter sequences $(\chi_n)$, $(\eta_n)$, $(\zeta_n)$ with $\sum_{i=0}^\infty \chi_i<+\infty$ and $\sum_{i=0}^\infty \eta_i<+\infty$, then it holds that $(X_n)$ converges to a finite limit and $\sum_{i=0}^\infty \zeta_i<+\infty$. It what follows we do not require the second part of the Robbins-Siegmund theorem and we thus restrict our attention to the special case that $\zeta_n=0$. However, akin to the previous results on finitary supermartingales, the key aspect of our quantitative variant of the Robbins-Siegmund theorem is that it only requires a finitary variant of the above almost-supermartingale condition:

\begin{definition}
Let $\lambda,\varepsilon>0$ and $N\in\mathbb{N}$. Let $(X_n)$ be a nonnegative sequence of random variables adapted to $(\mathcal{F}_n)$. The process is called a $\lambda$-$\varepsilon$-$N$-Robbins-Siegmund process w.r.t.\ nonnegative sequences of random variables $(\eta_n),(\chi_n)$ if 
\begin{gather*}
\PP\left(\EE[X_{n+1}\mid\mathcal{F}_n]> (1+\chi_n)X_n+\eta_n+ \varepsilon\right)<\lambda
\end{gather*}
for all $n< N$.
\end{definition}

We now present our quantitative version of the Robbins-Siegmund theorem as appropriately phrased for our finitary almost-supermartingales in the sense of the above definition. Just as the Robbins-Siegmund theorem is a direct extension of the Doob convergence theorem for nonnegative supermartingales, our quantitative result is a direct extension of Theorem \ref{thm:martingalefluc}, where the summability assumptions on $(\eta_n)$ and $(\chi_n)$ are interpreted computationally via a pair of moduli of almost sure finiteness $f,h:(0,\infty)\to \mathbb{N}$ and in the case of $(\chi_n)$ we treat the assumption $\sum_{i=0}^\infty \chi_i<+\infty$ in the equivalent form $\prod_{i=0}^\infty (1+\chi_i)<+\infty$.

\begin{theorem}\label{thm:finRS}
Let $(X_n),(\eta_n),(\chi_n)$ be sequences of nonnegative random variables adapted to a filtration $(\mathcal{F}_n)$ and let $f$ and $h$ be moduli of almost sure finiteness for $\sum_{i=0}^\infty\eta_i<+\infty$ and $\prod_{i=0}^\infty(1+\chi_i)<+\infty$, respectively. Further, let $K> \EE[X_0]$ and let $\mu_n$ be a modulus of absolute continuity for $X_n$. Then there are functions $Z_{f,h,K}(\lambda,\varepsilon)$, $e_{f,h,K}(\lambda,\varepsilon,N)$ and $p_{f,h,K}(\lambda,\varepsilon,N)$, that satisfy the following: For any $\lambda,\varepsilon>0$ and $N\in\NN$, if $(X_n)$ is a $(p_{f,h,K}(\lambda,\varepsilon,N))$-$(e_{f,h,K}(\lambda,\varepsilon,N))$-$N$-Robbins-Siegmund process w.r.t.\ $(\eta_n),(\chi_n)$, then
\[
\PP\left(\fluc{N}{\varepsilon}{X_n} \ge  Z_{f,h,K}(\lambda,\varepsilon) \right)<\lambda.
\]
Moreover, $Z_{f,h,K},e_{f,h,K}$ and $p_{f,h,K}$ can be explicitly defined via
\begin{gather*}
Z_{f,h,K}(\lambda,\varepsilon):=4\Bigl\lceil \frac{2^{17}(K+f(\lambda/4)+1)^2h(\lambda/10)^2}{\lambda^2\varepsilon^2}\Bigr\rceil,\\
e_{f,h,K}(\lambda,\varepsilon,N):=\min\left\{\frac{\varepsilon_0\lambda_0}{8N},\frac{\lambda(b/2+\alpha)}{8},\frac{\varepsilon}{2N}\right\},\\
p_{f,h,K}(\lambda,\varepsilon,N):=\min\left\{\mu^M_N\left(\frac{\varepsilon_0\lambda_0}{16N}\right),\mu^M_N\left(\frac{\varepsilon}{8N^2}\right)\right\},
\end{gather*}
where $\varepsilon_0:=\varepsilon^2/256Nh(\lambda/10)^2$ and $\lambda_0:=\min\{\mu^M_e(\varepsilon^2/512e^2h(\lambda/10)^2)\mid e\leq N\}$ and, as before, $\mu^M_N(\varepsilon):=\min\{\mu_{n}(\varepsilon)\mid n\leq N\}$. 
\end{theorem}

\begin{proof}
We define sequences $(Y_n),(Z_n)$ of nonnegative random variables by $Y_n:=\prod_{i=0}^{n-1}(1+\chi_i)$ and $Z_n:=\sum_{i=0}^{n-1}\frac{\eta_i}{Y_{i+1}}$, noting that both are monotone. We then define $U_n:=X_n/Y_n-Z_n$. Through a repeated application of Lemma \ref{lem:fluc:ineq}, we have for any $a,b,c>0$ and $N \in \NN$:
\[
\begin{aligned}
\PP\left(J_{N,\varepsilon}(X_i)\geq a\right)&\leq \PP\left(J_{N,\frac{\varepsilon}{c}}(U_i+Z_i)\geq \frac{a}{2}\right)+\PP\left(J_{N,\frac{\varepsilon}{b}}(Y_i)\geq \frac{a}{2}\right)\\
&\quad +\PP\left(\max_{i\leq N}(|U_i+Z_i|)\geq b\right)+\PP\left(\max_{i\leq N} Y_i\geq c\right)\\
&\leq \PP\left(J_{N,\frac{\varepsilon}{2c}}(U_i)\geq\frac{a}{4}\right)+\PP\left(J_{N,\frac{\varepsilon}{2c}}(Z_i)\geq\frac{a}{4}\right)+\PP\left(J_{N,\frac{\varepsilon}{b}}(Y_i)\geq \frac{a}{2}\right)\\
&\quad +\PP\left(\max_{i\leq N}|U_i|\geq \frac{b}{2}\right)+\PP\left(\max_{i\leq N}Z_i\geq \frac{b}{2}\right)+\PP\left(\max_{i\leq N} Y_i\geq c\right).
\end{aligned}
\]
Now, let $\varepsilon, \lambda >0$ and $N \in \NN$ be given. Our aim is to now define $a,b,c>0$ and $\mu,\delta > 0$ so that if $(X_n)$ is a $\mu$-$\delta$-$N$-Robbins-Siegmund process w.r.t.\ $(\eta_n),(\chi_n)$, then the first term on the right hand side above is bounded by $\lambda/2$ and the remaining by $\lambda/10$, so that combined we obtain $\PP\left(J_{N,\varepsilon}(X_i)\geq a\right)<\lambda$. We first note that since $Y_i\geq 1$ for all $i\in\mathbb{N}$, we get
\[
\PP\left(\max_{i\leq N} Z_i\geq \frac{b}{2}\right)\leq \PP\left(\sum_{i=0}^\infty\eta_i\geq \frac{b}{2}\right)<\frac{\lambda}{10}\text{ and }\PP\left(\max_{i\leq N}Y_i\geq c\right)\leq  \PP\left(\prod_{i=0}^\infty (1+\chi_i)\geq c\right)<\frac{\lambda}{10}
\]
for $b\geq 2 f\left(\frac{\lambda}{10}\right)$ as well as $c\geq h\left(\frac{\lambda}{10}\right)$. Further, using Lemma \ref{lem:monotonerates}, we have
\[
\PP\left(J_{N,\frac{\varepsilon}{2c}}(Z_i)\geq \frac{a}{4}\right)<\frac{\lambda}{10} \ \ \ \mbox{for} \ a\geq \frac{16c}{\varepsilon}\cdot f\left(\frac{\lambda}{10}\right)
\]
and similarly
\[
\PP\left(J_{N,\frac{\varepsilon}{b}}(Y_i)\geq\frac{a}{2}\right)<\frac{\lambda}{10} \ \ \ \mbox{for} \ a\geq \frac{4b}{\varepsilon}\cdot h\left(\frac{\lambda}{10}\right)
\]
It therefore remains to bound fluctuations and maxima for $(U_n)$. Now suppose $(X_n)$ is a $\mu$-$\delta$-$N$-Robbins-Siegmund process w.r.t.\ $(\eta_n),(\chi_n)$ for some $\mu,\delta>0$. Then we must have that $(U_n)$ is a $\mu$-$\delta$-$N$-supermartingale w.r.t.\ $(\mathcal{F}_n)$. To see this, let $n< N$ and let $A$ denote the event
\[
\EE[X_{n+1}\mid\mathcal{F}_n]\leq (1+\chi_n)X_n+\eta_n+ \delta.
\]
Then by standard properties of the conditional expectation, there is a set $\Omega'$ of measure one such that on $\Omega'\cap A$ we have 
\[
\EE[U_{n+1} \mid \mathcal{F}_n]\leq \frac{(1+\chi_n)X_n+\eta_n+\delta}{Y_{n+1}}-Z_{n+1}\leq \frac{X_n}{Y_n}-Z_n+\delta=U_n+\delta.
\]
Therefore, we in particular have
\[
\PP\left(\EE[U_{n+1} \mid \mathcal{F}_n]>U_n+\delta\right)\leq \PP\left(\EE[X_{n+1}\mid\mathcal{F}_n]> (1+\chi_n)X_n+\eta_n+ \delta\right)<\mu
\]
for any such $n<N$. Now, for arbitrary $\alpha>0$, we define $T_\alpha:=\inf\left\{j\in\mathbb{N} \mid Z_{j+1}>\alpha\right\}$. Since $Z_{j+1}$ is $\mathcal{F}_j$-measurable for all $j\in\NN$, this forms a stopping time. By Lemma \ref{lem:stoppingtime}, the stopped process $(U_{i\wedge T_\alpha})$ is also a $\mu$-$\delta$-$N$-supermartingale, and moreover, since $U_{i\wedge T_\alpha}\geq -Z_{i\wedge T_\alpha}\geq -\alpha$, the process $(U_{i\wedge T_\alpha}+\alpha)$ is a nonnegative $\mu$-$\delta$-$N$-supermartingale w.r.t.\ $(\mathcal{F}_n)$. Next, we observe that
\[
\PP\left(J_{N,\frac{\varepsilon}{2c}}(U_i)\geq \frac{a}{4}\right)\leq \PP\left(J_{N,\frac{\varepsilon}{2c}}(U_i)\geq \frac{a}{4}\cap \max_{i\leq N} Z_{i+1}<f\left(\frac{\lambda}{4}\right)\right)+\PP\left(Z_{N+1} \geq f\left(\frac{\lambda}{4}\right)\right)
\]
and since $\PP\left(Z_{N+1} \geq f\left(\frac{\lambda}{4}\right)\right)\leq \PP\left(\sum_{i=0}^\infty\eta_i\geq f\left(\frac{\lambda}{4}\right)\right)<\frac{\lambda}{4}$ it suffices to show that
\[
\PP\left(J_{N,\frac{\varepsilon}{2c}}(U_i)\geq \frac{a}{4}\cap \max_{i\leq N} Z_{i+1}<f\left(\frac{\lambda}{4}\right)\right)<\frac{\lambda}{4}.
\]
For this, we now fix $\alpha:=f(\lambda/4)$ so that if $\omega$ is in the above measured set, we must have, $T_\alpha(\omega)>N$ and thus $U_{i\wedge T_\alpha(\omega)}(\omega)=U_i(\omega)$ for all $i\leq N$. So in particular
\[
\PP\left(J_{N,\frac{\varepsilon}{2c}}(U_i)\geq \frac{a}{4}\cap \max_{i\leq N} Z_{i+1}<f\left(\frac{\lambda}{4}\right)\right)\leq \PP\left(J_{N,\frac{\varepsilon}{2c}}(U_{i\wedge T_\alpha}+\alpha)\geq \frac{a}{4}\right),
\]
where the last inequality follows from the fact that the fluctuations of a process are unchanged by adding a constant. We can now appeal to Theorem \ref{thm:martingalefluc}, also noting that $\EE[U_{0\wedge T_\alpha}+\alpha]< K+\alpha$, by which we have
\[
\PP\left(J_{N,\frac{\varepsilon}{2c}}(U_{i\wedge T_\alpha}+\alpha)\geq \frac{a}{4}\right)<\frac{\lambda}{4}
\]
for $\delta\leq \varepsilon_0\lambda_0/(8N)$ and $\mu\leq \mu^M_N(\varepsilon_0\lambda_0/(16N))$ with $\varepsilon_0:=\varepsilon^2/(256Nc^2)$ as well as $\lambda_0:=\min\{\mu^M_e(\varepsilon^2/(512e^2c^2))\mid e\leq N\}$ as well as $a\geq 4\lceil 2^{17}(K+\alpha+1)^2c^2/(\lambda^2\varepsilon^2)\rceil$. Finally, using a similar argument, we see that
\[
\PP\left(\max_{i\leq N}|U_i|\geq \frac{b}{2}\right)\leq \PP\left(\max_{i\leq N}|U_i|\geq \frac{b}{2}\cap \max_{i\leq N} Z_{i+1}<\alpha\right)+\frac{\lambda}{4}\leq\PP\left(\max_{i\leq N}|U_{i\wedge T_\alpha}|\geq \frac{b}{2}\right)+\frac{\lambda}{4}
\]
and if $b/2> \alpha$, then since $U_{i\wedge T_\alpha}\geq-\alpha$ we have $|U_{i\wedge T_\alpha}|\geq b/2$ if and only if $U_{i\wedge T_\alpha}\geq b/2$, and thus
\[
\PP\left(\max_{i\leq N}|U_{i\wedge T_\alpha}|\geq \frac{b}{2}\right)=\PP\left(\max_{i\leq N}(U_{i\wedge T_\alpha}+\alpha)\geq \frac{b}{2}+\alpha\right)\leq \frac{K+\alpha}{b/2+\alpha}+\frac{\delta}{b/2+\alpha}
\]
with the final inequality following from Lemma \ref{lem:fin:Ville} if $\mu\leq \mu^M_N(\varepsilon/(8N^2))$ and $\delta\leq \varepsilon/(2N)$. Therefore we get $\PP\left(\max_{i\leq n}|U_{i\wedge T_\alpha}|\geq \frac{b}{2}\right)<\frac{\lambda}{4}$ if $\frac{b}{2}> \alpha$ and $\delta<\frac{\lambda(b/2+\alpha)}{8}$ and $b\geq \frac{16}{\lambda}(K+\alpha)$. We now put together all of our conditions to obtain
\begin{align*}
\alpha&:=f\left(\frac{\lambda}{4}\right),\quad c:=h\left(\frac{\lambda}{10}\right),\quad b:=\max\left\{2f\left(\frac{\lambda}{10}\right),2\alpha+1,\frac{16(K+\alpha)}{\lambda}\right\},\\
a&:=\max\left\{\frac{16c}{\varepsilon}\cdot f\left(\frac{\lambda}{10}\right),\frac{4b}{\varepsilon}\cdot h\left(\frac{\lambda}{10}\right),4\Bigl\lceil \frac{2^{17}(K+\alpha+1)^2c^2}{\lambda^2\varepsilon^2}\Bigr\rceil\right\},\\
\delta&:=\min\left\{\frac{\varepsilon_0\lambda_0}{8N},\frac{\lambda(b/2+\alpha)}{8},\frac{\varepsilon}{2N}\right\},\\
\mu&:=\min\left\{\mu^M_N\left(\frac{\varepsilon_0\lambda_0}{16N}\right),\mu^M_N\left(\frac{\varepsilon}{8N^2}\right)\right\},
\end{align*}
with $\varepsilon_0:=\varepsilon^2/(256Nc^2)$ and $\lambda_0:=\min\{\mu^M_e(\varepsilon^2/(512e^2c^2))\mid e\leq N\}$ as before. The value $a$ can be simplified to
\[
a:=4\Bigl\lceil \frac{2^{17}(K+f(\lambda/4)+1)^2h(\lambda/10)^2}{\lambda^2\varepsilon^2}\Bigr\rceil.
\]
All in all, for these values we have shown that if $(X_n)$ is a $\mu$-$\delta$-$N$-Robbins-Siegmund process w.r.t.\ $(\eta_n),(\chi_n)$, then $\PP(J_{n,\varepsilon}(X_i)\geq a)<\lambda$ and from this, the result follows.
\end{proof}

\begin{remark}\label{rem:RSStrong}
Akin to Remarks \ref{rem:optStopStrong} and \ref{rem:martFlucStrong}, if $(X_n)$ is already a finitary Robbins-Siegmund process in the stronger sense that the associated probability is not resolved, then the above Theorem \ref{thm:finRS} can be simplified in a way that the resulting quantitative information no longer depends on associated moduli of absolute continuity. The proof then in particular relies on the more uniform variant of Theorem \ref{thm:martingalefluc} that is discussed in Remark \ref{rem:martFlucStrong}.
\end{remark}

\section{Quantitative convergence of stochastic quasi-Fej\'er monotone sequences}
\label{sec:main}

We now come to the secon main part of the paper, where we combine our quantitative theory of finitary almost-supermartingales with a stochastic lift of the quantitative framework for quasi-Fej\'er monotone sequences in \cite{KohlenbachLeusteanNicolae2018} to provide a quantitative convergence theorem for stochastic quasi-Fej\'er monotonicity in the presence of compactness. We begin with a series of definitions, most of which are stochastic analogues of simpler concepts from \cite{KohlenbachLeusteanNicolae2018}.

Throughout, as mentioned before, let $(X,d)$ be a metric space with a fixed but arbitrary point $o\in X$ acting as a center and let $(\Omega,\mathsf{F},\PP)$ be a given probability space. On $X$, let $F\subseteq X$ be a non-empty set. Our main concern is then with sequences of $X$-valued random variables $(x_n)$ which are stochastically quasi-Fej\'er monotone relative to $F$ in the following sense, generalizing the seminal notion as e.g.\ investigated in \cite{CombettesPesquet2015} in the setting of Hilbert spaces\footnote{Recall footnotes \ref{foot:diff1} -- \ref{foot:diff3} from the discussion in the introduction which in detail discuss the motivation for the slight differences between the notion studied in \cite{CombettesPesquet2015} and the one presented in the following.} to general metric spaces, suitable to the present setting.

\begin{definition}\label{def:stochasticFejer}
Let $\mathcal{F}=(\mathcal{F}_n)$ be a filtration of $\mathsf{F}$. An $X$-valued stochastic process $(x_n)$ adapted to $\mathcal{F}$ is stochastically quasi-Fej\'er monotone w.r.t.\ $F$ and $\mathcal{F}$ if 
\[
\forall n\in\mathbb{N}\ \forall z\in F\left(\EE[G(d(x_{n+1},z))\mid \mathcal{F}_n]\leq (1+\chi_n)G(d(x_n,z))+\eta_n\right) \AS
\]
where $G:[0,\infty)\to [0,\infty)$ is an increasing and continuous function with $G(0)=0$, and $(\eta_n),(\chi_n)\in \ell^1_+(\mathcal{F})$.
\end{definition}

These sequences enjoy particularly nice convergence properties if they are combined with a suitable approximation property towards the set $F$, regarded as an abstract representation of a set of solutions. Following \cite{KohlenbachLeusteanNicolae2018}, we first encapsulate abstractly what being approximate to being a point in $F$ means here by presuming a decomposition
\[
F=\bigcap_{k\in\mathbb{N}}AF_k
\]
for a sequence $(AF_k)$ of measurable subsets of $X$ with the property that $AF_{k+1}\subseteq AF_k$ for any $k\in\mathbb{N}$. The set $AF_k$ serves as an abstract representation of the set of all points conceived of as being ``$k$-good'' approximations of being in $F$. In this abstract formulation, the present setup for handling solution sets and their approximations is indeed very general and immediately encompasses many if not most problems naturally occuring in the literature. Deterministic examples are discussed at length in \cite{KohlenbachLeusteanNicolae2018}, while a genuine stochastic example will be studied in Section \ref{sec:common} later on.

The approximation property we then consider here is then the following particularly mild assumption, which forms a stochastic analogue of \cite[Definition 3.1, (ii)]{KohlenbachLeusteanNicolae2018}:

\begin{definition}\label{def:stochasticLiminf}
An $X$-valued stochastic process $(x_n)$ has the stochastic $\liminf$-property w.r.t.\ $F$ and $(AF_k)$ if
\[
\forall k,N\in\mathbb{N}\ \exists n\geq N(x_n\in AF_k)\AS
\]
\end{definition}

Both stochastic quasi-Fej\'er monotonicity and the above approximation property need to be treated in a quantitative way, which we now discuss in detail, beginning with the latter. Under the assumption of these quantitative variants, we will then provide a quantitative almost sure convergence result for stochastic quasi-Fej\'er monotone sequences under a relative compactness assumption on $X$.

\subsection*{Notational convention} We so far have used general real numbers $\varepsilon>0$ to represent arbitrarily small approximate errors but in this section choose to represent said quantities by expressions of the form $1/(k+1)$ for $k\in\NN$. This is consistent with earlier papers on the topic of quantitative aspects of Fej\'er monotone sequences and in particular allows one to easily judge the measurability of various involved sets in the following parts, which feature various complicated quantitative notions inside probabilities. In order to now freely move between these conventions without onerous bureaucracy, we make use of a slight abuse of notation where for a function $f(\varepsilon)$ taking as argument some real $\varepsilon>0$, for $k\in\NN$ we write $f(k)$ to implicity mean $f(1/(k+1))$ (for example, we write $J_{N,k}(X_n)$ for the number of $1/(k+1)$-fluctuations experienced by $(X_n)$ before time $N$). Conversely, if $g(k)$ is a function that refers to the quantity $1/(k+1)$, for $\varepsilon>0$ we write $g(\varepsilon)$ to mean $g(k)$ for the least $k\in\NN$ with $1/(k+1)<\varepsilon$. In practice, this notation is completely unambiguous.

\subsection{Stochastic $\liminf$-moduli}

We start by noting that through elementary properties of probability spaces (recall Lemma \ref{lem:prenex}), the stochastic $\liminf$-property discussed in Definition \ref{def:stochasticLiminf} above can be rewritten in such a way that the hidden quantifiers in the almost sure requirement are expanded and then brought outside of the probability measure. Concretely, note that we can equivalently write
\begin{align*}
&\forall k,N\in\mathbb{N}\ \exists n\geq N(x_n\in AF_k)\AS\\
&\qquad\text{if, and only if, }\PP(\forall k,N\in\mathbb{N}\ \exists n\geq N(x_n\in AF_k))=1\\
&\qquad\text{if, and only if, }\forall\lambda>0\ \forall k,N\in\mathbb{N}\ \exists m\in\mathbb{N} \left(\PP(\forall n\in [N;m] (x_n\not\in AF_k))<\lambda\right).
\end{align*}
Therefore the stochastic $\liminf$-property for an $X$-valued stochastic process $(x_n)$ is equivalent to the existence a so-called stochastic $\liminf$-modulus in the sense of the following definition, providing a bound (and hence a witness) for the quantifier over $m$ in the last line above:

\begin{definition}\label{def:stochasticLiminfMod}
A function $\Phi:(0,\infty)\times\mathbb{N}^2\to\mathbb{N}$ is a stochastic $\liminf$-modulus w.r.t.\ $F$ and $(AF_k)$ for an $X$-valued stochastic process $(x_n)$ if
\[
\forall\lambda>0\ \forall k,N\in\NN\left(\PP\left(\forall n\in [N;\Phi(\lambda,k,N)](x_n\not\in AF_k)\right)<\lambda\right).
\]    
\end{definition}

At a crucial point in the course of the proof of the main result, we will however have to rely on a qualitatively equivalent but quantitatively strengthened variant of the above modulus, namely we will require the existence of a function $\Psi:(0,\infty)\times\mathbb{N}^2\to\mathbb{N}$ such that
\[
\forall\lambda>0\ \left(\PP\left(\exists k,N\in\NN\ \forall n\in [N;\Psi(\lambda,k,N)](x_n\not\in AF_k)\right)<\lambda\right).\tag{$*$}
\]
Naturally such a modulus is also a stochastic $\liminf$-modulus in the sense of Definition \ref{def:stochasticLiminfMod} and so is a priori stronger. On the other hand, we can construct a modulus satisfying $(*)$ in terms of a stochastic $\liminf$-modulus as the following result shows, though this comes at the expense of an exponential speedup in the probability.

\begin{lemma}\label{lem:stochasticLiminfModStrong}
If $\Phi$ is a stochastic $\liminf$-modulus w.r.t.\ $F$ and $(AF_k)$ for an $X$-valued stochastic process $(x_n)$, then $\Psi(\lambda,k,N):=\Phi\left(\lambda 2^{-(k+N+2)},k,N\right)$ satisfies $(*)$.
\end{lemma}
\begin{proof}
Let $\lambda>0$. Then we have
\begin{align*}
&\PP\left(\exists k,N\in\NN\, \forall n\in [N;\Phi\left(\lambda 2^{-(k+N+2)},k,N\right)](x_n\notin AF_k)\right)\\
&\qquad\leq\sum_{k=0}^\infty\sum_{N=0}^\infty \PP\left(\forall n\in [N;\Phi\left(\lambda 2^{-(k+N+2)},k,N\right)](x_n\notin AF_k)\right)<\sum_{k=0}^\infty\sum_{N=0}^\infty\lambda 2^{-(k+N+2)}=\lambda.\qedhere
\end{align*}
\end{proof}

\begin{remark}[For logicians]\label{rem:liminfExtr}
Similar to \cite{KohlenbachLeusteanNicolae2018}, the above notion of a stochastic $\liminf$-modulus is motivated by the proof-theoretic methodology of proof mining. In particular, general logical metatheorems for proof mining in probability (see again \cite{NeriPischke2024} and in particular also \cite{NeriOlivaPischke2026}) guarantee that such moduli can in fact be extracted from large classes of proofs of the underlying non-quantitative stochastic $\liminf$-property, provided all $(AF_k)$ can be expressed via an $\exists$-formula (in the terminology of \cite{NeriPischke2024}).
\end{remark}

\subsection{Moduli of uniform stochastic quasi-Fej\'er monotonicity}

We now arrive at a uniform strengthening of the notion of being stochastically quasi-Fej\'er monotone, again following the precedent set out in \cite{KohlenbachLeusteanNicolae2018} but expanding everything to the stochastic setting. To begin with, we highlight the hidden quantifiers in $z\in F$ in the formulation of the property in Definition \ref{def:stochasticFejer} by expanding it into $\forall k\in\mathbb{N}\left( z\in AF_k\right)$: 
\begin{gather*}
\forall n\in\mathbb{N}\ \forall z\in X\big(\forall k\in\mathbb{N}\left( z\in AF_k\right)\\
\to\forall m\leq n\left(\EE[G(d(x_{m+1},z))\mid \mathcal{F}_m]\le (1+\chi_m)G(d(x_{m},z))+\eta_m\AS\right)\big).
\end{gather*}
Note that we also have introduced an additional bounded quantifier over $m$, which does not change the property but has the effect that the inner matrix of the property becomes monotone, which is particularly preferable in the context of probabilistic notions. Next, we expand the hidden quantifier in the inequality $\leq$ and in the $\PP$-a.s.\ property and suitably prenex the expression to arrive at the following equivalent phrasing:
\begin{gather*}
\forall \lambda>0\ \forall r,n\in\mathbb{N}\ \forall z\in X\ \exists k\in\mathbb{N}\ \forall m\leq n\\
\left(z\in AF_k\to \left(\PP\left(\EE[G(d(x_{m+1},z))\mid \mathcal{F}_m]> (1+\chi_m)G(d(x_{m},z))+\eta_m+\frac{1}{r+1}\right)<\lambda\right)\right).
\end{gather*}
We now consider a uniform variant of this property, where such a $k$ exists uniformly for all points $z$, giving rise to a new, stronger notion of stochastic Fej\'er monotonicity naturally accompanied by quantitative information in terms of a certain modulus providing a bound (and hence a witness) for the $k$ in terms of $\lambda$ and $r,n$ in the above property:

\begin{definition}\label{def:uniformstochasticFejer}
Let $\mathcal{F}=(\mathcal{F}_n)$ be a filtration of $\mathsf{F}$. An $X$-valued stochastic process $(x_n)$ adapted to $\mathcal{F}$ is uniformly stochastically quasi-Fej\'er monotone w.r.t.\ $F$, $(AF_k)$ and $\mathcal{F}$ if
\begin{gather*}
\forall \lambda>0\ \forall r,n\in\mathbb{N}\ \exists k\in\mathbb{N}\ \forall z\in AF_k\ \forall m\leq n\\
\left(\PP\left(\EE[G(d(x_{m+1},z))\mid \mathcal{F}_m]> (1+\chi_m)G(d(x_{m},z))+\eta_m+\frac{1}{r+1}\right)<\lambda\right)
\end{gather*}
where $G:[0,\infty)\to [0,\infty)$ is an increasing and continuous function with $G(0)=0$, and $(\eta_n),(\chi_n)\in \ell^1_+(\mathcal{F})$. A function $\zeta:(0,\infty)\times \mathbb{N}^2\to\mathbb{N}$ is a modulus of uniform stochastic quasi-Fej\'er monotonicity w.r.t.\ $F$, $(AF_k)$ and $\mathcal{F}$ for $(x_n)$ if $k=\zeta(\lambda,r,n)$ satisfies the above property for any $\lambda>0$ and all $r,n\in\mathbb{N}$.
\end{definition}

In fact, it can be shown that this notion coincides with the usual notion of stochastic quasi-Fej\'er monotonicity over compact spaces, and under suitable continuity assumptions, and we discuss this later in the context of the proof of Theorem \ref{thm:final}.

\begin{remark}[For logicians]\label{rem:FejerExtr}
Similar to \cite{KohlenbachLeusteanNicolae2018}, the above notion of a modulus of uniform stochastic quasi-Fej\'er monotonicity is motivated by logic, and also here, general logical metatheorems for probability \cite{NeriOlivaPischke2026,NeriPischke2024} guarantee that such moduli can in fact be extracted from large classes of proofs of the underlying non-quantitative stochastic quasi-Fej\'er monotonicity property, provided all $(AF_k)$ can be expressed via a $\forall$-formula (in the terminology of \cite{NeriPischke2024}). While this seems to be add odds with the existential condition on the sets $(AF_k)$ as discussed in Remark \ref{rem:liminfExtr}, this tension can be resolved in many practical circumstances and we refer to \cite{KohlenbachLeusteanNicolae2018} for discussions on that matter.
% the proof-theoretic methodology of proof mining
\end{remark}

\subsection{A quantitative almost sure convergence theorem under relative compactness}

We now move on to our main quantitative result, providing an effective almost convergence result for uniformly stochastic quasi-Fej\'er monotone sequences which have the stochastic $\liminf$-property. As mentioned before, this will take place under a relative compactness assumption which also has to be effectivized:

\begin{definition}[essentially \cite{Gerhardy2008}]
For $A\subseteq X$, a function $\gamma:\mathbb{N}\to\mathbb{N}$ is called a modulus of total boundedness for $A$ if for all $k\in\mathbb{N}$ and any $(x_n)\subseteq A$:
\[
\exists 0\leq i<j\leq\gamma(k)\left( d(x_i,x_j)\leq \frac{1}{k+1}\right).
\]
\end{definition}

It can be easily seen that a set $A$ permits such a modulus if, and only if, it is totally bounded in the usual sense (see also the discussion in \cite{KohlenbachLeusteanNicolae2018}). Also, note that every modulus of total boundedness $\gamma$ satisfies $\gamma(k)\geq 1$ for any $k\in\NN$. For a wide range of examples of totally bounded sets where such moduli can be readily computed, we refer to \cite{KohlenbachLeusteanNicolae2018}. For instance, consider a bounded set $A\subseteq H$ with bound $b$ in a finite dimensional Hilbert space $H$ with dimension $d$. Then
\[
\gamma(k):=\ceil*{2(k+1)\sqrt{d}b}^d
\]
is a modulus of total boundedness for $A$ (cf.\ Example 2.8 in \cite{KohlenbachLeusteanNicolae2018}).

For the rest of this section, unless stated otherwise, we will assume that $AF_k\subseteq X_0$ for any $k\in\mathbb{N}$ where $X_0\subseteq X$ is a totally bounded set. Next to a modulus of total boundedness, we will often rely on a further numerical datum corresponding to $X_0$, a bound $r_0$ relative to $o$ for $X_0$, i.e.\ $r_0$ satisfies $d(z,o)\leq r_0$ for all $z\in X_0$.

We now build up to our main result via a sequence of lemmas. The first such result is the following Lemma \ref{lem:fromRStoPsi}, which represents the main instantiation of our quantitative Robbins-Siegmund theorem (Theorem \ref{thm:finRS}) in the context of uniform stochastic quasi-Fej\'er monotone sequences, by which we obtain fluctuation bounds for the sequence $(G(x_n,z))$ for all $z$ in a suitably good approximation $AF_k$ to $F$.

\begin{lemma}\label{lem:fromRStoPsi}
Let $(x_n)$ be an $X$-valued stochastic process adapted to a filtration $\mathcal{F}=(\mathcal{F}_n)$ and let $(\eta_n),(\chi_n)\in \ell^1_+(\mathcal{F})$. Assume that we have functions $f$ and $h$ which are moduli of almost sure finiteness for $\sum_{i=0}^\infty\eta_i<+\infty$ and $\prod_{i=0}^\infty(1+\chi_i)<+\infty$, respectively, along with a modulus of uniform stochastic quasi-Fej\'er monotonicity $\zeta(\lambda,r,n)$ for $(x_n)$ w.r.t.\ $F$, $(AF_k)$ and $\mathcal{F}$. Further, let $r_0$ be a bound relative to $o$ for $X_0$. Finally, assume that $b_0$ is such that $b_0> \EE[G(d(x_0,o)+r_0)]$ and that $\mu_n$ is a modulus of absolute continuity for $G(d(x_n,z))$, uniformly in $z\in X_0$.

Then there are functions $\chi_0(\lambda,k,N)$ and $\psi_0(\lambda,k)$ such that for any $\lambda>0$ and $k,N\in\mathbb{N}$:
\[
\forall z\in AF_{\chi_0(\lambda,k,N)}\left(\PP\left(\fluc{N}{k}{G(d(x_n,z))}\geq \psi_0(\lambda,k)\right)<\lambda\right).
\]
In fact, $\chi_0$ and $\psi_0$ can simply be defined by $\psi_0(\lambda,k):=Z_{f,h,b_0}(\lambda,k)$ and
\[
\chi_0(\lambda,k,N):=\zeta(p_{f,h,b_0}(\lambda,k,N),e_{f,h,b_0}(\lambda,k,N),N)
\]
with $Z$, $e$ and $p$ from Theorem \ref{thm:finRS} (also defined in terms of $\mu_n$).
\end{lemma}
\begin{proof}
As $G$ is increasing, note that we have 
\[
\EE[G(d(x_0,z))]\leq \EE[G(d(x_0,o)+r_0)]< b_0
\]
for any $z\in X_0$, hence in particular for any $z\in AF_{\zeta(\lambda,r,n)}$ for any $r,n\in\mathbb{N}$. The lemma then follows, for any $z\in AF_{\chi_0(\lambda,k,N)}$, directly from Theorem \ref{thm:finRS} with $X_n=G(d(x_n,z))$.
\end{proof}

\begin{remark}
The above Lemma \ref{lem:fromRStoPsi}, and by that also the later Theorem \ref{thm:mainQuant}, assumes that we are equipped with moduli of absolute continuity $\mu_n$ which apply to $G(d(x_n,z))$ uniformly over $z\in X_0$. This uniformity over $z$ could be quite strong at times, but in many situations can still be readily derived. First, in case that $G=\mathrm{Id}$, this assumption reduces to the existence of moduli of absolute continuity for $d(x_n,o)$ as if $\mu_n$ is a modulus of absolute continuity for $d(x_n,o)$, then by the triangle inequality and the fact that $X_0$ is bounded relative to $o$ by $r_0$, we have that $\tilde{\mu}_n(\varepsilon):= \min\{\mu_n(\varepsilon/2),\varepsilon/(2r_0)\}$ is a modulus of absolute continuity for $d(x_n,z)$, uniformly over $z \in X_0$. For more complicated $G$, such a modulus might potentially be more difficult to derive in certain situations. However, in most typical situations these are also readily available, already from the way the iterative method $(x_n)$ is defined. A concrete example for such a case, featuring an extremely common type of iteration, is given in Section \ref{sec:common}, and in general, the existence of such moduli can often be guaranteed through results from mathematical logic (cf.\ Remark \ref{rem:modAbsContDer}). Lastly, recall the previous Remarks \ref{rem:optStopStrong}, \ref{rem:martFlucStrong} and \ref{rem:RSStrong} (and note the upcoming Remark \ref{rem:mainQuantStrong}), by which the dependence on such moduli disappears under a stronger approximate stochastic quasi-Fej\'er monotonicity condition.
\end{remark}

Note that whenever $F\neq \emptyset$ and $G$ is coercive, we have that $(d(x_n,o))$ is uniformly almost surely finite, i.e.
\[
\sup_{n\in\NN}d(x_n,o)<+\infty \AS
\]
This follows by applying the Robbins-Siegmund theorem (in its usual, that is non-finitary, form), by which $(G(d(x_n,z)))$ converges and is thus almost surely finite for any $z\in F$. Since $G$ is coercive, we therefore have that $(d(x_n,z))$, and thus also $(d(x_n,o))$, is almost surely finite. In the following, we will crucially rely on the assumption that $(d(x_n,o))$ is uniformly almost surely finite, and hence will require a quantitative rendering of that property. For that purpose, we introduce the notion of a modulus of uniform almost sure finiteness for the random variables $(d(x_n,o))$, that is a function $b:(0,\infty)\to\mathbb{N}$ such that
\[
\forall\lambda>0\left( \PP\left(\exists n\in\mathbb{N}\left( d(x_n,o)>b_\lambda\right)\right)<\lambda\right),
\]
which will feature crucially in the following results.\footnote{This explicit formulation of almost sure finiteness is related to the notion of \emph{tightness}, see \cite[Section 4.1]{NeriPowell2025a} for a discussion of this point.}

\begin{remark}\label{rem:uniformfinite}
Note that such a modulus of uniform almost sure finiteness for $(d(x_n,o))$ can be derived from an upper bound for the distance $d(z,o)$ for a solution $z\in F$ together with a modulus of coercivity $\kappa$ for $G$, i.e.\ such that $\kappa$ satisfies $G(x)\geq b$ for all $x\geq \kappa(b)$, by analysing the above argument. Concretely, one can first derive a modulus of uniform almost sure finiteness for $(G(d(x_n,z)))$ via the quantitative Robbins-Siegmund theorem of \cite{NeriPowell2026} (see Corollary 3.3 therein), which applies under the same quantitative assumptions on $(\eta_n)$ and $(\chi_n)$ as used in this paper. This can then be converted to a corresponding modulus for $(d(x_n,o))$ using the modulus of coercivity. Instead of resolving this argument fully, that is assuming $F\neq\emptyset$ and the coercivity of $G$ everywhere and dealing with the resulting complex modulus explicitly, we here leave this construction as a sketch, preferring a more abstract approach where we simply assume directly that $(d(x_n,o))$ is uniformly almost surely finite and that we are given a modulus of uniform almost sure finiteness for $d(x_n,o)$.
\end{remark}

The following Lemma \ref{lem:unifLemma} now outlines conditions which allow us to bring the quantification over $z$ inside the measure, controlling the probability that there exists some $z$ such that $(G(x_n,z))$ experiences a large number of fluctuations. In particular, we rely on a modulus of uniform continuity for $G$ on bounded sets, i.e.\ a function $\iota:(0,\infty)\times\mathbb{N}\to \mathbb{N}$ such that
\[
\forall b>0\ \forall k\in\mathbb{N}\ \forall x,y\in [0,b]\left( \vert x-y\vert\leq\frac{1}{\iota_b(k)+1}\to \vert G(x)-G(y)\vert\leq\frac{1}{k+1}\right).
\]

\begin{lemma}\label{lem:unifLemma}
Let $G$ be uniformly continuous on bounded sets with a modulus $\iota$. Let $\gamma$ be a modulus of total boundedness, $r_0$ be a bound relative to $o$ for $X_0$, and let $(x_n)$ be a sequence of $X$-valued random variables with functions $\chi_0,\psi_0$ such that for any $\lambda>0$ and $k,N\in\mathbb{N}$:
\[
\forall z\in AF_{\chi_0(\lambda,k,N)}\left(\PP\left(\fluc{N}{k}{G(d(x_n,z))}\geq \psi_0(\lambda,k)\right)<\lambda\right).
\]
Assume that $b:(0,\infty)\to\mathbb{N}$ is a modulus of uniform almost sure finiteness for $(d(x_n,o))$. Then for any $\lambda>0$ and $k,N\in\mathbb{N}$ and  any finite set $Z_0\subseteq AF_{\chi_1(\lambda,k,N)}$:
\[
\PP\left(\exists z\in Z_0\left( \fluc{N}{k}{G(d(x_n,z))}\geq \psi_1(\lambda,k)\right)\right)<\lambda
\]
where 
\begin{gather*}
\chi_1(\lambda,k,N)=\chi_0\left(\frac{\lambda}{2\gamma(\iota_{b_{\lambda/2}+r_0}(3k+2))},3k+2,N\right),\\
\psi_1(\lambda,k)=\psi_0\left(\frac{\lambda}{2\gamma(\iota_{b_{\lambda/2}+r_0}(3k+2))},3k+2\right).
\end{gather*}
\end{lemma}

\begin{proof}
Fix $\lambda>0$ and $k,N\in\mathbb{N}$ and write
\[
k_1:=3k+2\text{ as well as }\lambda_1:=\frac{\lambda}{2\gamma(\iota_{b_{\lambda/2}+r_0}(3k+2))}.
\]
By assumption, we have
\[
\forall z\in AF_{\chi_0(\lambda_1,k_1,N)}\left(\PP\left(\fluc{N}{k_1}{G(d(x_n,z))}\geq \psi_0(\lambda_1,k_1)\right)<\lambda_1\right).
\]
As $\gamma$ is a modulus of total boundedness for $X_0$, it is also a modulus of total boundedness for $AF_{\chi_0(\lambda_1,k_1,N)}$ and so by \cite[Proposition 2.4, (ii)]{KohlenbachLeusteanNicolae2018} there are  
\[
a_1,\dots,a_{\gamma(\iota_{b_{\lambda/2}+r_0}(3k+2))}\in AF_{\chi_0(\lambda_1,k_1,N)}=AF_{\chi_1(\lambda,k,N)}
\]
such that for any $z\in AF_{\chi_1(\lambda,k,N)}$, there exists an $h\in [1;\gamma(\iota_{b_{\lambda/2}+r_0}(3k+2))]$ such that
\[
d(z,a_h)\leq\frac{1}{\iota_{b_{\lambda/2}+r_0}(3k+2)+1}.
\]
First observe that we have
\[
\PP\left(\fluc{N}{k_1}{G(d(x_n,a_h))}\geq \psi_1(\lambda,k)\right)<\lambda_1=\frac{\lambda}{2\gamma(\iota_{b_{\lambda/2}+r_0}(3k+2))}
\]
for any $h\in[1;\gamma(\iota_{b_{\lambda/2}+r_0}(3k+2))]$ and so in particular
\[
\PP\left(\exists h\in[1;\gamma(\iota_{b_{\lambda/2}+r_0}(3k+2))]\left(\fluc{N}{k_1}{G(d(x_n,a_h))}\geq \psi_1(\lambda,k)\right)\right)<\frac{\lambda}{2}.
\]
Now let $\omega$ be such that $\fluc{N}{k_1}{G(d(x_n(\omega),a_h))}<\psi_1(\lambda,k)$ for all $h\in [1;\gamma(\iota_{b_{\lambda/2}+r_0}(3k+2))]$ as well as $d(x_n(\omega),o)\leq b_{\lambda/2}$ for all $n\in\mathbb{N}$. Let $z\in Z_0$ for some fixed finite set $Z_0\subseteq AF_{\chi_1(\lambda,k,N)}$. By the above, there exists an $h\in [1;\gamma(\iota_{b_{\lambda/2}+r_0}(3k+2))]$ such that
\[
\vert d(x_i(\omega),z)-d(x_i(\omega),a_h)\vert\leq d(z,a_h)\leq\frac{1}{\iota_{b_{\lambda/2}+r_0}(3k+2)+1}
\]
as well as $d(x_i(\omega),z),d(x_i(\omega),a_h)\leq b_{\lambda/2}+r_0$ for any $i\in\mathbb{N}$. Thus, we have 
\[
\vert G(d(x_i(\omega),z))-G(d(x_i(\omega),a_h))\vert\leq \frac{1}{3(k+1)}
\]
for any such $i$ and, using triangle inequality, this in particular yields that for any $i,j\in\mathbb{N}$:
\[
\vert G(d(x_i(\omega),z))-G(d(x_j(\omega),z))\vert\leq \frac{2}{3(k+1)}+\vert G(d(x_i(\omega),a_h))-G(d(x_j(\omega),a_h))\vert.
\]
So, any $1/(k+1)$-fluctuation of $\left(G(d(x_i(\omega),z))\right)_{i=0}^N$ is also a $1/(k_1+1)=1/3(k+1)$-fluctuation of $\left(G(d(x_n(\omega),a_h))\right)_{i=0}^N$. Therefore we have 
\[
\fluc{N}{k}{G(d(x_n(\omega),z))}\leq \fluc{N}{k_1}{G(d(x_n(\omega),a_h))}<\psi_1(\lambda,k) 
\]
and so we have shown that
\begin{gather*}
\left\{\omega\mid \forall h\in [1;\gamma(\iota_{b_{\lambda/2}+r_0}(3k+2))]\left( \fluc{N}{k_1}{G(d(x_n(\omega),a_h))}<\psi_1(\lambda,k)\right)\right\}\\
\cap \{\omega\mid \forall n\in\mathbb{N}\left(d(x_n(\omega),o)\leq b_{\lambda/2}\right)\}\\
\subseteq\left\{\omega\mid \forall z\in Z_0\left(\fluc{N}{k}{G(d(x_n(\omega),z))}<\psi_1(\lambda,k )\right)\right\}
\end{gather*}
and the result follows immediately.
\end{proof}

We now introduce a new kind of `set-fluctuation' parameterised by a finite set $Z_0$: Let
\[
\flucinf{k}{G(d(x_n,Z_0))}
\]
denote the largest $p$ such that there exist indices $i_1<j_1\leq i_2<j_2\leq \ldots< i_p<j_p$ such that
\[
\forall l\leq p\ \exists z\in Z_0 \left(\vert G(d(x_{i_l},z))-G(d(x_{j_l},z)\vert\geq \frac{1}{k+1}\right)
\]
and $\fluc{N}{k}{G(d(x_n,Z_0))}$ for the same notion on $\left(G(d(x_n,Z_0))\right)_{n=0}^N$.\\

The next Lemma \ref{lem:pointtosetfluc} now converts the previous assertion from Lemma \ref{lem:unifLemma} to a bound on the `set-fluctuations' of the sequence $(G(x_n,Z_0))$ where $Z_0$ is a finite set in a suitably good approximation $AF_j$ to $F$.

\begin{lemma}\label{lem:pointtosetfluc}
Fix $N\in\NN$ and let $Z_0$ be a finite set of size $p$. Suppose $\psi$ is such that for all $\lambda>0$ and $k\in\NN$:
\[
\PP\left(\exists z\in Z_0\left(\fluc{N}{k}{G(d(x_n,z))}\geq \psi(\lambda,k)\right)\right)<\lambda.
\]
Then for all $\lambda>0$ and $k\in\NN$:
\[
\PP\left(\fluc{N}{k}{G(d(x_n,Z_0))}\geq p\cdot \psi(\lambda,k)\right)<\lambda.
\]
\end{lemma}

\begin{proof}
Let $\omega$ be such that $\fluc{N}{k}{G(d(x_n(\omega),z))}<\psi(\lambda,k)$ for all $z\in Z_0$. By definition, each set-$1/(k+1)$-fluctuation of $\left(G(d(x_n(\omega),Z_0))\right)_{n=0}^N$ represents a standard $1/(k+1)$-fluctuation of $\left(G(d(x_n(\omega),z))\right)_{n=0}^N$ for at least one $z\in Z_0$, and therefore, as by the above there are $<\psi(\lambda,k)$ of such fluctuations for every $z\in Z_0$, there can be only $<p\cdot \psi(\lambda,k)$ of the former set-fluctuations as $Z_0$ has size $p$. That is, we have shown that
\[
\left\{\omega\mid \forall z\in Z_0\left(\fluc{N}{k}{G(d(x_n(\omega),z))}<\psi(\lambda,k)\right)\right\}\subseteq \left\{\omega\mid\fluc{N}{k}{G(d(x_n(\omega),Z_0))}<p\cdot \psi(\lambda,k)\right\}
\]
from which the result follows.
\end{proof}

The following Lemma \ref{lem:fromSetFlucToUnifMeta} then converts this set-fluctuation bound from the previous Lemma \ref{lem:pointtosetfluc} to a corresponding rate of pointwise metastability, guaranteeing that for any $k\in\NN$ and $g:\NN\to \NN$ we can find an interval $[n;n+g(n)]$ and $z\in Z_0$ such that $G(d(x_i,z))$ is $1/(k+1)$-stable with high probability.

\begin{lemma}\label{lem:fromSetFlucToUnifMeta}
Let $Z_0$ be a finite set and fix $\lambda>0$ as well as $k\in\mathbb{N}$ and $g:\mathbb{N}\to\mathbb{N}$. Suppose that $\psi$ is such that
\[
\PP(\fluc{N^{\lambda,k}_g}{k}{G(d(x_n,Z_0))}\geq \psi(\lambda,k))<\lambda.
\]
Then also
\[
\PP\left(\forall n\leq\Psi(\lambda,k,g)\ \exists i,j\in [n;n+g(n)]\ \exists z\in Z_0\left(\vert G(d(x_i,z))-G(d(x_j,z))\vert > \frac{1}{k+1}\right)\right)<\lambda
\]
where we define $\Psi(\lambda,k,g):=N^{\lambda,k}_g:=\tilde{g}^{(\psi(\lambda,k))}(0)$ with $\tilde{g}(n):=n+g(n)$.
\end{lemma}
\begin{proof}
Let $\omega$ be such that
\[
\forall n\leq\Psi(\lambda,k,g)\ \exists i,j\in [n;n+g(n)]\ \exists z\in Z_0\left(\vert G(d(x_i(\omega),z))-G(d(x_j(\omega),z))\vert > \frac{1}{k+1}\right).
\]
Take $n_0:=0$ and $n_{l+1}:=\tilde{g}(n_l)$ for $l<\psi(\lambda,k)$. Since then $n_{l}<\Psi(\lambda,k,g)$ for $l<\psi(\lambda,k)$, for each such $l$ there exist pairs $i,j\in [n_l;n_{l+1}]$ together with an accompanying $z\in Z_0$ such that $\vert G(d(x_i(\omega),z))-G(d(x_j(\omega),z))\vert > 1/(k+1)$. As all these $\psi(\lambda,k)$-many pairs $i,j$ lie below $N^{\lambda,k}_g=\Psi(\lambda,k,g)$, we have shown $\fluc{N^{\lambda,k}_g}{k}{G(d(x_n(\omega),Z_0))}\geq\psi(\lambda,k)$, and the result follows immediately.
\end{proof}

The next result, Lemma \ref{lem:fromPsitoConv}, is now the last one in our line of preparatory results for the main theorem in which we apply a quantitative compactness argument along with the $\liminf$-property and continuity of $G$ to transform rates of metastability for $G(d(x_i,z))$ in the above sense to a rate of metastability for $(x_i)$ in the usual sense. Here, we further rely on a modulus of inverse continuity for $G$ at $0$, i.e.\ a function $\nu:\mathbb{N}\to\mathbb{N}$ such that
\[
\forall k\in\mathbb{N}\ \forall x\geq 0\left( G(x) \leq\frac{1}{\nu(k)+1}\to x\leq\frac{1}{k+1}\right).
\]

\begin{lemma}\label{lem:fromPsitoConv}
Let $G$ be uniformly continuous on bounded sets with a modulus $\iota$. Also, let $G$ be inverse continuous at $0$ with a modulus $\nu$. Let $\gamma$ be a modulus of total boundedness and $r_0$ be a bound relative to $o$ for $X_0$ and let $(x_n)$ be a sequence of $X$-valued random variables. Assume that $b:(0,\infty)\to\mathbb{N}$ is a modulus of uniform almost sure finiteness for $(d(x_n,o))$. Assume that there are $\Psi,Z$ such that for all $\lambda>0$ and all $k\in\mathbb{N}$ and $g\in\mathbb{N}^\mathbb{N}$ as well as $p\in\mathbb{N}$: 
\begin{gather*}
\forall Z_0\subseteq AF_{Z(\lambda,k,g,p)}\text{ finite of size $p$}\tag{I}\label{eq:prem1}\\
\bigg(\PP\bigg( \forall m\leq\Psi(\lambda,k,g,p)\ \exists i,j\in [m;m+g(m)]\ \exists z\in Z_0 \\
\left(\vert G(d(x_i,z))-G(d(x_j,z))\vert>\frac{1}{k+1}\right)\bigg)<\lambda\bigg).
\end{gather*}
Assume further $Z(\lambda,k,g,p)\geq k$. Finally, let $\Phi$ be such that for all $\lambda>0$:
\[
\PP\left(\exists k,N\in\mathbb{N}\ \forall n\in [N;\Phi(\lambda,k,N)]\left(x_n\not\in AF_k\right)\right)<\lambda.\tag{II}\label{eq:prem2}
\]
Then $(x_n)$ is almost surely Cauchy and furthermore, for all $\lambda>0$ and all $k\in\mathbb{N}$ and $g\in\mathbb{N}^\mathbb{N}$:
\[
\PP\left(\forall n\leq \Delta(\lambda,k,g)\ \exists i,j\in [n;n+g(n)]\left(d(x_i,x_j)>\frac{1}{k+1}\right)\right)<\lambda,\tag{III}\label{eq:conc}
\]
where $\Delta$ is defined by
\[
\Delta(\lambda,k,g):=\max\{\Delta_i(\lambda,k,g)\mid i\in [0;\gamma(\iota_{2b_{\lambda/3}}(2\nu(2k+1)+1))]\},
\]
with
\[
\Delta_i(\lambda,k,g):=\max\{\Phi(\lambda/3,k^\lambda_i,m)\mid m\leq\Psi(\widehat{\lambda}_k,k^\lambda_i,h^{\lambda/3}_{k^\lambda_i},p^\lambda_k)\},
\]
with 
\[
h^\mu_l(m):=\max\{m'+g(m')\mid m'\in [m;\Phi(\mu,l,m)]\}-m,
\]
and 
\[
\begin{cases}
k^\lambda_0:=\max\{6\nu(2k+1)+5,k\},\\
k^\lambda_{j+1}:=\max\{Z(\widehat{\lambda}_k,k^\lambda_i,h^{\lambda/3}_{k^\lambda_i},p^\lambda_k)\mid i\leq j\},
\end{cases}
\]
as well as $p^\lambda_k:=\gamma(\iota_{\max\{2b_{\lambda/3},b_{\lambda/3}+r_0\}}(6\nu(2k+1)+5))$ and
\[
\widehat{\lambda}_k:=\frac{\lambda}{3(\gamma(\iota_{2b_{\lambda/3}}(2\nu(2k+1)+1))+1)}.
\]
\end{lemma}

\begin{proof}
Given $\lambda,\mu>0$ and $k,j\in\mathbb{N}$ as well as $g:\mathbb{N}\to\mathbb{N}$, we denote by 
\[
A^{\lambda,k,g}_{\mu,j}=\{a^{\lambda,k,g}_{\mu,j,r}\mid r\in [1;\gamma(\iota_{\max\{2b_{\mu/3},b_{\mu/3}+r_0\}}(6\nu(2j+1)+5))]\}
\]
the $1/(\iota_{\max\{2b_{\mu/3},b_{\mu/3}+r_0\}}(6\nu(2j+1)+5)+1)$-net of $AF_{Z(\lambda,k,g,p^\mu_j)}$ with size
\[
p^\mu_j:=\gamma(\iota_{\max\{2b_{\mu/3},b_{\mu/3}+r_0\}}(6\nu(2j+1)+5))
\]
which exists by \cite[Proposition 2.4, (ii)]{KohlenbachLeusteanNicolae2018} as $X_0$, and hence also $AF_{Z(\lambda,k,g,p^\mu_j)}$, is totally bounded with a modulus of total boundedness $\gamma$. In other words, the $a^{\lambda,k,g}_{\mu,j,h}$ have the property that
\[
\forall z\in AF_{Z(\lambda,k,g,p^\mu_j)}\ \exists r\in [1;p^\mu_j]\left(d(z,a^{\lambda,k,g}_{\mu,j,r})\leq\frac{1}{\iota_{\max\{2b_{\mu/3},b_{\mu/3}+r_0\}}(6\nu(2j+1)+5)+1}\right).
\]
Then, by assumption \eqref{eq:prem1}, we get for any $\lambda,\mu>0$ and $k,j\in\mathbb{N}$ as well as $g:\mathbb{N}\to\mathbb{N}$:
\begin{align*}
\PP\bigg(&\forall m\leq\Psi(\lambda,k,g,p^\mu_j)\ \exists i,j\in [m;m+g(m)]\tag{I'}\label{eq:prem1prime}\\
&\exists r\in [1;p^\mu_j]\left(\left\vert G(d(x_i,a^{\lambda,k,g}_{\mu,j,r}))-G(d(x_j,a^{\lambda,k,g}_{\mu,j,r}))\right\vert>\frac{1}{k+1}\right)\bigg)<\lambda.
\end{align*}
Denote the set measured in \eqref{eq:prem1prime} by $P_{\lambda,k,g,\mu,j}$ and denote by $Q_\lambda$ the set measured in \eqref{eq:prem2} and lastly by $R_{\lambda,k,g}$ the set measured in \eqref{eq:conc}. Denote by $S_\lambda$ the set measured in the assumption of uniform almost sure finiteness. Fix $\lambda>0$ and $k\in\mathbb{N}$ as well as $g:\mathbb{N}\to\mathbb{N}$, and take $\omega$ such that
\[
\omega\in S^c_{\lambda/3}\cap Q^c_{\lambda/3}\cap \bigcap_{l=0}^{\gamma(\iota_{2b_{\lambda/3}}(2\nu(2k+1)+1))}P^c_{\widehat{\lambda}_k,k^\lambda_l,h^{\lambda/3}_{k^\lambda_l},\lambda,k}.
\]
Then, by the last part of the intersection, we have
\begin{align*}
\exists m_l\leq\Psi(\widehat{\lambda}_k,k^\lambda_l,h^{\lambda/3}_{k^\lambda_l},p^\lambda_k)\ \forall i,j\in [m_l;m_l+h^{\lambda/3}_{k^\lambda_l}(m_l)]\ \forall r\in [1;p^\lambda_k]\tag*{$(+)_l$}\label{eq:prem1prime2}\\
\left(\left\vert G(d(x_i(\omega),a^l_{r}))-G(d(x_j(\omega),a^l_{r}))\right\vert\leq\frac{1}{6(\nu(2k+1)+1)}\right)
\end{align*}
for any $l\in [0;\gamma(\iota_{2b_{\lambda/3}}(2\nu(2k+1)+1))]$, where we write
\[
a^l_{r}=a^{\widehat{\lambda}_k,k^\lambda_l,h^{\lambda/3}_{k^\lambda_l}}_{\lambda,k,r}.
\]
and the inner formula follows from the assumed lower bound on $Z$, specifically that for $l>0$ we have
\[
k_{l}^\lambda\geq Z(\widehat{\lambda}_k,k^\lambda_0,h^{\lambda/3}_{k^\lambda_0},p^\lambda_k)\geq k^\lambda_0\geq 6\nu(2k+1)+5,
\]
which implies $k^\lambda_l+1\geq 6(\nu(2k+1)+1)$ for all such $l$. Further, by the second part of the intersection, we have 
\[
\forall k',N'\in\mathbb{N}\ \exists n'\in [N';\Phi(\lambda/3,k',N')]\left( x_{n'}(\omega)\in AF_{k'}\right).
\]
Lastly, by the first part of the intersection, we have $d(x_i(\omega),o)\leq b_{\lambda/3}$ for any $i\in\mathbb{N}$. Now, taking for each $l\in [0;\gamma(\iota_{2b_{\lambda/3}}(2\nu(2k+1)+1))]$ an $m_l$ satisfying \ref{eq:prem1prime2}, we get
\[
\exists n_l\in [m_l;\Phi(\lambda/3,k^\lambda_l,m_l)]\left( x_{n_l}(\omega)\in AF_{k^\lambda_l}\right).
\]
By definition of $h^{\lambda/3}_{k^\lambda_l}$, we have $[n_l;n_l+g(n_l)]\subseteq [m_l;m_l+h^{\lambda/3}_{k^\lambda_l}(m_l)]$ and so we in particular have
\begin{gather*}
\exists n_l\leq\Delta_l(\lambda,k,g)\bigg( x_{n_l}(\omega)\in AF_{k^\lambda_l}\land \forall i,j\in [n_l;n_l+g(n_l)]\ \forall r\in [1;p_k]\\
\left(\left\vert G(d(x_i(\omega),a^l_{r}))-G(d(x_j(\omega),a^l_{r}))\right\vert\leq\frac{1}{6(\nu(2k+1)+1)}\right)\bigg).
\end{gather*}
for all $l\in [0;\gamma(\iota_{2b_{\lambda/3}}(2\nu(2k+1)+1))]$. Note that $x_{n_l}(\omega)\in X_0$ for any $l\in [0;\gamma(\iota_{2b_{\lambda/3}}(2\nu(2k+1)+1))]$ so that, by using the modulus of total boundedness $\gamma$ of $X_0$ on the sequence $(x_{n_l}(\omega))$, we have
\[
\exists 0\leq I<J\leq\gamma(\iota_{2b_{\lambda/3}}(2\nu(2k+1)+1))\left( d(x_{n_I}(\omega),x_{n_J}(\omega))\leq\frac{1}{\iota_{2b_{\lambda/3}}(2\nu(2k+1)+1)+1}\right).
\]
For that pair, we thus have $G(d(x_{n_I}(\omega),x_{n_J}(\omega)))\leq \frac{1}{2(\nu(2k+1)+1)}$, and in particular, by the above, we have
\[
x_{n_J}(\omega)\in AF_{k^\lambda_J}\subseteq AF_{Z(\widehat{\lambda}_k,k^\lambda_I,h^{\lambda/3}_{k^\lambda_I},p^\lambda_k)}
\]
as well as
\[
\forall r\in [1;p^\lambda_k]\ \forall i,j\in [n_I;n_I+g(n_I)]\left(\left\vert G(d(x_i(\omega),a^I_{r}))-G(d(x_j(\omega),a^I_{r}))\right\vert\leq\frac{1}{6(\nu(2k+1)+1)}\right).
\]
Therefore, by the properties of the $a^I_{r}$'s, we have that there exists an $r\in [1;p^\lambda_k]$ such that
\[
\vert d(x_{n_J}(\omega),x_i(\omega))-d(a^I_{r},x_i(\omega))\vert\leq d(x_{n_J}(\omega),a^I_{r})\leq\frac{1}{\iota_{\max\{2b_{\lambda/3},b_{\lambda/3}+r_0\}}(6\nu(2k+1)+5)+1}
\]
and so
\[
\vert G(d(x_{n_J}(\omega),x_i(\omega)))-G(d(a^I_{r},x_i(\omega)))\vert\leq\frac{1}{6(\nu(2k+1)+1)}
\]
for any $i\in [n_I;n_I+g(n_I)]$. Combined, we have for all $i,j\in [n_I;n_I+g(n_I)]$:
\[
\vert G(d(x_i(\omega),x_{n_J}(\omega)))-G(d(x_j(\omega),x_{n_J}(\omega)))\vert\leq\frac{1}{2(\nu(2k+1)+1)}.
\]
This yields $G(d(x_i(\omega),x_{n_J}(\omega)))\leq\frac{1}{\nu(2k+1)+1}$ and so in particular $d(x_i(\omega),x_{n_J}(\omega))\leq\frac{1}{2(k+1)}$ for all $i\in [n_I;n_I+g(n_I)]$ from which we can deduce that
\[
\forall i,j\in [n_I;n_I+g(n_I)]\left(d(x_i(\omega),x_{j}(\omega))\leq\frac{1}{k+1}\right)
\]
and so, as $n_I\leq \Delta(\lambda,k,g)$, we get $\omega\in R^c_{\lambda,k,g}$. In summary, since $\omega$ was arbitrary, we have shown
\[
S^c_{\lambda/3}\cap Q^c_{\lambda/3}\cap \bigcap_{l=0}^{\gamma(\iota_{2b_{\lambda/3}}(2\nu(2k+1)+1))}P^c_{\widehat{\lambda}_k,k^\lambda_l,h^{\lambda/3}_{k^\lambda_l},\lambda,k}\subseteq R^c_{\lambda,k,g}
\]
and so
\begin{align*}
\PP(R_{\lambda,k,g})&\leq \PP(S_{\lambda/3})+\PP(Q_{\lambda/3})+\sum_{l=0}^{\gamma(\iota_{2b_{\lambda/3}}(2\nu(2k+1)+1))}\PP\left(P_{\widehat{\lambda}_k,k^\lambda_l,h^{\lambda/3}_{k^\lambda_l},\lambda,k}\right)\\
&< \lambda/3+\lambda/3+(\gamma(\iota_{2b_{\lambda/3}}(2\nu(2k+1)+1))+1)\widehat{\lambda}_k=\lambda,
\end{align*}
the last step following by definition of $\widehat{\lambda}_k$.
\end{proof}

\begin{remark}\label{rem:Opial}
Lemmas \ref{lem:unifLemma} -- \ref{lem:fromPsitoConv}, if considered deterministically, provide a kind of quantitative version of Opial's lemma (see e.g.\ Fact 1.1 in the recent \cite{ArakcheevBauschke2025a} and also the surrounding discussion therein) on totally bounded metric spaces in the case where we have bounds on the fluctuations for the key assumption of Opial's lemma that $d(x_k,z)$ converges for all solutions $z$ in a suitably uniform way. This illustrates the considerably combinatorial complexity inherent in Opial's lemma (already in such special cases). In particular, the fact that we have to resolve this fragment of Opial's lemma in this way is a key reason for the complexity of our final bound and a key difference to the previous quantitative results for deterministic quasi-Fej\'er monotone sequences given in \cite{KohlenbachLeusteanNicolae2018}. 
\end{remark}

We can now derive our main quantitative result. For that, we just recall the setup: We are working over a probability space $(\Omega,\mathsf{F},\PP)$ and an arbitrary metric space $X$, together with a problem set $F=\bigcap_{k\in\mathbb{N}}AF_k$ for a sequence $(AF_k)$ of measurable sets with $AF_{k+1}\subseteq AF_k\subseteq X_0$ for all $k\in\mathbb{N}$, all included in a set $X_0\subseteq X$ which is totally bounded.

\begin{theorem}\label{thm:mainQuant}
Let $G$ be uniformly continuous on bounded sets with a modulus $\iota:(0,\infty)\times\mathbb{N}\to \mathbb{N}$. Also, let $G$ be inverse continuous at $0$ with a modulus $\nu:\mathbb{N}\to\mathbb{N}$. Let $\gamma$ be a modulus of total boundedness and $r_0$ be a bound relative to $o$ for $X_0$.

Let $(x_n)$ be a sequence of $X$-valued random variables adapted to $\mathcal{F}=(\mathcal{F}_n)$ and $(\eta_n),(\chi_n)\in \ell^1_+(\mathcal{F})$. Assume that we have moduli of almost sure finiteness $f$ and $h$ for $\sum_{i=0}^\infty\eta_i<+\infty$ and $\prod_{i=0}^\infty (1+\chi_i)<+\infty$, respectively. Further assume that we have a modulus of uniform stochastic quasi-Fej\'er monotonicity $\zeta(\lambda,r,n)$ for $(x_n)$ w.r.t.\ $F$, $(AF_k)$ and $\mathcal{F}$ and that $(x_n)$ has the stochastic $\liminf$-property w.r.t.\ $F$ and $(AF_k)$ with a corresponding modulus $\Phi(\lambda,k,N)$. Assume that $b:(0,\infty)\to\mathbb{N}$ is a modulus of uniform almost sure finiteness for $(d(x_n,o))$. Lastly, assume that $b_0$ is such that $b_0> \EE[G(d(x_0,o)+r_0)]$ and that $\mu_n$ is a modulus of absolute continuity for $G(d(x_n,z))$, uniformly in $z\in X_0$.

Then $(x_n)$ is almost surely Cauchy and further, for all $\lambda>0$ and all $k\in\mathbb{N}$ and $g\in\mathbb{N}^\mathbb{N}$:
\[
\PP\left(\forall n\leq \Delta(\lambda,k,g)\ \exists i,j\in [n;n+g(n)]\left(d(x_i,x_j)>\frac{1}{k+1}\right)\right)<\lambda,
\]
where $\Delta$ is defined by
\[
\Delta(\lambda,k,g):=\max\{\Delta_i(\lambda,k,g)\mid i\in [0;\gamma(\iota_{2b_{\lambda/3}}(2\nu(2k+1)+1))]\},
\]
with
\[
\Delta_i(\lambda,k,g):=\max\{\Phi'(\lambda/3,k^\lambda_i,m)\mid m\leq\Psi(\widehat{\lambda}_k,k^\lambda_i,h^{\lambda/3}_{k^\lambda_i},p^\lambda_k)\},
\]
with 
\[
h^\mu_l(m):=\max\{m'+g(m')\mid m'\in [m;\Phi'(\mu,l,m)]\}-m,
\]
and 
\[
\begin{cases}
k^\lambda_0:=\max\{6\nu(2k+1)+5,k\},\\
k^\lambda_{j+1}:=\max\{\tilde Z(\widehat{\lambda}_k,k^\lambda_i,h^{\lambda/3}_{k^\lambda_i},p^\lambda_k),k^\lambda_i\mid i\leq j\},
\end{cases}
\]
as well as $p^\lambda_k:=\gamma(\iota_{\max\{2b_{\lambda/3},b_{\lambda/3}+r_0\}}(6\nu(2k+1)+5))$ and
\[
\widehat{\lambda}_k:=\frac{\lambda}{3(\gamma(\iota_{2b_{\lambda/3}}(2\nu(2k+1)+1))+1)},
\]
where now $\Psi$ and $\tilde Z$ are defined by
\begin{align*}
\tilde Z(\lambda,k,g,q):=\zeta\Bigg(&p_{f,h,b_0}\left(\frac{\lambda}{2\gamma(\iota_{b_{\lambda/2}+r_0}(3k+2))},3k+2,N^{\lambda,k}_{g,q}\right),\\
&e_{f,h,b_0}\left(\frac{\lambda}{2\gamma(\iota_{b_{\lambda/2}+r_0}(3k+2))},3k+2,N^{\lambda,k}_{g,q}\right),N^{\lambda,k}_{g,q}\Bigg)
\end{align*}
and
\[
\Psi(\lambda,k,g,q):=N^{\lambda,k}_{g,q}:=\tilde{g}^{\left(q\cdot Z_{f,h,b_0}\left(\frac{\lambda}{2\gamma(\iota_{b_{\lambda/2}+r_0}(3k+2))},3k+2\right)\right)}(0)
\]
for $\tilde{g}(n):=n+g(n)$ and with $Z$, $e$ and $p$ from Theorem \ref{thm:finRS} (also defined in terms of $\mu_n$), and $\Phi'$ is defined by
\[
\Phi'(\lambda,k,N):=\Phi\left(\lambda 2^{-(k+N+2)},k,N\right).
\]
\end{theorem}
\begin{proof}
First, combining Lemmas \ref{lem:fromRStoPsi} and \ref{lem:unifLemma}, we have that for any $\lambda>0$ and $k,N\in\NN$, for any finite set $Z_0\subseteq AF_{\chi_1(\lambda,k,N)}$
\[
\PP\left(\exists z\in Z_0\left(J_{N,k}\left(G(d(x_n,z))\right)\geq \psi_1(\lambda,k)\right)\right)<\lambda
\]
for
\begin{align*}
\chi_1(\lambda,k,N):=\zeta\Bigg(&p_{f,h,b_0}\left(\frac{\lambda}{2\gamma(\iota_{b_{\lambda/2}+r_0}(3k+2))},3k+2,N\right),\\
&e_{f,h,b_0}\left(\frac{\lambda}{2\gamma(\iota_{b_{\lambda/2}+r_0}(3k+2))},3k+2,N\right),N\Bigg)
\end{align*}
and
\[
\psi_1(\lambda,k):=Z_{f,h,b_0}\left(\frac{\lambda}{2\gamma(\iota_{b_{\lambda/2}+r_0}(3k+2))},3k+2\right).
\]
Thus by Lemma \ref{lem:pointtosetfluc}, whenever $Z_0$ has size $q$ then
\[
\PP\left(J_{N,k}\left(G(d(x_n,Z_0))\right)\geq q\cdot\psi_1(\lambda,k)\right)<\lambda,
\]
and so by Lemma \ref{lem:fromSetFlucToUnifMeta} we derive that $\Psi$ and $\tilde Z$ defined as above satisfy for all $\lambda>0$ and all $k\in\mathbb{N}$ and $g\in\mathbb{N}^\mathbb{N}$ as well as $q\in\mathbb{N}$: 
\begin{gather*}
\forall Z_0\subseteq AF_{\max\{\tilde Z(\lambda,k,g,q),k\}}\text{ finite of size $q$}\\
\bigg(\PP\bigg( \forall m\leq\Psi(\lambda,k,g,q)\ \exists i,j\in [m;m+g(m)]\ \exists z\in Z_0 \\
\left(\vert G(d(x_i,z))-G(d(x_j,z))\vert>\frac{1}{k+1}\right)\bigg)<\lambda\bigg).
\end{gather*}
Briefly observing, using Lemma \ref{lem:stochasticLiminfModStrong}, that $\Phi'$ satisfies 
\[
\forall\lambda>0\ \left(\PP\left(\exists k,N\in\NN\ \forall n\in [N;\Phi'(\lambda,k,N)](x_n\not\in AF_k)\right)<\lambda\right),
\]
the result then follows from Lemma \ref{lem:fromPsitoConv}.
\end{proof}

\begin{remark}\label{rem:mainQuantStrong}
Akin to Remarks \ref{rem:optStopStrong}, \ref{rem:martFlucStrong} and \ref{rem:RSStrong}, if $(x_n)$ is already uniformly quasi-Fej\'er monotone in the stronger sense that the associated probability is not resolved, then the above Theorem \ref{thm:mainQuant} can be simplified in a way that the resulting quantitative information no longer depends on associated moduli of absolute continuity. The proof then in particular relies on the more uniform variant of Theorem \ref{thm:finRS} that is discussed in Remark \ref{rem:RSStrong}.
\end{remark}

Forgetting about the quantitative data, we obtain the following ``plain'' almost sure convergence result as an immediate corollary of our main theorem above:

\begin{theorem}\label{thm:pre_final}
Let $(\Omega,\mathsf{F},\PP)$ be a probability space and let $X$ be a metric space. Further, let $X_0\subseteq X$ be totally bounded and let $r_0$ be a bound relative to $o$ for $X_0$. Let $F=\bigcap_{k\in\mathbb{N}}AF_k$ for a sequence $(AF_k)$ of measurable sets with $AF_{k+1}\subseteq AF_k\subseteq X_0$ for all $k\in\mathbb{N}$. Let $G$ be increasing and continuous with $G(0)=0$. Let $(x_n)$ be a sequence of $X$-valued random variables which is uniformly stochastically quasi-Fej\'er monotone w.r.t.\ $F$, $(AF_k)$ and $\mathcal{F}$ and which has the stochastic $\liminf$-property w.r.t.\ $F$ and $(AF_k)$. Assume that $(d(x_n,o))$ is uniformly almost surely finite and that $G(d(x_0,o)+r_0)$ has finite mean. Then $(x_n)$ is almost surely Cauchy.
\end{theorem}

If we now assume that $X$ is complete and separable, then we can further guarantee that $(x_n)$ converges almost surely to a limit. To guarantee the feasibility of the solution in our abstract setup, we further require a suitable closedness property for $F$ relative to the $(AF_k)$. For that, we further follow the approach of \cite{KohlenbachLeusteanNicolae2018} and consider the abstract property that $F$ is explicitly closed w.r.t.\ $(AF_k)$, in the sense that
\[
\forall p\in X\left( \forall N,M\in\mathbb{N}\left(AF_M\cap \overline{B}_{1/(N+1)}(p)\neq\emptyset\right)\to p\in F\right).
\]
Under this assumption, we can then guarantee that the above limit is almost surely a solution, i.e.\ that it belongs to $F$, and hence derive the following extended result. In it, we also further reduce the property of uniform stochastic quasi-Fej\'er monotonicity to its ordinary counterpart, and following the discussion before Remark \ref{rem:uniformfinite} replace the assumption that $(d(x_n,o))$ is almost surely finite by the assumptions that $F\neq \emptyset$ and that $G$ is coercive.

\begin{theorem}\label{thm:final}
Let $(\Omega,\mathsf{F},\PP)$ be a probability space and let $X$ be a complete and separable metric space. Further, let $X_0\subseteq X$ be compact and let $r_0$ be a bound relative to $o$ for $X_0$. Let $\emptyset\neq F=\bigcap_{k\in\mathbb{N}}AF_k$ for a sequence $(AF_k)$ with $AF_{k+1}\subseteq AF_k\subseteq X_0$ for all $k\in\mathbb{N}$, such that $F$ is explicitly closed w.r.t.\ $(AF_k)$. Let $G$ be increasing, continuous and coercive with $G(0)=0$. Let $(x_n)$ be a sequence of $X$-valued random variables which is stochastically quasi-Fej\'er monotone w.r.t.\ $F$, $(AF_k)$ and a filtration $\mathcal{F}$ of $\mathsf{F}$, and which has the stochastic $\liminf$-property w.r.t.\ $F$ and $(AF_k)$. Assume that $G(d(x_0,o)+r_0)$ has finite mean. Then $x_n \to x$ almost surely, where $x\in F$ almost surely.
\end{theorem}
\begin{proof}
We first show that, in the context of the present assumptions of compactness and explicit closedness, the sequence $(x_n)$ is in fact already uniformly stochastically quasi-Fej\'er monotone. Suppose for a contradiction that this is not the case, i.e.\ that there are $\lambda>0$ and $r,n\in\mathbb{N}$ such that for any $k\in\mathbb{N}$, there is a $z_k\in AF_k$ and an $m_k\leq n$ such that
\[
\PP\left(\EE[G(d(x_{m_k+1},z_k))\mid \mathcal{F}_{m_k}]> (1+\chi_{m_k})G(d(x_{m_k},z_k))+\eta_{m_k}+\frac{1}{r+1}\right)\geq \lambda.
\]
Using the compactness of $X_0$, $(z_k)$ has a convergent subsequence. W.l.o.g., we simply assume that $z_k\to z$, and by the infinite pigeonhole principle, we can further assume w.l.o.g.\ that $m_k=m$ is constant. Note first that since $z_k\in AF_k$ for all $k$, and since $z_k\to z$, we have $z\in F$ as $F$ is explicitly closed. Note now that
\[
\vert d(x_{m+1},z)-d(x_{m+1},z_k)\vert,\vert d(x_{m},z)-d(x_{m},z_k)\vert\leq d(z,z_k)
\]
so that $d(x_{m+1},z_k)\to d(x_{m+1},z)$ and $d(x_{m},z_k)\to d(x_{m},z)$, uniformly. In particular, as $G$ is continuous, let $k_j$ be such that
\[
\vert G(d(x_{m+1},z))-G(d(x_{m+1},z_{k_j}))\vert,\vert G(d(x_{m},z))-G(d(x_{m},z_{k_j}))\vert\leq \frac{1}{j+1}
\]
for all $j\in\mathbb{N}$, uniformly on $\Omega$. In particular, we have
\[
G(d(x_{m+1},z_{k_j}))\leq \frac{1}{j+1}+G(d(x_{m+1},z))\text{ and }G(d(x_{m},z))\leq \frac{1}{j+1}+G(d(x_{m},z_{k_j})).
\]
This thereby yields that if $\omega$ is such that
\[
\EE[G(d(x_{m+1},z_{k_j}))\mid \mathcal{F}_{m}](\omega)> (1+\chi_{m}(\omega))G(d(x_{m}(\omega),z_{k_j}))+\eta_{m}(\omega)+\frac{1}{r+1},
\]
then we also have
\[
-\EE[G(d(x_{m+1},z))\mid \mathcal{F}_{m}](\omega)+(1+\chi_{m}(\omega))G(d(x_{m}(\omega),z))+\eta_{m}(\omega)+\frac{1}{r+1}<\frac{2+\chi_{m}(\omega)}{j+1}.
\]
Denoting the left random variable $X_m$, we hence have
\begin{align*}
&\PP\left(X_m<\frac{2+\chi_m}{j+1}\right)\\
&\qquad\geq \PP\left(\EE[G(d(x_{m+1},z_{k_j}))\mid \mathcal{F}_{m}]> (1+\chi_{m})G(d(x_{m},z_{k_j}))+\eta_{m}+\frac{1}{r+1}\right)\geq \lambda.
\end{align*}
By assumption, the sum $\sum_{n=0}^\infty\chi_n$ is almost surely finite, and so also $\chi_{m}$ is almost surely finite. So, let $b_{\lambda}$ be an associated modulus of almost sure finiteness, that is $\PP(\chi_m>b_\lambda)<\lambda$ for all $\lambda>0$. Hence, using the Fr\'echet inequality, we have
\[
\PP\left(X_m<\frac{2+b_{\lambda/2}}{j+1}\right)\geq \PP\left(X_m<\frac{2+\chi_m}{j+1}\right)+\PP(\chi_m\leq b_{\lambda/2})-1\geq \lambda-\lambda/2=\lambda/2.
\]
However, we have
\[
\lim_{j\to 0}\PP\left(X_m<\frac{2+b_{\lambda/2}}{j+1}\right)=\PP(X_m\leq 0)
\]
as $\PP(X_m<\varepsilon)$ is right-continuous in $\varepsilon$. Now, however then
\[
\PP\left(\EE[G(d(x_{m+1},z))\mid \mathcal{F}_{m}]\geq (1+\chi_{m})G(d(x_{m},z))+\eta_{m}+\frac{1}{r+1}\right)=\PP(X_m\leq 0)\geq\lambda/2
\]
and so the stochastic quasi-Fej\'er monotonicity is violated on a set with positive probability. Hence $(x_n)$ is not stochastically quasi-Fej\'er monotone w.r.t.\ $F$, $(AF_k)$ and $\mathcal{F}$, a contradiction.

By Theorem \ref{thm:pre_final}, we thus get that $(x_n)$ is almost surely Cauchy. As $X$ is complete and separable, we have that $(x_n)$ almost surely converges to some random variable $x$. To show that $x\in F$ almost surely, note that $(x_n)$ has the stochastic $\liminf$-property w.r.t.\ $F$ and $(AF_k)$. Fix an $\omega$ such that $(x_n(\omega))$ satisfies 
\[
\forall k,N\in\mathbb{N} \exists n\geq N (x_n(\omega)\in AF_k)
\]
and simultaneously $x_n(\omega)\to x(\omega)$. Using explicit closedness gives $x(\omega)\in F$. As the set of all such $\omega$ has measure $1$, we get the result.
\end{proof}

\begin{remark}
The above Theorem \ref{thm:final} can in fact be proved rather directly without appealing to Theorem \ref{thm:mainQuant} by following the folklore approach towards convergence of (stochastic) quasi-Fej\'er monotone sequences (see e.g.\ \cite{CombettesPesquet2015} as well as \cite{Pischke2026}), suitably adapted to (compact) metric spaces. Indeed, if $(x_n)$ is stochastically quasi-Fej\'er monotone w.r.t.\ $F$, $(AF_k)$ and $\mathcal{F}$, then by the Robbins-Siegmund theorem, we have that $G(d(x_n,z))$ converges $\PP$-a.s.\ for all $z\in F$, say on a set $\Omega_z$ with measure $1$. By compactness, take $(z_n)$ dense in $F$ and denote by $\Omega'$ the set of measure $1$ on which the stochastic $\liminf$-property holds for $(x_n)$ w.r.t.\ $F$ and $(AF_k)$. Let $\omega\in \widehat{\Omega}:=\Omega'\cap\bigcap_{i\in\mathbb{N}}\Omega_{z_i}$. Then by $\omega\in\Omega'$, we have that there exists a subsequence $x_{n_k}$ such that $x_{n_k}(\omega)\in AF_k$ for all $k\in\mathbb{N}$. By compactness, let $x^*(\omega)$ be a limit point of that sequence. Note that by explicit closedness of $F$, we have $x^*(\omega)\in F$. As $(z_i)$ is dense in $F$, for any $\varepsilon>0$ there exists an $i\in\mathbb{N}$ with $d(x^*(\omega),z_i)<\varepsilon$. As $\omega\in\Omega_{z_i}$, we have that $G(d(x_n(\omega),z_i))$ converges and by coercivity, $d(x_n(\omega),z_i)$ converges. In particular, $d(x^*(\omega),z_i))$ is a cluster point of the latter, so that given $\varepsilon$, there is some $i\in\mathbb{N}$ with $\lim_{n\to\infty}d(x_n(\omega),z_i)<\varepsilon$. Thus, we have
\[
\limsup_{n\to\infty} d(x_n(\omega),x^*(\omega))\leq \lim_{n\to\infty}d(x_n(\omega),z_i)+\varepsilon\leq 2\varepsilon.
\]
Hence, as $\varepsilon$ was arbitrary, we have $\lim_{n\to\infty} d(x_n(\omega),x^*(\omega))=0$.

The accomplishment of the above approach to Theorem \ref{thm:final} via Theorem \ref{thm:mainQuant} is thereby rather that it illustrates that Theorem \ref{thm:mainQuant} is a full finitization (in the sense of Tao, see \cite{Tao2008a}) of the result given in Theorem \ref{thm:final}.
\end{remark}

\begin{remark}\label{rem:unifClosed}
Similar to \cite{KohlenbachLeusteanNicolae2018} (see Theorem 5.3 therein), if the assumption that $F$ is explicitly closed is uniformly and quantitatively resolved via so-called moduli of uniform closedness, that is functions $\alpha_1,\alpha_2:\mathbb{N}\to\mathbb{N}$ such that 
\[
\forall k\in\mathbb{N}\ \forall z,z'\in X\left( z\in AF_{\alpha_1(k)}\cap \overline{B}_{1/(\alpha_2(k)+1)}(z')\to z'\in AF_k\right),
\]
then the previous quantitative convergence result given in Theorem \ref{thm:mainQuant} can be strengthened using these moduli to maintain that the process $(x_n)$ contains approximate solutions along the region of metastability. Concretely, write $\alpha(k):=\max\{\alpha_1(k),\alpha_2(k)\}$ and $\widehat{k}:=\max\{k,\alpha(k)\}$. Now, fix $\lambda>0$, $k\in\mathbb{N}$ and $g\in\mathbb{N}^\mathbb{N}$. Following the proof of Theorem \ref{thm:mainQuant} (and therefore that of Lemma \ref{lem:fromPsitoConv}) with $\widehat{k}$ instead of $k$, we get a set $A$ with $\PP(A^c)<\lambda$ such that, for any $\omega\in A$, there are numbers $n_I,n_J$ with $n_I\leq \Delta(\lambda,\widehat{k},g)$ and $x_{n_J}(\omega)\in AF_{{\widehat{k}}^\lambda_J}$
as well as $d(x_i(\omega),x_{n_J}(\omega))\leq\frac{1}{2({\widehat{k}}+1)}$ for all $i\in [n_I;n_I+g(n_I)]$. Using that $\widehat{k}\geq k$, we can conclude from this as before that 
\[
\forall i,j\in [n_I;n_I+g(n_I)]\left(d(x_i(\omega),x_{j}(\omega))\leq\frac{1}{k+1}\right).
\]
Further however, note that 
\[
x_{n_J}(\omega)\in AF_{{\widehat{k}}^\lambda_J}\subseteq AF_{{\widehat{k}}^\lambda_0}\subseteq AF_{\widehat{k}}\subseteq AF_{\alpha_1(k)}
\]
and so as $d(x_i(\omega),x_{n_J}(\omega))\leq 1/(\alpha_2(k)+1)$ for all $i\in [n_I;n_I+g(n_I)]$ since $\widehat{k}\geq \alpha_2(k)$, we get $x_i(\omega)\in AF_k$ for any $i\in [n_I;n_I+g(n_I)]$. Hence, we have shown that
\[
\PP\left(\forall n\leq \Delta(\lambda,\widehat{k},g)\ \exists i,j\in [n;n+g(n)]\left(d(x_i,x_j)>\frac{1}{k+1}\text{ or }x_i\not\in AF_k\right)\right)\leq\PP(A)<\lambda,
\]
so that $x_i\in AF_k$ is maintained along the same region of metastability with high probability.
\end{remark}

\section{Application: Common fixed point problems}
\label{sec:common}

We end by outlining how the above results, both in their quantitative but also in their qualitative form, can be used to provide new convergence results for stochastic approximation procedures. We consider an example which incorporates many of the new notions we have introduced (while considerably simplifying others), namely a stochastic common fixed point problem and an associated randomized Krasnoselskii-Mann scheme, as e.g.\ recently investigated (in a broad context) over Hilbert spaces by Combettes and Madariaga \cite{CombettesMadariaga2025}.

Crucially however, our results are developed over the broad nonlinear context of separable Hadamard spaces, instead of just over Hilbert spaces, and our presentation is quite different, being tailored to the specifics setting of this paper in which we work over locally compact spaces, as opposed to assuming other regularity notions for guaranteeing convergence.

We begin by outlining our setup. As mentioned above, we will develop our results in the context of Hadamard spaces, that is complete geodesic metric spaces of nonpositive curvature in the sense of Alexandrov. These spaces, which uniformly cover examples such as Hilbert spaces, Hadamard manifolds (i.e.\ complete simply connected Riemannian manifolds of nonpositive sectional curvature) as well as $\mathbb{R}$-trees and beyond, frequently appear as one of the more canonical nonlinear generalizations of Hilbert spaces. We here only discuss the essential properties of these spaces as they are relevant to the present paper, and beyond that refer to \cite{AlexanderKapovitchPetrunin2023,Bacak2014a,BridsonHaefliger1999} for further exposition, each of which have a different style and focus.

For this purpose of this paper, a metric space $(X,d)$ is called (uniquely) geodesic if all two points are joined by a (unique) geodesic. In the uniquely geodesic case, given $x,y\in X$ and $\lambda\in [0,1]$, we write $(1-\lambda)x\oplus \lambda y$ for the unique point $z$ on the geodesic connecting $x$ and $y$ satisfying $d(x,z)=\lambda d(x,y)$ and $d(y,z)=(1-\lambda)d(x,y)$. A uniquely geodesic metric space $(X,d)$ is a $\mathrm{CAT}(0)$ space if it satisfies
\begin{equation}
d^2((1-\lambda)x\oplus \lambda y,z)\leq (1-\lambda)d^2(x,z)+\lambda d^2(y,z)-\lambda(1-\lambda)d^2(x,y),
\label{CN}\tag{CN}
\end{equation}
that is an extension of the so-called Bruhat-Tits $\mathrm{CN}$-inequality to geodesics. A complete $\mathrm{CAT}(0)$ space is called a Hadamard space. We will here rely on another characterization of $\mathrm{CAT}(0)$ spaces due to Berg and Nikolaev \cite{BergNikolaev2008}, using their so-called quasi-inner product defined by
\[
\ql{xy}{uv}:=\frac{1}{2}\left( d^2(x,v) + d^2(y,u) - d^2(x,u)-d^2(y,v) \right)
\]
for all $x,y,u,v\in X$, where we write $\overrightarrow{xy},\overrightarrow{uv}$ as a shorthand for pairs $(x,y),(u,v)\in X^2$. As shown in \cite{BergNikolaev2008}, already in a general metric space $(X,d)$, this represents the unique function $X^2\times X^2\to \RR$ that satisfies
\begin{enumerate}[(1)]
	\item $\ql{xy}{xy}=d^2(x,y)$,
	\item $\ql{xy}{uv}=\ql{uv}{xy}$,
	\item $\ql{xy}{uv}=-\ql{yx}{uv}$,
	\item $\ql{xy}{uv}+\ql{xy}{vw}=\ql{xy}{uw}$,
\end{enumerate}
for all $x,y,u,v,w\in X$. It then follows from the results in \cite{BergNikolaev2008} that a geodesic metric space $(X,d)$ is a $\mathrm{CAT}(0)$ space if, and only if, the following metric version of the Cauchy-Schwartz inequality holds for all $x,y,u,v\in X$:
\begin{equation}
\ql{xy}{uv}\leq d(x,y)d(u,v).\tag{CS}\label{CS}
\end{equation}
Beyond the above, we in the following will only require a few other standard observations about $\mathrm{CAT}(0)$ spaces, which we collected in the following lemma.

\begin{lemma}
Let $(X,d)$ be a $\mathrm{CAT}(0)$ space. Then
\begin{equation}
d^2(x,z)=d^2(x,y)+d^2(y,z)+2\ql{xy}{yz}
\label{Q}\tag*{$(\ast)_1$}
\end{equation}
for all $x,y,z\in X$, as well as
\begin{equation}
d((1-\lambda)x\oplus \lambda y,z)\leq (1-\lambda)d(x,z)+\lambda d(y,z)
\label{T}\tag*{$(\ast)_2$}
\end{equation}
for all $x,y,z\in X$ and $\lambda\in [0,1]$.
\end{lemma}

The first of the above properties is the result of a direct calculation, the second immediately follows from the joint convexity of the metric in $\mathrm{CAT}(0)$ (see e.g.\ in Proposition 1.1.5 in \cite{Bacak2014a}). We hence omit the proofs.

\subsection{Common fixed points on separable Hadamard spaces}
\label{sec:common:spaces}

For the remainder of this section, we assume that $(X,d)$ is some proper Hadamard space (which is hence separable), and let $\mathcal{B}(X)$ denote the usual Borel $\sigma$-algebra on $X$. Beyond these, we assume an ambient probability space $(\Omega,\mathsf{F},\PP)$. Fixing some other measure space $(V,\mathcal{V})$, we let $(T_v)_{v\in V}$ be a family of nonexpansive mappings on $X$, i.e.
\[
d(T_vx,T_vy)\leq d(x,y)
\]
for all $x,y\in X$ and $v\in V$. We assume furthermore that the mapping $V\times X\to X$ given by $(v,x)\mapsto T_vx$ is $\mathcal{V}\otimes \mathcal{B}(X)/ \mathcal{B}(X)$-measurable. We now let $v:(\Omega,\mathsf{F},\PP)\to (V,\mathcal{V})$ be a $V$-valued random variable over our ambient probability space, and consider the following problem:
\[
\text{find some element of }F:=\{z\in X\mid T_vz=z\; \PP\text{-a.s.}\}.\tag{sCFP}\label{sCFP}
\]
Throughout, we assume that $F\neq\emptyset$. Crucially, we can express $F$ equivalently by
\[
F:=\{z\in X\mid \EE[d^2(T_vz,z)]=0\},
\]
which motivates a sequence of approximating sets $(AF_k)$ defined by
\[
AF_k:=\left\{z\in X_0\large\mid \EE[d^2(T_vz,z)]\leq \frac{1}{k+1}\right\},
\]
where we already incorporate a closed set $X_0\subseteq X$, which will later be concretely instantiated in the context of our method.

Indeed, it is then immediate that $(AF_k)$ approximate $F$ in the sense that $AF_{k+1}\subseteq AF_k$ for all $k\in\NN$ and that $\bigcap_{k\in\NN}AF_k=F$. Further, also the following result is rather immediate:

\begin{lemma}
For any $k\in\mathbb{N}$, $AF_k$ is closed. In particular, $F$ is closed.
\end{lemma}
\begin{proof}
Fix $k\in\mathbb{N}$, and let $(z_n)\subseteq AF_k$ with $z_n\to z$. Note that almost surely, we then have $d(T_vz,z)\leq d(T_vz_n,z_n)+2d(z_n,z)$ and so
\[
d^2(T_vz,z)\leq d^2(T_vz_n,z_n)+4d(z_n,z)d(T_vz_n,z_n)+4d^2(z_n,z)
\]
for all $n\in\mathbb{N}$. Now, note that $\EE[d^2(T_vz_n,z_n)]\leq \frac{1}{k+1}$ and so $\EE[d(T_vz_n,z_n)]\leq \sqrt{\frac{1}{k+1}}$. Applying expectations, we get
\begin{align*}
\EE[d^2(T_vz,z)]&\leq \EE[d^2(T_vz_n,z_n)]+4d(z_n,z)\EE[d(T_vz_n,z_n)]+4d^2(z_n,z)\\
&\leq \frac{1}{k+1}+4d(z_n,z)\sqrt{\frac{1}{k+1}}+4d^2(z_n,z).
\end{align*}
Taking limits yields $\EE[d^2(T_vz,z)]\leq \frac{1}{k+1}$ and so $z\in AF_k$.
\end{proof}

In particular, we thereby get that all $AF_k$ and also $F$ is measurable. Further, we also get that $F$ is explicitly closed w.r.t.\ $(AF_k)$ in the following particularly uniform way (recall Remark \ref{rem:unifClosed}):

\begin{lemma}\label{lem:FExpClosed}
Fix $k\in\mathbb{N}$. For any $z,z'\in X$, if $z\in AF_{12k+11}\cap\overline{B}_{1/12(k+1)}(z')$, then $z'\in AF_k$.
\end{lemma}
\begin{proof}
Almost surely, we have $d(T_vz',z')\leq d(T_vz,z)+2d(z',z)$ as before. Also as before, we have $\EE[d(T_vz,z)]\leq \sqrt{\frac{1}{12(k+1)}}$. Squaring and taking expectations yields
\begin{align*}
\EE[d^2(T_vz',z')]&\leq \EE[d^2(T_vz,z)]+4d(z',z)\EE[d(T_vz,z)]+4d^2(z',z)\\
&\leq \frac{1}{12(k+1)}+\frac{4}{12(k+1)}+\frac{4}{12(k+1)}\leq\frac{1}{k+1}
\end{align*}
which gives $z'\in AF_k$.
\end{proof}

\begin{remark}
We can define a sequence of sets 
\[
AF'_k:=\left\{z\in X_0\large\mid \PP\left(d^2(T_vz,z)\geq \frac{1}{k+1}\right)\leq \frac{1}{k+1}\right\}.
\]
which similarly approximate $F$, and which simultaneously stays closer in character to the phrasing of $F$ as an almost-sure equation. Indeed, it can be rather immediately seen that $AF_{k^2+2k}\subseteq AF'_k$ and conversely, assuming that $X_0$ has bounded diameter and letting $c\in\mathbb{N}^*$ be a bound on this diameter, one can establish that $AF'_{\rho(k)}\subseteq AF_k$, with $\rho(k):=\max\{2k+1,4c^2(2k+2)-1\}$. However, the above formulation of $(AF_k)$ using the mean proves to be considerably more productive for the method studied in the following section, so that we have opted for that specific representation.
\end{remark}

\subsection{A Krasnoselskii-Mann scheme for approximating common fixed points} We now consider a simple Krasnoselskii-Mann scheme for approximating elements of $F$. Let $(x_n)$ be a sequence of $X$-valued random variables defined by
\begin{equation}
x_{n+1}:=(1-\lambda_n)x_n\oplus \lambda_n T_{v_n}x_n
\label{sM}\tag{sKM}
\end{equation}
for some $x_0\in X$, where $(\lambda_n)\subseteq (0,1]$ and $(v_n)$ is a sequence of independent $V$-valued random variables distributed as $v$. In particular, $v_n$ is independent of $\mathcal{F}_n:=\sigma(x_0,\ldots,x_n)$ for all $n\in\NN$. As geodesics are continuous in their endpoints (see e.g.\ Lemma 1.2.2 in \cite{Bacak2014a}), it immediately follows by induction that each $x_n$ is measurable. As before, we assume that $F\neq \emptyset$ and now fix some $p\in F$, choosing $r_0\geq 1$ to satisfy $d(x_0,p)\leq r_0$.

\begin{lemma}
\label{lem:MannBounded}
For any $n\in\NN$, we have $d(x_n,p)\leq r_0$ almost surely.
\end{lemma}
\begin{proof}
For $n\in\NN$, since $v_n$ is distributed as $v$ and $p\in F$, we have $d(T_{v_n}p,p)=0$ almost surely. Now, using \ref{T}, it follows rather immediately that
\[
d(x_{n+1},p)\leq (1-\lambda_n)d(x_n,p)+\lambda_n \left(d(T_{v_n}x_n,T_{v_n}p)+d(T_{v_n}p,p)\right)\leq d(x_n,p)
\]
so that $d(x_n,p)\leq d(x_0,p)\leq r_0$ for any $n\in\NN$, almost surely.
\end{proof}

In the following, we will thereby work over the associated space $X_0=\overline{B}_{r_0}(p)$, which is compact as $X$ is proper. In particular, note that we have $d(x_n,z)\leq 2r_0$ for any $z\in X_0$ and $n\in\NN$.

\begin{remark}\label{rem:KMAbsCont}
Note that we by Lemma \ref{lem:MannBounded} in particular have that $\mu(\varepsilon):=\varepsilon/4r_0^2$ is a modulus of absolute continuity for any $d^2(x_n,z)$, since $\EE[d^2(x_n,z)\mathbf{1}_A]\leq 4r_0^2\PP(A)\leq \varepsilon$ for any $A\in\mathcal{F}$ with $\PP(A)<\varepsilon/4r_0^2$. Further, we can construct a modulus of almost sure finiteness for $(d(x_n,p))$ by setting $b_\lambda:=r_0$, as we have $\PP\left(\exists n\in\mathbb{N}\left( d(x_n,p)>r_0\right)\right)=0<\lambda$ for any $\lambda>0$.
\end{remark}

\begin{lemma}\label{lem:CN}
For any $n\in\NN$ and $z\in X$ we have
\begin{align*}
&\EE\left[d^2(x_{n+1},z)\mid\mathcal{F}_n\right]\\
&\qquad\leq (1-\lambda_n)d^2(x_n,z)+\lambda_n\EE\left[d^2(T_{v_n}x_n,z)\mid\mathcal{F}_n\right]-\lambda_n(1-\lambda_n)\EE\left[d^2(x_n,T_{v_n}x_n)\mid\mathcal{F}_n\right].
\end{align*}
\end{lemma}
\begin{proof}
By \eqref{CN} we have
\[
d^2(x_{n+1},z)\leq (1-\lambda_n)d^2(x_n,z)+\lambda_nd^2(T_{v_n}x_n,z)-\lambda_n(1-\lambda_n)d^2(x_n,T_{v_n}x_n)
\]
and the result then follows by taking conditional expectations and using that $d(x_n,z)$ is $\mathcal{F}_n$-measurable. 
\end{proof}

\begin{lemma}\label{lem:Fejer}
The scheme $(x_n)$ satisfies
\[
\forall r,n\in\NN\ \forall z\in AF_{64r_0^2(r+1)}\ \forall m\leq n\left(\EE\left[d^2(x_{m+1},z)\mid\mathcal{F}_m\right]\leq d^2(x_m,z)+\frac{1}{r+1}\right).
\]
and is therefore uniformly stochastic quasi-Fej\'er monotone w.r.t. $F$, $(AF_k)$ and $(\mathcal{F}_n)$, and more specifically for $G(a):=a^2$, $\chi_m:=\eta_m:=0$, and $\zeta(\lambda,r,n):=64r_0^2(r+1)^2$.
\end{lemma}
\begin{proof}
For any $m\in\NN$ and $z\in X_0$, we first observe that we have
\begin{align*}
d^2(T_{v_m}x_m,z)&\leq d^2(T_{v_m}x_m,T_{v_m}z)+d^2(T_{v_m}z,z)+2d(T_{v_m}x_m,T_{v_m}z)d(T_{v_m}z,z)\\
&\leq d^2(x_m,z)+d^2(T_{v_m}z,z)+2d(x_m,z)d(T_{v_m}z,z)\\
&\leq d^2(x_m,z)+d^2(T_{v_m}z,z)+4r_0d(T_{v_m}z,z),
\end{align*}
where the last line follows from Lemma \ref{lem:MannBounded}. Therefore
\begin{align*}
\EE\left[d^2(T_{v_m}x_m,z)\mid\mathcal{F}_m\right]&\leq d^2(x_m,z)+\EE\left[d^2(T_{v_m}z,z)\mid\mathcal{F}_m\right]+4r_0\EE\left[d(T_{v_m}z,z)\mid\mathcal{F}_m\right]\\
&=d^2(x_m,z)+\EE\left[d^2(T_{v}z,z)\right]+4r_0\EE\left[d(T_{v}z,z)\right]
\end{align*}
for the second inequality using that $v_m$ is independent of $\mathcal{F}_m$ and distributed as $v$. By Lemma \ref{lem:CN} we therefore have
\begin{align*}
\EE\left[d^2(x_{m+1},z)\mid\mathcal{F}_m\right]&\leq (1-\lambda_m)d(x_m,z)+\lambda_m\EE\left[d^2(T_{v_m}x_m,z)\mid\mathcal{F}_m\right]\\
&\leq d^2(x_m,z)+\lambda_m\left(\EE\left[d^2(T_{v}z,z)\right]+4r_0\EE\left[d(T_{v}z,z)\right]\right)\\
&\leq d^2(x_m,z)+\EE\left[d^2(T_{v}z,z)\right]+4r_0\EE\left[d(T_{v}z,z)\right]
\end{align*}
for any $m\in\NN$ and $z\in X_0$, using $\lambda_m\leq 1$ in the last line. Now fix $r,n\in\NN$, $z\in AF_{64r_0^2(r+1)^2}$ and $m\leq n$. Using Jensen's inequality, it follows that $\EE\left[d^2(T_{v}z,z)\right]\leq \frac{1}{2(r+1)}$ and $\EE\left[d(T_{v}z,z)\right]\leq \frac{1}{8r_0(r+1)}$ from which the result follows.
\end{proof}

\begin{remark}
Our scheme is stochastically quasi-Fej\'er monotone in an extremely simple way, where in particular the error sequences $(\chi_m),(\eta_m)$ disappear and the modulus $\zeta(\lambda,r,n)$ depends only on $r$, and indeed we deliberately chose an example through which our main results could be illustrated in a simplified setting. However, even in this simple setting, there are small extensions that would lead to more elaborate versions of the corresponding quasi-Fej\'er property. For example, one could study an extension of our scheme with noise terms $e_n$:
\[
x_{n+1}:=(1-\lambda_n)x_n\oplus \lambda_ny_n \ \ \ \mbox{where} \ \ \ d(y_n,T_{v_n}x_n)\leq e_n\AS
\]
A simple instance of this scheme in a Hilbert space and for non-random mappings, given via
\[
x_{n+1}:=x_n+ \lambda_n(Tx_n+e_n-x_n)
\]
was considered in \cite{CombettesPesquet2015} in the presence of compactness, giving rise to a stochastic quasi-Fej\'er property in which the noise terms are now incorporated into the sequence $(\eta_m)$. Another form of generality would be obtained by weaking the i.i.d.\ assumption of the $(v_n)$ to
\[
\forall n\in\mathbb{N}\ \forall z\in X\left(z\in F\to \EE\left[d^2(T_{v_n}z,z)\mid\mathcal{F}_n\right]=0\right),
\]
and then defining the $(AF_k)$ implicitly through a quantitative rendering of that property, namely an abstract modulus $\rho$ satisfying
\[
\forall n,k\in\mathbb{N}\ \forall z\in X\left(z\in AF_{\rho(n,k)}\to \EE\left[d^2(T_{v_n}z,z)\mid\mathcal{F}_n\right]\leq\frac{1}{k+1}\right).
\]
A simple adaptation of Lemma \ref{lem:Fejer} would then establish uniform stochastic quasi-Fej\'er monotonicity of $(x_n)$ where now $\zeta(\lambda,r,n):=\rho(n,64r_0^2(r+1)^2)$ is also dependent on $n$. However, we leave a proper study of various generalisations of the scheme discussed here, along with further concrete examples of our framework, to future work.
\end{remark}

Having shown that the scheme $(x_n)$ is stochastically quasi-Fej\'er monotone, we now show how to define an explicit $\liminf$-modulus $\Phi$, under the additional (standard) assumption that the step-sizes $(\lambda_n)$ satisfy $\sum_{n=0}^\infty \lambda_n(1-\lambda_n)=\infty$. Quantitatively, this assumption is resolved using a rate of divergence $\theta:\NN\times (0,\infty)\to\NN$, i.e.
\[
\forall N\in\NN\ \forall b>0\left(\sum_{n=N}^{\theta(N,b)}\lambda_n(1-\lambda_n)\geq b\right).
\]
We begin by identifying a simpler $\liminf$ property, widely seen in the convergence analysis of Krasnoselskii-Mann-type schemes:
\begin{lemma}
\label{lem:MannLiminfSimple}
Define $\Psi:(0,\infty)\times \NN\to \NN$ by $\Psi(\mu,N):=\theta\left(N,\frac{r_0^2+1}{\mu}\right)$. Then we have
\[
\forall \mu>0\ \forall N\in\mathbb{N}\ \exists n\in [N;\Psi(\mu,N)]\left(\EE\left[\int_\Omega d^2(T_{v(\omega)}x_n,x_n)\,\mathrm{d}\PP(\omega)\right]<\mu\right).
\]
\end{lemma}
\begin{proof}
For any $n\in\NN$, by \eqref{CN} we have
\begin{align*}
d^2(x_{n+1},p)&\leq (1-\lambda_n)d^2(x_n,p)+\lambda_nd^2(T_{v_n}x_n,p)-\lambda_n(1-\lambda_n)d^2(T_{v_n}x_n,x_n)\\
&\leq d^2(x_n,p)-\lambda_n(1-\lambda_n)d^2(T_{v_n}x_n,x_n)
\end{align*}
where here we use that $d(T_{v_n}p,p)=0$ and hence $d(T_{v_n}x_n,p)\leq d(x_n,p)$ almost surely. But then for any $N\in\NN$ we have
\[
\sum_{n=N}^k \lambda_n(1-\lambda_n)d^2(T_{v_n}x_n,x_n)\leq d^2(x_N,p)-d^2(x_{k+1},p)\leq d^2(x_N,p)\leq r_0^2
\]
and therefore, if $\EE\left[d^2(T_{v_n}x_n,x_n)\right]\geq \mu$ for all $n\in [N;\Psi(\mu,N)]$, then
\[
r_0^2+1\leq \mu\sum_{n=N}^{\theta(N,(r_0^2+1)/\mu)}\lambda_n(1-\lambda_n)\leq \sum_{n=N}^{\theta(N,(r_0^2+1)/\mu)} \lambda_n(1-\lambda_n)\EE\left[d^2(T_{v_n}x_n,x_n)\right] \leq r_0^2,
\]
a contradiction. As we have $\EE\left[d^2(T_{v_n}x_n,x_n)\right]=\EE\left[\int_\Omega d^2(T_{v(\omega)}x_n,x_n)\,\mathrm{d}\PP(\omega)\right]$ from Fubini's theorem, and since $v_n$ is independent of $x_n$ and distributed as $v$, the result follows.
\end{proof}

\begin{lemma}
\label{lem:MannLiminf}
Let $\Psi$ be defined as in Lemma \ref{lem:MannLiminfSimple}, and set
\[
\Phi(\lambda,k,N):=\Psi\left(\frac{\lambda}{k+1},N\right)=\theta\left(N,\frac{(r_0^2+1)(k+1)}{\lambda}\right).
\]
Then $\Phi$ is a stochastic $\liminf$-modulus w.r.t. $F$ and $(AF_k)$ for $(x_n)$ i.e.
\[
\forall \lambda>0\ \forall k,N\in\NN\left(\PP\left(\forall n\in [N;\Phi(\lambda,k,N)](x_n\notin AF_k)\right)<\lambda\right).
\]
\end{lemma}
\begin{proof}
Fix $\lambda>0$ and $k,N\in\NN$, and define $\mu:=\lambda/(k+1)$. By Lemma \ref{lem:MannLiminfSimple}, there exists $n\in [N;\Phi(\lambda,k,N)]$ such that $\EE\left[\int_\Omega d^2(T_{v(\omega)}x_n,x_n)\,\mathrm{d}\PP(\omega)\right]< \mu$. We then have
\[
\PP\left(x_n\notin AF_k\right)=\PP\left( \int_\Omega d^2(T_{v(\omega)}x_n,x_n)\,\mathrm{d}\PP(\omega)>\frac{1}{k+1}\right)\leq (k+1)\EE\left[\int_\Omega d^2(T_{v(\omega)}x_n,x_n)\,\mathrm{d}\PP(\omega)\right]
\]
by Markov's inequality, and so $\PP\left(x_n\notin AF_k\right)<\lambda$.
\end{proof}

We now have everything we need to instantiate the main theorem of the paper, by which we can derive the following (quantitative) convergence theorem for our stochastic Krasnoselskii-Mann-type scheme.

\begin{theorem}\label{thm:application}
Let $(X,d)$ be a proper Hadamard space, $(V,\mathcal{V})$ a measure space, and $(T_v)_{v\in V}$ a family of nonexpansive mappings on $X$ such that $(v,x)\mapsto T_vx$ is measurable. Let $v$ be a $V$-valued random variable and consider the corresponding problem \eqref{sCFP} with solution set $F$. Assume that $F\neq \emptyset$ and fix some $p\in F$ as well as a $r_0\geq 1$ with $d(x_0,p)\leq r_0$. Set $X_0:=\overline{B}_{r_0}(p)$. Let $(x_n)$ be defined by \eqref{sM}, using some $x_0\in X$, $(\lambda_n)\subseteq (0,1]$ with $\sum_{n=0}^\infty \lambda_n(1-\lambda_n)=\infty$, and a sequence $(v_n)$ of independent $V$-valued random variables distributed as $v$.

Then $x_n\to x$ almost surely, where $x\in F$ almost surely. Furthermore, if $\gamma$ is a modulus of total boundedness for $X_0$, $\theta(N,b)$ is a rate of divergence for $\sum_{n=0}^\infty \lambda_n(1-\lambda_n)$, then a rate of metastability for $(x_n)$ is explicitly definable in those data. 
\end{theorem}

\begin{proof}
The first part follows from Theorem \ref{thm:final} for $G(a):=a^2$ and $(AF_k)$ as defined in Section \ref{sec:common:spaces}, where in particular $(x_n)$ is stochastically quasi-Fej\'er monotone by Lemma \ref{lem:Fejer} (by which it also satisfies the stronger property of being uniformly stochastic quasi-Fej\'er monotone), has the stochastic $\liminf$ property by Lemma \ref{lem:MannLiminf}, and is explicitly closed by Lemma \ref{lem:FExpClosed}. For the second part, we appeal to Theorem \ref{thm:mainQuant}, which in addition to the modulus of uniform boundedness $\gamma$, has its assumptions satisfied as follows:
\begin{itemize}
	\item the modulus of continuity and inverse continuity for $G(a):=a^2$ are given by $\iota_b(k):=2b(k+1)$ and $\nu(k):=(k+1)^2$, respectively;
	\item we set $o:=p$ and then $r_0$ is by definition a bound relative to $p$ for $X_0$;
	\item $\eta_n:=\chi_n:=0$ so we can set $f(\lambda):=1$ and $h(\lambda):=2$;
	\item $\zeta(\lambda,r,n):=64r_0^2(r+1)^2$ as defined in Lemma \ref{lem:Fejer};
	\item $\Phi(\lambda,k,N):=\theta(N,(2r_0+1)(k+1)/\lambda)$ as defined in Lemma \ref{lem:MannLiminf};
	\item since $d(x_0,p)+r_0\leq 2r_0$ we have $\EE[(d(x_0,p)+r_0)^2]<b_0$ for $b_0:=4r_0^2+1$.
\end{itemize}
Recall also Remark \ref{rem:KMAbsCont} for definitions of moduli of absolute continuity and of almost sure finiteness for the respective sequence in terms of $r_0$. Therefore, the rate of metastability as given in Theorem \ref{thm:mainQuant} depends only on $\gamma$, $r_0$ and $\theta(N,b)$.
\end{proof}

\noindent\textbf{Acknowledgments.} The authors want to thank Ulrich Kohlenbach and Pedro Pinto for helpful discussions and remarks on the topic of this paper. We are grateful to the anonymous referee for their many helpful comments, which improved the paper.\\

\noindent\textbf{AI declaration.} Generative AI was not used to produce this article.\\

\noindent\textbf{Funding.} The first author was partially supported by the EPSRC Centre for Doctoral Training in Digital Entertainment EP/L016540/1, and the third author was partially supported by the EPSRC grant EP/W035847/1. The second author was partially supported by the Deutsche Forschungsgemeinschaft Project DFG KO 1737/6-2.

\bibliographystyle{plain}
\bibliography{ref}

\begin{thebibliography}{10}

\bibitem{AlexanderKapovitchPetrunin2023}
S.~Alexander, V.~Kapovitch, and A.~Petrunin.
\newblock {\em {Alexandrov Geometry: Foundations}}, volume 236 of {\em Graduate
  Studies in Mathematics}.
\newblock American Mathematical Society, Providence, RI, 2024.

\bibitem{ArakcheevBauschke2025b}
A.~Arakcheev and H.H. Bauschke.
\newblock {Fej\'er and Fej\'er$^*$ Monotonicity: Differences, Similarities, and
  the Role of the Relative Interior}.
\newblock {\em Journal of Optimization Theory and Applications}, 210(3), 2026.
\newblock Article no.\ 46, 24pp.

\bibitem{ArakcheevBauschke2025a}
A.~Arakcheev and H.H. Bauschke.
\newblock {On Opial's Lemma}.
\newblock {\em Mathematical Programming}, 2026.
\newblock To appear, preprint available at
  \url{https://arxiv.org/abs/2503.22004}.

\bibitem{AvigadGerhardyTowsner2010}
J.~Avigad, P.~Gerhardy, and H.~Towsner.
\newblock {Local stability of ergodic averages}.
\newblock {\em Transactions of the American Mathematical Society},
  362(1):261--288, 2010.

\bibitem{BauschkeBorweinBorwein2003}
H.H. Bauschke, J.M. Borwein, and P.L. Borwein.
\newblock {Bregman monotone optimization algorithms}.
\newblock {\em SIAM Journal on Control and Optimization}, 42(2):596--636, 2003.

\bibitem{BauschkeCombettes2017}
H.H. Bauschke and P.L. Combettes.
\newblock {\em {Convex Analysis and Monotone Operator Theory in Hilbert
  Spaces}}.
\newblock CMS Books in Mathematics. Springer Cham, 2nd edition, 2017.

\bibitem{Bacak2014a}
M.~Ba\v{c}\'ak.
\newblock {\em {Convex analysis and optimization in Hadamard spaces}},
  volume~22 of {\em De Gruyter Series in Nonlinear Analysis and Applications}.
\newblock Walter de Gruyter GmbH, Berlin/Boston, 2014.

\bibitem{BacakSearstonSims2012}
M.~Ba\v{c}\'ak, I.~Searston, and B.~Sims.
\newblock {Alternating projections in CAT(0) spaces}.
\newblock {\em Journal of Mathematical Analysis and Applications},
  385:599--607, 2012.

\bibitem{BehlingBelloCruzIusemLiuSantos2024}
R.~Behling, Y.~Bello-Cruz, A.~Iusem, D.~Liu, and L.R. Santos.
\newblock {A finitely convergent circumcenter method for the convex feasibility
  problem}.
\newblock {\em SIAM Journal on Optimization}, 34:2535--2556, 2024.

\bibitem{BehlingBelloCruzIusemAlvesRibeiroSantos2024}
R.~Behling, Y.~Bello-Cruz, A.N. Iusem, A.~Alves Ribeiro, and L.-R. Santos.
\newblock {Fej\'er$^*$ monotonicity in optimization algorithms}.
\newblock {\em Optimization}, 2026.
\newblock To appear, preprint available at
  \url{https://arxiv.org/abs/2410.08331}.

\bibitem{BergNikolaev2008}
I.D. Berg and I.G. Nikolaev.
\newblock {Quasilinearization and curvature of Aleksandrov spaces}.
\newblock {\em Geometriae Dedicata}, 133:195--218, 2008.

\bibitem{BridsonHaefliger1999}
M.R. Bridson and A.~Haefliger.
\newblock {\em {Metric Spaces of Non-Positive Curvature}}, volume 319 of {\em
  Grundlehren der mathematischen Wissenschaften}.
\newblock Springer Berlin, Heidelberg, 1999.

\bibitem{Chashka1994}
A.A. Chashka.
\newblock {Fluctuations in martingales}.
\newblock {\em Uspekhi Matematicheskikh Nauk}, 49(2):179--180, 1994.
\newblock English translation in Russian Math. Surveys 49:2 (1994).

\bibitem{Combettes2001}
P.L. Combettes.
\newblock {Quasi-Fej\'erian Analysis of Some Optimization Algorithms}.
\newblock In D.~Butnariu, Y.~Censor, and S.~Reich, editors, {\em {Inherently
  Parallel Algorithms in Feasibility and Optimization and Their Applications}},
  volume~8 of {\em Studies in Computational Mathematics}, pages 115--152.
  North-Holland, Amsterdam, 2001.

\bibitem{Combettes2009}
P.L. Combettes.
\newblock {Fej\'er monotonicity in convex optimization}.
\newblock In C.A. Floudas and P.M. Pardalos, editors, {\em {Encyclopedia of
  Optimization}}, pages 1016--1024. Springer, Boston, MA, 2009.

\bibitem{CombettesMadariaga2025}
P.L. Combettes and J.I. Madariaga.
\newblock {A Geometric Framework for Stochastic Iterations}.
\newblock {\em Mathematics of Computation}, 2026.
\newblock To appear, preprint available at
  \url{https://arxiv.org/abs/2504.02761}.

\bibitem{CombettesPesquet2015}
P.L. Combettes and J.-C. Pesquet.
\newblock {Stochastic quasi-Fej\'er block-coordinate fixed point iterations
  with random sweeping}.
\newblock {\em SIAM Journal on Optimization}, 25(2):1221--1248, 2015.

\bibitem{CombettesPesquet2019}
P.L. Combettes and J.-C. Pesquet.
\newblock {Stochastic quasi-Fej\'er block-coordinate fixed point iterations
  with random sweeping II: mean-square and linear convergence}.
\newblock {\em Mathematical Programming}, 174(1):433--451, 2019.

\bibitem{CombettesVu2013}
P.L. Combettes and B.C.~V\ u.
\newblock {Variable metric quasi-Fejér monotonicity}.
\newblock {\em Nonlinear Analysis: Theory, Methods \& Applications}, 78:17--31,
  2013.

\bibitem{Doob1953}
J.L. Doob.
\newblock {\em {Stochastic Processes}}.
\newblock Wiley, N.Y., 1953.

\bibitem{Doob1961}
J.L. Doob.
\newblock {Notes on Martingale Theory}.
\newblock In {\em Berkeley Symposium on Mathematical Statistics and
  Probability}, pages 95--102. University of California Press, 1961.

\bibitem{Eremin1968a}
I.I. Eremin.
\newblock {Methods of Fej\'er approximations in convex programming}.
\newblock {\em Mathematical notes of the Academy of Sciences of the USSR},
  3:139--149, 1968.

\bibitem{Eremin1968b}
I.I. Eremin.
\newblock {On the speed of convergence in the method of Fej\'er
  approximations}.
\newblock {\em Mathematical notes of the Academy of Sciences of the USSR},
  4:522--527, 1968.

\bibitem{Ermolev1969}
Y.M. Ermol'ev.
\newblock {On the method of generalized stochastic gradients and quasi-Fej\'er
  sequences}.
\newblock {\em Cybernetics}, 5:208--220, 1969.

\bibitem{Ermolev1971}
Y.M. Ermol'ev.
\newblock {On convergence of random quasi-Fej\'er sequences}.
\newblock {\em Cybernetics}, 7:655--656, 1971.

\bibitem{ErmolevTuniev1973}
Y.M. Ermol'ev and A.D. Tuniev.
\newblock {Random Fej\'er and quasi-Fej\'er sequences}.
\newblock {\em Theory of Optimal Solutions – Akademiya Nauk Ukrainsko\u{i},
  SSR Kiev}, 2:76--83, 1968.
\newblock English translation in Amer. Math. Soc. Select. Translat. Math.
  Statist. Probab., 13 (1973), pp. 143–148.

\bibitem{Fejer1922}
L.~Fej\'er.
\newblock {\"Uber die Lage der Nullstellen von Polynomen, die aus
  Minimumforderungen gewisser Art entspringen}.
\newblock {\em Mathematische Annalen}, 85:41--48, 1922.

\bibitem{Gerhardy2008}
P.~Gerhardy.
\newblock {Proof mining in topological dynamics}.
\newblock {\em Notre Dame Journal of Formal Logic}, 49:431--446, 2008.

\bibitem{JonesKaufmanRosenblattWierdl1998}
R.L. Jones, R.~Kaufman, J.M. Rosenblatt, and M.~Wierdl.
\newblock Oscillation in ergodic theory.
\newblock {\em Ergodic Theory and Dynamical Systems}, 18(4):889--935, 1998.

\bibitem{Kachurovskii1996}
A.G. Kachurovskii.
\newblock {The rate of convergence in ergodic theorems}.
\newblock {\em Russian Mathematical Surveys}, 51(4):653--703, 1996.

\bibitem{Klenke2020}
A.~Klenke.
\newblock {\em {Probability Theory}}.
\newblock Springer, 3rd edition, 2020.

\bibitem{Kohlenbach2005}
U.~Kohlenbach.
\newblock {Some computational aspects of metric fixed point theory}.
\newblock {\em Nonlinear Analysis: Theory, Methods \& Applications},
  61(5):823--837, 2005.

\bibitem{Kohlenbach2008}
U.~Kohlenbach.
\newblock {\em {Applied Proof Theory: Proof Interpretations and their Use in
  Mathematics}}.
\newblock Springer Monographs in Mathematics. Springer-Verlag Berlin
  Heidelberg, 2008.

\bibitem{Kohlenbach2019}
U.~Kohlenbach.
\newblock {Proof-theoretic Methods in Nonlinear Analysis}.
\newblock In B.~Sirakov, P.~Ney de~Souza, and M.~Viana, editors, {\em
  {Proceedings of ICM 2018}}, volume~2, pages 61--82. World Scientific,
  Singapure, 2019.

\bibitem{KohlenbachLambov2004}
U.~Kohlenbach and B.~Lambov.
\newblock {Bounds on iterations of asymptotically quasi- nonexpansive
  mappings}.
\newblock In J.~Garcia-Falset, E.~Llorens-Fuster, and B.~Sims, editors, {\em
  Proceedings of the International Conference on Fixed Point Theory and
  Applications (Valencia 2003)}, pages 143--172. Yokohama Press, 2004.

\bibitem{KohlenbachLeusteanNicolae2018}
U.~Kohlenbach, L.~Leu\c{s}tean, and A.~Nicolae.
\newblock {Quantitative results on Fej\'er monotone sequences}.
\newblock {\em Communications in Contemporary Mathematics}, 20(2), 2018.
\newblock 1750015, 42pp.

\bibitem{KohlenbachLopezAcedoNicolae2019}
U.~Kohlenbach, G.~Lopez-Acedo, and A.~Nicolae.
\newblock {Moduli of regularity and rates of convergence for Fej\'er monotone
  sequences}.
\newblock {\em Israel Journal of Mathematics}, 232:261--297, 2019.

\bibitem{KohlenbachOliva2003}
U.~Kohlenbach and P.~Oliva.
\newblock {Proof Mining: A Systematic Way of Analysing Proofs in Mathematics}.
\newblock {\em Proceedings of the Steklov Institute of Mathematics},
  242:136--164, 2003.

\bibitem{KohlenbachPinto2023}
U.~Kohlenbach and P.~Pinto.
\newblock {Fej\'er Monotone Sequences Revisited}.
\newblock {\em Journal of Convex Analysis}, 33(1/2):75--90, 2026.

\bibitem{KohlenbachSafarik2014}
U.~Kohlenbach and P.~Safarik.
\newblock Fluctuations, effective learnability and metastability in analysis.
\newblock {\em Annals of Pure and Applied Logic}, 165:266--304, 2014.

\bibitem{Metivier1982}
M.~M\'etivier.
\newblock {\em {Semimartingales}}, volume~2 of {\em De Gruyter Studies in
  Mathematics}.
\newblock De Gruyter, Berlin, 1982.

\bibitem{MotzkinSchoenberg1954}
T.~Motzkin and I.~Schoenberg.
\newblock {The Relaxation Method for Linear Inequalities}.
\newblock {\em Canadian Journal of Mathematics}, 6:393--404, 1954.

\bibitem{Neri2025}
M.~Neri.
\newblock {A finitary Kronecker's lemma and large deviations in the Strong Law
  of Large Numbers on Banach spaces}.
\newblock {\em Annals of Pure and Applied Logic}, 176(6), 2025.
\newblock Article no.\ 103569, 31pp.

\bibitem{NeriOlivaPischke2026}
M.~Neri, P.~Oliva, and N.~Pischke.
\newblock {A systematic way of analysing proofs in probability theory}, 2026.
\newblock Preprint, available at \url{https://arxiv.org/abs/2604.08078}.

\bibitem{NeriPischke2024}
M.~Neri and N.~Pischke.
\newblock {Proof mining and probability theory}.
\newblock {\em Forum of Mathematics, Sigma}, 13, 2025.
\newblock Article no.\ e187, 47pp.

\bibitem{NeriPischkePowell2025b}
M.~Neri, N.~Pischke, and T.~Powell.
\newblock {Generalized learnability of stochastic principles}.
\newblock In A.~Beckmann, I.~Oitavem, and F.~Manea, editors, {\em Proceedings
  of the 21st Conference on Computability in Europe (CiE'25)}, volume 15764 of
  {\em LNCS}, pages 333--348. Springer, Cham, 2025.

\bibitem{NeriPischkePowell2025}
M.~Neri, N.~Pischke, and T.~Powell.
\newblock {An abstract effective convergence theorem for stochastic processes,
  with applications to stochastic approximation}, 2026.
\newblock Preprint, available at \url{https://arxiv.org/abs/2504.12922}.

\bibitem{NeriPowell2025a}
M.~Neri and T.~Powell.
\newblock {On quantitative convergence for stochastic processes: Crossings,
  fluctuations and martingales}.
\newblock {\em Transactions of the American Mathematical Society, Series B},
  12:974--–1019, 2025.

\bibitem{NeriPowell2026}
M.~Neri and T.~Powell.
\newblock {A quantitative {R}obbins-{S}iegmund theorem}.
\newblock {\em Annals of Applied Probability}, 36(1):636--651, 2026.

\bibitem{Neumann2015}
E.~Neumann.
\newblock {Computational problems in metric fixed point theory and their
  Weihrauch degrees}.
\newblock {\em Logical Methods in Computer Science}, 11(4), 2015.
\newblock Article no.\ 20, 44pp.

\bibitem{Pischke2025}
N.~Pischke.
\newblock {Generalized Fej\'er monotone sequences and their finitary content}.
\newblock {\em Optimization}, 74(14):3771--3838, 2025.

\bibitem{Pischke2026}
N.~Pischke.
\newblock {On Busemann subgradient methods for stochastic minimization in
  Hadamard spaces}, 2026.
\newblock Preprint, available at \url{https://arxiv.org/abs/2602.08127}.

\bibitem{PischkePowell2024}
N.~Pischke and T.~Powell.
\newblock {Asymptotic regularity of a generalised stochastic Halpern scheme}.
\newblock {\em Journal of Optimization Theory and Applications}, 210(1), 2026.
\newblock Article no.\ 3, 43pp.

\bibitem{PischkePowell2026}
N.~Pischke and T.~Powell.
\newblock {Convergence guarantees for stochastic algorithms solving non-unique
  problems in metric spaces}, 2026.
\newblock Preprint, available at \url{https://arxiv.org/abs/2605.06129}.

\bibitem{RobbinsSiegmund1971}
H.~Robbins and D.~Siegmund.
\newblock {A convergence theorem for non-negative almost supermartingales and
  some applications}.
\newblock In {\em Optimizing methods in statistics}, pages 233--257. Elsevier,
  1971.

\bibitem{Specker1949}
E.~Specker.
\newblock {Nicht konstruktiv beweisbare S\"atze der Analysis}.
\newblock {\em The Journal of Symbolic Logic}, 14:145--208, 1949.

\bibitem{Tao2008b}
T.~Tao.
\newblock {Norm convergence of multiple ergodic averages for commuting
  transformations}.
\newblock {\em Ergodic Theory and Dynamical Systems}, 28(2):657--688, 2008.

\bibitem{Tao2008a}
T.~Tao.
\newblock {\em {Soft Analysis, Hard Analysis, and the Finite Convergence
  Principle}}, pages 17--29.
\newblock American Mathematical Society, Providence, RI, 2008.

\bibitem{Ville1939}
J.~Ville.
\newblock {\em {\'Etude Critique de la Notion de Collectif}}.
\newblock PhD thesis, \'Ecole Polytechnique, 1939.

\bibitem{Williams1991}
D.~Williams.
\newblock {\em {Probability with Martingales}}.
\newblock Cambridge University Press, Cambridge, 1991.

\end{thebibliography}

\end{document}